\documentclass{article}
\usepackage{makeidx}

\usepackage{amssymb,amsfonts,stmaryrd,mathrsfs}
\usepackage{amsmath,latexsym,amsthm}
\usepackage[disable]{todonotes}

\usepackage{float}

\usepackage{manfnt}
\usepackage{verbatim}
\usepackage[applemac]{inputenc} % MacOS
\usepackage[cyr]{aeguill}
\usepackage[russian,francais,english]{babel}

\usepackage{frcursive}

\usepackage{url}
\let\urlorig\url
\renewcommand{\url}[1]{%
  \begin{otherlanguage}{english}\urlorig{#1}\end{otherlanguage}%
}
\usepackage{graphicx}

\makeindex

\newtheorem{lemma}{Lemma}
\newtheorem{corollary}{Corollary}
\newtheorem{proposition}{Proposition}
\newtheorem{theorem}{Theorem}
\newtheorem{remark}{Remark}

\newtheorem*{lemma*}{Lemma}

\newcommand{\class}{{\mathcal C}}

\newcommand{\sJ}{{\mathscr J}}
\newcommand{\bJ}{{\bf B}}
\newcommand{\tbJ}{\widetilde{\bf B}}

\newcommand{\dif}{{\rm d}}
\newcommand{\Dif}{{\rm D}}

\newcommand{\vits}{{u}}
\newcommand{\vitsL}{{w}}

\newcommand{\vol}{{v}}

\newcommand{\bu}{{u}}
\newcommand{\ubu}{\underline{u}}
\newcommand{\bv}{{v}}
\newcommand{\ubv}{\underline{v}}

\newcommand{\bbu}{{\bf u}}
\newcommand{\bbv}{{\bf v}}
\newcommand{\bw}{{\bf w}}

\newcommand{\bP}{{\bf P}}

\newcommand{\blambda}{\boldsymbol{\lambda}}

\newcommand{\umu}{{\mu}_{\!\!\!\!_{{\bf -}}}}

\newcommand{\bB}{{\bf B}}

\newcommand{\bU}{{\bf U}}
\newcommand{\ubU}{\underline{\bf U}}
\newcommand{\uubU}{\underline{\underline{\bf U}}}

\newcommand{\ublambda}{\underline{\blambda}}
\newcommand{\uX}{\underline{X}}
\newcommand{\uuXi}{\underline{\Xi}}
\newcommand{\uXi}{{\Xi}}

\newcommand{\uq}{q_{\!\!\!\!_{{\bf -}}}}

\newcommand{\bQ}{{\bf Q}}
\newcommand{\bV}{{\bf V}}
\newcommand{\bM}{{\bf M}}

\newcommand{\Real}{{\rm Re}}
\newcommand{\Imag}{{\rm Im}}

\newcommand{\uL}{\underline{L}}
\newcommand{\uY}{\underline{Y}}
\newcommand{\uvits}{\underline{u}}
\newcommand{\uvitsL}{\underline{w}}
\newcommand{\uvol}{\underline{v}}
\newcommand{\uvoly}{v_{\!\!\!\!_{{\bf -}}y}}
\newcommand{\uvolyy}{v_{\!\!\!\!_{{\bf -}}yy}}

\newcommand{\urho}{\rho_{\!\!\!\!_{{\bf -}}}}
\newcommand{\urhox}{\rho_{\!\!\!\!_{{\bf -}}x}}
\newcommand{\urhoxx}{\rho_{\!\!\!\!_{{\bf -}}xx}}

\newcommand{\UU}{{\mathbb U}}
\newcommand{\uUU}{\underline{\mathbb U}}
\newcommand{\HH}{{\mathbb H}}
\newcommand{\R}{{\mathbb R}}
\newcommand{\C}{{\mathbb C}}
\newcommand{\Z}{{\mathbb Z}}

\newcommand{\W}{{\mathbb W}}

\newcommand{\ee}{{\rm e}}

\newcommand{\Euler}{{\sf E}}
\newcommand{\Legendre}{{\sf L}}

\newcommand{\Hess}{{\sf Hess}}

\newcommand{\Lag}{\mathscr{L}}
\newcommand{\lag}{{\ell}}

\newcommand{\press}{{p}}
\newcommand{\chem}{{g}}

\renewcommand{\Cap}{\mathscr{K}}
\renewcommand{\cap}{{\kappa}}
\newcommand{\ucap}{\underline{\kappa}}
\newcommand{\uCap}{\underline{K}}
\newcommand{\ualpha}{\underline{\alpha}}
\newcommand{\ubeta}{\beta_{\!\!\!\!_{{\bf -}}}}
\newcommand{\ugamma}{{\gamma}_{\!\!\!\!_{{\bf -}}}}
\newcommand{\ugammax}{{\gamma}_{\!\!\!\!_{{\bf -}x}}}
\newcommand{\uM}{{\bf M}}

\newcommand{\action}{\theta}
\newcommand{\Action}{\Theta}
\newcommand{\mineur}{m}
\newcommand{\Mineur}{M}
\newcommand{\En}{\mathscr{E}}
\newcommand{\enred}{\mathscr{U}}
\newcommand{\en}{\mbox{{\begin{cursive}{\emph{e}}\end{cursive}}}}
\newcommand{\Ec}{\mathscr{I}}
\newcommand{\Ham}{\mathscr{H}}
\newcommand{\Potential}{\mathscr{W}}

\newcommand{\tHam}{\widetilde{\mathscr{H}}}

\newcommand{\Impulse}{\mathscr{Q}}
\newcommand{\Impulseflux}{\mathscr{S}}

\newcommand{\imp}{\mbox{{\begin{cursive}{\emph{q}}\end{cursive}}}}

\newcommand{\Linvar}{{\bf A}}
\newcommand{\linvar}{{\bf a}}

\newcommand{\tLin}{{\widetilde{\mathscr{A}}}}
\newcommand{\Lin}{{\mathscr{A}}}
\newcommand{\Linb}{{\mathscr{L}}}

\newcommand{\Evans}{{D}}
\newcommand{\tEvans}{\widetilde{D}}
\newcommand{\evans}{{d}}
\newcommand{\tevans}{\widetilde{d}}

\newcommand{\cO}{\mathcal{O}}

\def\transp#1{{#1}^{{\sf T}}}
\newcommand{\negsign}{{\sf n}}
\newcommand{\ind}{{\sf I}}
\newcommand{\ev}{{z}}
\newcommand{\evr}{{r}}
\newcommand{\disc}{{\rm T}}
\newcommand{\Monod}{{\bf F}}
\newcommand{\monod}{{\bf f}}
\newcommand{\speed}{{c}}
\newcommand{\uspeed}{\underline{c}}
\newcommand{\bC}{{\bf C}}
\newcommand{\bc}{{\bf c}}
\newcommand{\man}{{\mathscr C}}
\newcommand{\funct}{\mathscr{F}}

\title{Co-periodic stability of periodic waves in some Hamiltonian PDEs}
\author{S. Benzoni-Gavage\thanks{Universit\'e de Lyon,
CNRS UMR 5208,
Universit\'e Lyon 1,
Institut Camille Jordan,
43 bd 11 novembre 1918;
F-69622 Villeurbanne cedex}, 
C. Mietka\thanks{Universit\'e de Lyon,
CNRS UMR 5208,
Universit\'e Lyon 1,
Institut Camille Jordan,
43 bd 11 novembre 1918;
F-69622 Villeurbanne cedex}, 
and L.M. Rodrigues\thanks{Universit\'e de Lyon,
CNRS UMR 5208,
Universit\'e Lyon 1,
INRIA \'EP Kaliffe,
Institut Camille Jordan,
43 bd 11 novembre 1918;
F-69622 Villeurbanne cedex}}

\begin{document}

\maketitle

\begin{abstract}
The stability theory of periodic traveling waves is much less advanced than for solitary waves, which were first studied by Boussinesq and have received a lot of attention in the last decades. In particular, despite recent breakthroughs regarding periodic waves in reaction-diffusion equations and viscous systems of conservation laws [Johnson--Noble--Rodrigues--Zumbrun, Invent math (2014)], the stability of periodic traveling wave solutions to dispersive PDEs with respect to `arbitrary' perturbations is still widely open in the absence of a dissipation mechanism.
The focus is put here on \emph{co-periodic} stability of periodic waves, that is, stability with respect to perturbations of the same period as the wave, for KdV-like systems of one-dimensional Hamiltonian PDEs. Fairly general nonlinearities are allowed in these systems, so as to include various models of mathematical physics, and this precludes  
complete integrability techniques. Stability criteria are derived and investigated first in a general abstract framework, and then applied to three basic examples that are very closely related, and ubiquitous in mathematical physics, namely, a quasilinear version of the generalized Korteweg--de Vries equation (qKdV), and the Euler--Korteweg system in both Eulerian coordinates (EKE) and in mass Lagrangian coordinates (EKL). Those criteria consist of a necessary condition for \emph{spectral stability}, and of a sufficient condition for \emph{orbital stability}. Both are expressed in terms of a single function, 
the abbreviated action integral along the orbits of waves in the phase plane, which is the counterpart of the solitary waves moment of instability  
introduced by Boussinesq. However, the resulting criteria are more complicated for periodic waves because they have more degrees of freedom than solitary waves, so that 
the action is a function of $N+2$ variables for a system of $N$ PDEs, while 
the moment of instability is a function of the wave speed 
only once the endstate of the solitary wave is fixed. Regarding solitary waves, the celebrated Grillakis--Shatah--Strauss stability criteria amount to looking for the sign of 
the second derivative of the moment of instability with respect to the wave speed. For periodic waves, stability criteria involve all the second order, partial derivatives of the action. This had already been pointed out by various authors for some specific equations, in particular the generalized Korteweg--de Vries equation --- which is special case of (qKdV) --- but not from a general point of view, up to the authors' knowledge. The most striking results obtained here can be summarized as: an \emph{odd} value for the difference between $N$ and the negative signature of  the Hessian of the action implies spectral instability, whereas a negative signature of the same Hessian being equal to $N$ implies orbital stability. Furthermore, it is shown that, when applied to the Euler--Korteweg system, this approach yields several interesting connexions between (EKE), (EKL), and (qKdV).
More precisely, (EKE) and (EKL) share the same abbreviated action integral, which is related to that of (qKdV) in a simple way. This basically proves simultaneous stability in both formulations (EKE) and (EKL) --- as one may reasonably expect from the physical point view ---, which is interesting to know when these models are used for different phenomena --- \emph{e.g. } shallow water waves or nonlinear optics. In addition, stability in (EKE) and (EKL) is found to be linked to stability in the scalar equation (qKdV). Since the relevant stability criteria are merely encoded by the negative signature of $(N+2)\times (N+2)$ matrices,  they can at least be checked numerically. In practice, when $N=1$ or $2$, this can be done without even requiring an ODE solver. Various numerical experiments are presented, which clearly discriminate between stable cases and unstable cases for (qKdV), (EKE) and (EKL), thus confirming some known results for the generalized KdV equation and the Nonlinear Schr\"odinger equation, and pointing out some new results for more general (systems of) PDEs.

\end{abstract}



\tableofcontents

\section{Introduction}
Hamiltonian PDEs include a number of model equations in mathematical physics, like the (generalized) 
Korteweg-de Vries equation  
(KdV) or the Non-{L}inear {Schr\"odinger} equation (NLS). These equations and many others are known to admit rich families of planar traveling wave solutions, with more or less degrees of freedom. The most `rigid' traveling waves are the so-called \emph{kinks}, corresponding to heteroclinic orbits of the ODEs governing their profiles. \emph{Periodic} traveling waves, which are the purpose of this paper,  have the highest number of degrees of freedom. In between kinks and periodic waves in terms of degrees of freedom, we can find \emph{solitary} waves, corresponding to homoclinic orbits. 

The actual existence of such waves follows from the Hamiltonian structure of the governing ODEs. We are most interested in their nonlinear stability, even though we can only hope for \emph{orbital} stability, because of \emph{translation invariance}. The most efficient approach to tackle the orbital stability of Hamiltonian traveling waves has been known as the Grillakis--Shatah--Strauss (GSS) theory \cite{GrillakisShatahStrauss}, which provides a way of using a \emph{constrained energy} as a Lyapunov function. This method crucially relies on the conservation of a quantity associated with translation invariance, termed `\emph{impulse}' by Benjamin \cite{Benjamin}, and known as the \emph{momentum} in the NLS literature. For solitary waves, the GSS theory provides a sufficient stability condition in terms of the convexity of the constrained energy as a function of the wave velocity. This constrained energy
happens to correspond to what was called `\emph{moment of instability}' by Boussinesq \cite{Boussinesq} more than 140 years ago.  Resurrected by Benjamin \cite{Benjamin72} in the early '70s, the ideas of Boussinesq have been made rigorous for many types of solitary waves in  \cite{GrillakisShatahStrauss,BonaSouganidisStrauss,BonaSachs,IFB05} (see also \cite{SBG-DIE}, \cite{DebievreGenoudRotaNodari} and references in \cite{AnguloPava}). Together with the Evans functions techniques brought in by Pego and Weinstein \cite{PegoWeinstein}, Kapitula and Sandstede \cite{KapitulaSandstede}, and many others, those pieces of work have led to a clear picture of which solitary waves are stable and which are not.

By contrast, the theory is much less advanced regarding periodic waves. Apart from the higher number of degrees of freedom, the main difficulty comes from the fact that the nice variational framework set up by Grillakis, Shatah, and Strauss does not work for all kinds of perturbations of those waves. As a matter of fact, the theory of \emph{linear} stability of periodic waves under `localized' perturbations~---~that is, perturbations going to zero at infinity~---~ is still in its infancy (see for instance \cite{BoDe,GallayHaragus,BottmanDeconinckNivala} as regards spectral stability for KdV and the cubic NLS, and \cite{R_linKdV} for asymptotic linear stability of KdV waves), and the nonlinear stability under such perturbations is an open problem. In \cite{BNR-JNLS14,BNR-GDR-AEDP}, the authors have contributed to the field 
by exhibiting several necessary conditions for the spectral stability of periodic waves in Hamiltonian PDEs. In particular, they have proved in a rather general setting that the hyperbolicity of the modulated equations `\`a la Whitham' is necessary for the spectral stability of the underlying wave. More precisely, the existence of a nonreal eigenvalue for the modulated equations implies a \emph{sideband instability}, which means that there are unstable modes for arbritrary small nonzero Floquet exponents. We shall not enter into details about these results here, we refer the reader to \cite{BNR-GDR-AEDP} and references therein --- see also the recent related analysis in \cite{JMMP_Klein-Gordon}. 

We are going to concentrate on the somehow easier problem of stability with respect to \emph{co-periodic} perturbations, that is, perturbations of the same period as the wave (or equivalently, corresponding to a zero Floquet exponent).  Our aim is to give as a clear picture as in the case of solitary waves, which by the way may be viewed as a limiting case of periodic waves~---~by letting their wavelength go to infinity. The great advantage of co-periodic perturbations is that they allow us to use the GSS approach in the simplest manner ---~by using basically only one additional conservation law (or constraint) to rule out the `bad' directions from the variational framework~---, and thus achieve \emph{nonlinear} stability results at once\footnote{Even in the simplest context of co-periodic stability, it is indeed still much unclear how spectral stability can imply nonlinear stability in Hamiltonian frameworks. Concerning localized perturbations we even lack a clear notion of \emph{dispersive spectral stability} that would be the analogue of \emph{diffusive spectral stability} \cite{S2,S1,JZ-generic,JNRZ-Inventiones,R}.}. This has been done for the cubic NLS by Gallay and Haragus \cite{GallayHaragus2} --- see also \cite{BottmanDeconinckNivala,GallayPelinovsky} for more recent results, dealing with \emph{subharmonic} perturbations, of which the period is a multiple of the period of the wave ---, and for the generalized KdV (gKdV) by Johnson \cite{Johnson} --- for the classical or the modified KdV, see also \cite{AnguloPava-Bona-Scialom}, \cite{AnguloPava_mKdV-NLS,Arruda} for co-periodic, orbital stability and \cite{Deconinck-Kapitula_KdV}, \cite{Deconinck-Nivala_mKdV} for a more general result, which even handles subharmonic perturbations. Both NLS and gKdV can be viewed as specific cases of the abstract setting we are going to consider. Furthermore, this setting is built to include the Euler--Korteweg system, a fairly general model that is involved in various applications (superfluids, water waves, incompressible fluid dynamics, and nonlinear optics). The abstract systems we consider hold in one space dimension, and read
\begin{equation}\label{eq:absHam}\partial_t\bU = \sJ (\Euler \Ham[\bU])\,,\end{equation}
where the unknown $\bU$ takes values in $\R^N$, $\sJ$ is a skew-adjoint differential operator, and $\Euler \Ham$ denotes the variational derivative of $\Ham=\Ham(\bU,\bU_x)$~---~the letter $\Euler$ standing for the \emph{Euler} operator. In practice, we are most concerned with the case $N=2$, which is the case for the various forms of the Euler--Korteweg system, as well as NLS.
In fact, \eqref{eq:absHam} includes both the original formulation of NLS, with $$\sJ=\left(\begin{array}{cc}0&1\\  -1&0\end{array}\right)$$ merely being the skew-symmetric matrix of Hamiltonian equations in `canonical' coordinates, and its fluid formulation via the Madelung transform. In the latter case, $$\sJ=\bB\partial_x\;  \mbox{with}\; \bB=\left(\begin{array}{cc}0&-1\\  -1&0\end{array}\right)\,.$$ From now on, we assume that $\sJ=\bB\partial_x$  with $\bB$ a \emph{symmetric} and \emph{nonsingular} matrix. This allows the case $N=1$ with $\sJ=\partial_x$, which
includes gKdV, and will enable us to make the connection with earlier results by Bronski, Johnson, and Kapitula.
Furthermore, if $N=2$, we assume that the Hamiltonian $\Ham$ splits as
$$\Ham=\Ham[\bU]=\Ham(\bU,\bU_x)=\Ec(\bv,\bu)+\En(\bv,\bv_x)\,,\;\mbox{with\,}\;\bU=\left(\begin{array}{c} \bv \\ \bu\end{array}\right)\,,$$
and then that 
$$\bB^{-1}=\,\left(\begin{array}{c|c}a& b \\ \hline b &0\end{array}\right)\,,\;b\neq 0\,.$$
Here above and throughout the paper, square brackets $[\cdot]$ signal a function of not only the dependent variable $\bU$ but also of its derivatives $\bU_x$, $\bU_{xx}$, \ldots
In this way, the abstract system in \eqref{eq:absHam} reads as a system of conservation laws
\begin{equation}\label{eq:absHamb}\partial_t\bU = \partial_x (\bB\,\Euler \Ham[\bU])\,.\end{equation}
We recall that, when, as in cases under consideration in the present paper, $\Ham$ depends only on $\bU$ and $\bU_x$, the $\alpha$-th component ($1\leq \alpha\leq N$) of the variational derivative $\Euler \Ham[\bU]$ is $$(\Euler \Ham[\bU])_\alpha\,:=\,\frac{\partial \Ham }{\partial U_\alpha}(\bU,\bU_x)\,-\,\Dif_{x}\left(\frac{\partial \Ham }{\partial U_{\alpha,x}}(\bU,\bU_x)\right)\,,$$
where $\Dif_{x}$ stands for the \emph{total} derivative, so that
$$\Dif_{x}\left(\frac{\partial \Ham }{\partial U_{\alpha,x}}(\bU,\bU_x)\right)\,=\,\frac{\partial^2 \Ham(\bU,\bU_x) }{\partial U_\beta \partial U_{\alpha,x}}\,U_{\beta,x}\,+\,\frac{\partial^2 \Ham(\bU,\bU_x) }{\partial U_{\beta,x} \partial U_{\alpha,x}}\,U_{\beta,xx}\,,
$$
where we have used Einstein's convention of summation over repeated indices.
The main examples that fit the abstract framework in \eqref{eq:absHamb} are, besides the generalized {Korteweg}-{de} {Vries} equation
$$\mbox{(gKdV)}\quad 
\partial_tv+\partial_xp(v)=-\partial_x^3v\,,$$
and its quasilinear counterpart which is written in the more general form
$$\mbox{(qKdV)}\quad \partial_tv=\partial_x(\Euler \en[v])\,,\quad \en=\en(v,v_x)\,,$$
the Euler--Korteweg system in Eulerian coordinates, 
\begin{equation*}\label{eq:EKabs1d}
\mbox{(EKE)}\quad \left\{\begin{array}{l}\partial_t\rho +\partial_x (\rho \vits)\,=\,0\,,\\ [5pt]
\partial_t \vits + \vits\partial_x\vits \,+\,\partial_x(\Euler \En [\rho])\,=\,0\,,\quad \En=\En(\rho,\rho_x)\,,
\end{array}\right.
\end{equation*}
or in mass Lagrangian coordinates, 
\begin{equation*}\label{eq:EKabsLag}
\mbox{(EKL)}\quad \left\{\begin{array}{l}\partial_s{\vol} \,=\,\partial_y{\vits}\,,\\ [5pt]
\partial_s {\vits} \,=\, \partial_y( \Euler \en [\vol])\,,\quad \en=\en(\vol,\vol_y)\,.
\end{array}\right.
\end{equation*}
We invite the reader to take a look at \cite{BNR-GDR-AEDP} for more details.

The main results of the present paper are concerned with periodic traveling wave solutions to  \eqref{eq:absHamb}, with applications to {\rm (qKdV)}, {\rm (EKE)} and (EKL). They consist of
a sufficient condition for their orbital, co-periodic stability,
and a necessary condition for their spectral, co-periodic stability.
Both are expressed in terms of the Hessian of the \emph{constrained energy}~---~ to be defined in Section \ref{s:not} hereafter~---~ viewed as a function of the $N+2$ parameters determining periodic waves. The value of this constrained energy at a given wave profile happens to be interpreted as an \emph{abbreviated action} integral along the corresponding orbit in the phase plane $\{(v,v_x)\}$. Remarkably enough, as far as capillary fluids are concerned, that is for the systems {\rm (EKE)} and {\rm (EKL)} with an energy of the form
\begin{equation}\label{eq:en}
\En = F(\rho)+\frac12 \Cap(\rho) \rho_x^2\,,\quad \en = f(\vol)+\frac12 \cap(\vol) \vol_y^2\,,
\end{equation}
the action integral physically corresponds to \emph{surface tension}. The abbreviated action of periodic wave profiles  
also admits an interpretation in terms of the averaged equations for modulated wavetrains, in that it is dual to the wavenumber in the generalized Gibbs relation satisfied by the averaged energy (see Eqs (63)(64) in \cite{BNR-JNLS14}) ---~this point of view is investigated further in a forthcoming paper.

Co-periodic stability conditions are replacements for the ---~simpler~--- ones known for solitary waves. Indeed, while the abbreviated action depends on $N+2$ parameters for periodic waves, it depends on the sole solitary wave velocity once the endstate of solitary waves is fixed in $\R^N$, in which case the abbreviated action merely coincides with the Boussinesq moment of instability. 

As is often the case, the necessary condition for spectral stability comes from an \emph{Evans function} calculation, see Section \ref{s:spectral}. Phrased explicitely,  it yields a sufficient condition for spectral instability, which is that the difference between $N$ and the negative signature of the Hessian of the abbreviated action be odd.

As to the sufficient condition for orbital stability, it relies on the GSS approach, together with a crucial algebraic relation, analogous to what has been pointed in \cite{PoganScheelZumbrun} (see also \cite[Proposition 5.3.1]{KapitulaPromislow} and references therein). This relation makes the connection between the negative signatures of two sorts of Hessians associated with the constrained energy, namely, the differential operator obtained as the Hessian at the wave's profile of the constrained energy viewed as a functional, and the $(N+1)\times (N+1)$ matrix, corresponding to what Kapitula and Promislow \cite{KapitulaPromislow} call a \emph{constraint matrix}, and arising here as the Hessian of the abbreviated action integral under the constraint that 
the period of waves is fixed. This leads us to introduce a most important \emph{orbital stability index}.
All this is made more precise in Section \ref{s:not} below, and those necessary/sufficient stability criteria are actually derived in an abstract setting in Sections \ref{s:spectral} and \ref{s:orb}. 

Section \ref{s:ex} is then devoted to the application of these abstract results to (qKdV), (EKL), and {\rm (EKE)}, with energies as in \eqref{eq:en}. A striking result is that, in all these cases, our sufficient condition for orbital stability mostly relies on the simple requirement that the negative signature of the Hessian of the abbreviated action be equal to $N$.

In addition, we point out a close connexion between stability criteria for (qKdV) and for the Euler--Korteweg systems {\rm (EKL)} and {\rm (EKE)}.  We show in Section \ref{ss:EK} that 
the systems {\rm (EKL)} and {\rm (EKE)} share the very same abbreviated action integral and constrained energy, in which the parameters of the waves turn out to be pairwise exchanged --- as well as the constraints actually, if the period itself is considered as a constraint. This readily implies that our spectral stability criterion coincides for these systems.
Furthermore,  we prove that {\rm (EKE)} and {\rm (EKL)} actually have the same orbital stability index, equal to 
the negative signature of the Hessian of the abbreviated action minus two,
even though the negative signatures of the unconstrained variational Hessians of respective Lagrangians can differ from each other. Regarding spectral stability with respect to `arbitrary' perturbations --- and not only co-periodic perturbations ---  we even show that the differential operators involved in the linearized systems associated with {\rm (EKE)} and {\rm (EKL)} are \emph{isospectral}. Even though this seems a very natural result, the underlying conjugacy between eigenfunctions is far from being trivial. Moreover, we stress that the spectral conjugacy is not restricted to co-periodic boundary conditions and respects the Floquet exponent by Floquet exponent Bloch-wave decomposition. Both spectral and variational connections are pointed out here for the first time, up to the authors' knowledge.

Some more specific examples --- dictated by more or less classical choices of nonlinearities --- are then investigated numerically in Section \ref{s:ex}. This part relies very much on the fact that all our stability criteria are expressed in terms of the abbreviated action integral, which can be computed in the phase plane without any need of an ODE solver. Its derivatives are then computed by means of finite differences. The coexistence of two grids of discretization --- one for the integral and one for finite differences, and the high condition number of the Hessian matrices that are to be computed, induce some numerical difficulties that have been coped with by a suitable choice of mesh sizes.
Numerous numerical experiments have been conducted, and their results are in accordance with those that can be computed analytically. In particular, our routine for computing the Hessian of the abbreviated action integral enables us to recover in a very precise way --- and up to the small amplitude limit and to the soliton limit --- the eigenvalues of modulated equations associated with some well-known completely integrable PDEs (namely, KdV, mKdV, and cubic NLS), as displayed in a forthcoming paper \cite{BMR2}.

Coming back to the analytical part of our results, let us mention the following, important difficulty.
In order to actually prove some orbital stability, we need a suitable local well-posedness theory for the Cauchy problem. 
This kind of result is of course heavily model-dependent. If there is for instance a huge literature on (g)KdV, it does not seem that anyone ever looked at the Cauchy problem for (qKdV) when the `capillarity' factor $\cap$ in
$$\en=f(v)+\frac12 \cap(v)v_x^2$$
is not constant. This is done in a forthcoming paper \cite{Mietka}. Regarding {\rm (EKE)} and (EKL), still with energies as in
\eqref{eq:en} with variable $\Cap$ and $\cap$, what we need is a basic adaptation to the 1D torus of earlier results dealing with the Cauchy problem on the whole real line \cite{BDD1d,BDDmultiD}.

\section{Summary of main results}\label{s:not}
In order to define the constrained energy, we observe that the system \eqref{eq:absHamb} formally admits the additional conservation law
\begin{equation}
\label{eq:impulsecl}
\partial_t\Impulse(\bU) = \partial_x( \Impulseflux[\bU])\,
\end{equation}
with $$\Impulse(\bU):= \tfrac{1}{2}\, \bU\cdot \bB^{-1} \bU\,,\; \Impulseflux[\bU]\,:=\,\bU \cdot \Euler \Ham[\bU]\,+\,\Legendre \Ham[\bU]\,,$$
$$\Legendre \Ham[\bU]\,:=\,
 U_{\alpha,x}\,\frac{\partial \Ham}{\partial  U_{\alpha,x}}(\bU,\bU_{x}) - \Ham(\bU,\bU_{x})\,.$$
The dots $\cdot$ in the definitions of $\Impulse$ and $\Impulseflux$ are for the `canonical' inner product 
 $\bU\cdot\bV=U_\alpha V_\alpha$ in $\R^N$.
Recall that the sans-serif letter $\Euler$ stands for the Euler operator defining variational derivatives. As to the other sans-serif letter $\Legendre$, it stands for a crude version --- without any change of variables involved --- of the `\emph{Legendre} transform', just defined by the formula above.
We see that $\Impulse$ is associated with the invariance of  \eqref{eq:absHamb} with respect to $x$-translations because of the algebraic relation
\begin{equation}\label{eq:translx}
\partial_x \bU = \partial_x (\bB\,\Euler \Impulse[\bU]) \,.
\end{equation}
As a consequence, for a travelling wave $\bU=\ubU(x-\speed t)$ of speed $\speed$ to be solution to \eqref{eq:absHamb}, one must have  by \eqref{eq:translx}  that
$$\partial_x(\Euler (\Ham+\speed\Impulse)[\ubU])=0\,,$$
or equivalently, there must exist $\blambda\in \R^N$ such that
\begin{equation}
\label{eq:EL}
\Euler (\Ham+\speed\Impulse)[\ubU] \,=\,\blambda\,.
\end{equation}
Equation \eqref{eq:EL} is the \emph{Euler--Lagrange equation} associated with the \emph{Lagrangian}
$$\Lag= \Lag(\bU,\bU_x;\blambda,\speed):= \Ham(\bU,\bU_x)+\speed\Impulse(\bU)\,-\,\blambda\cdot \bU\,.$$
From place to place we shall refer to the components of $\blambda$ as \emph{Lagrange multipliers}.
We thus see that  $\Legendre\Lag$ is  a first integral of the profile ODEs in \eqref{eq:EL}. We also notice that \eqref{eq:EL} implies
$$\Legendre\Lag[\ubU]\,=\, \Impulseflux[\ubU]+\speed \Impulse(\ubU)\,,$$
which is of course consistent with the conservation law in \eqref{eq:impulsecl}.  In this way, the possible profiles $\ubU$ are determined by the equations in \eqref{eq:EL} together with
\begin{equation}
\label{eq:ELham}
\Legendre\Lag[\ubU]\,=\,\mu\,,
\end{equation}
where $\mu$ is a constant of integration, which we shall sometimes refer to as an \emph{energy level}, since 
$\Legendre\Lag$ is the conserved `energy' associated with the  Lagrangian $\Legendre\Lag$.

All this roughly shows that a periodic travelling profile $\ubU$ is `generically' 
parametrized by $(\mu,\blambda,\speed)\in \R^{N+2}$. Furthermore, if such a profile $\ubU$  is of period $\Xi$, we can define the constrained energy of all $\Xi$-periodic, smooth enough functions $\bU$ by
$$\funct[\bU;\mu,\blambda,\speed]\,:=\, \int_{0}^{\Xi} (\Ham(\bU,\bU_x)+\speed\Impulse(\bU)-\blambda\cdot \bU+\mu)\,\dif x\,.$$
Denoting by 
\begin{equation}\label{eq:defTheta}
\Action(\mu,\blambda,\speed):= \funct[\ubU;\mu,\blambda,\speed]\,=\,\int_{0}^{\Xi} 
(\En(\ubv,\ubv_x)+\Ec(\ubv,\ubu) +\speed \Impulse(\ubU) - \blambda\cdot \ubU +\mu)\,\dif x\,,
\end{equation}
we can see by using  \eqref{eq:ELham} and a straightforward change of variable that 
\begin{equation}\label{eq:Theta}\Action(\mu,\blambda,\speed)\,=\,\oint \frac{\partial \En}{\partial v_x}(\ubv,\ubv_x)\,\dif v
\end{equation}
is the \emph{abbreviated action} of $\En$ along the orbit described by $\ubv$ in the $(\bv,\bv_x)$-plane. This is the reason why, as was observed in \cite[Proposition 1]{BNR-GDR-AEDP},  the partial derivatives of $\Action$ are merely given by
\begin{equation}
\label{eq:deraction}
\Action_\mu=\Xi\,,\;\nabla_{\blambda} \Action = -\int_{0}^{\Xi} \ubU\,\dif x\,,\;\Action_\speed=\int_{0}^{\Xi} \Impulse(\ubU)\,\dif x\,.
\end{equation}

We can now state our stability conditions in a more precise way.
\paragraph{Necessary condition for spectral, co-periodic stability.} 
A periodic traveling wave solution to \eqref{eq:absHam}, of profile $\ubU$ and period $\Xi$, is said spectrally stable with respect to co-periodic perturbations if the spectrum of 
the linearized operator $$\Lin:=\sJ \Hess(\Ham + \speed \Impulse)[\ubU]$$ in $L^2(\R/\Xi\Z)$ is purely imaginary. We cannot expect more than this neutral stability, because of symmetries\footnote{The eigenvalues of $\Lin$ outside real and imaginary axes arise as quadruplets $(\ev,\overline\ev,-\ev,-\overline\ev)$, nonzero real or imaginary eigenvalues come in pairs $(\ev,-\ev)$.}.
A necessary condition for spectral, co-periodic stability is 
\begin{itemize}
\item $\det (\Hess\Action)\leq 0$ in the case $N=1$,
\item $\det  (\Hess\Action)\geq 0$ in the case $N=2$.
\end{itemize}
The scalar case $N=1$ is a generalization to quasilinear equations of the results found by Bronski and Johnson \cite{BronskiJohnson}. For both cases, the proof is based on the fact that possible unstable eigenvalues $\ev$ are characterized by
$\Evans(\ev)=0$, with $$\Evans(\ev):= \det (\Monod(\Xi;\ev)-\Monod(0;\ev))$$ where $\Monod(\cdot;\ev)$ denotes the fundamental solution of the ODEs in $\Lin \bU=\ev \bU$, and on asymptotic expansions showing that when $r$ is real
$$D(r)\stackrel{r\to0}{=} (-r)^{N+2}\,|\det(\bB^{-1})|\,\det  (\Hess\Action)\,+\,o(r^{N+2})
\quad\mbox{and}\quad
\;D(r)>0 \mbox{ for } r\gg 1\,.$$
These results are collected in a rigorous manner in Theorem~\ref{thm:EvansEK} (for $N=2$) and Theorem~\ref{thm:EvansKdV} (for $N=1$) in Section \ref{s:spectral}.
 
\paragraph{Sufficient condition for orbital, co-periodic stability.} 
It is obtained through a variational argument. We assume that $\Xi_\mu=\Action_{\mu\mu}\neq 0$, and define the \emph{constraint matrix} as
$$\bC:= \frac{\check\nabla \Action_{\mu}\otimes \check\nabla \Action_{\mu}}{\Action_{\mu\mu}}\,-\,
\check\nabla^2 \Action\,,$$
with $\check\nabla$ being a shortcut for the gradient with respect to $(\blambda,\speed)$ at fixed $\mu$. 
(Note that the coefficients of $\bC$ are made, up to a factor $-1/\Action_{\mu\mu}$, of all the $2\times2$ minors of $\check\nabla^2 \Action$ containing  $\Action_{\mu\mu}$.)
Let us denote by $\Linvar$  the differential operator obtained as the Hessian of the constrained Hamiltonian:
$$\Linvar:=\Hess(\Ham+\speed \Impulse)[\ubU]\,.$$ If $\bC$ is nonsingular, and if the \emph{negative signatures} of the operator $\Linvar$ and of the matrix $\bC$ happen to coincide, then the periodic travelling wave solutions  to \eqref{eq:absHamb} of profile $\ubU$ are orbitally stable in $L^2(\R/\Xi\Z)$. 
This is the purpose of Theorem \ref{thm:orb}  in Section \ref{s:orb}.
Its proof is based on the following algebraic relation, shown in \cite{BNR-GDR-AEDP} (see also \cite{PoganScheelZumbrun}, \cite[Proposition 5.3.1]{KapitulaPromislow} for similar observations),
$$\negsign(\Linvar) = \negsign(\Linvar_{|T_{\ubU}{\man}}) \,+\,\negsign(\bC)\,,$$
where $\negsign$ denotes negative signature, and $T_{\ubU}{\man}$ is the \emph{tangent subspace} to the $(N+1)$ codimensional manifold
$${\man}:=\{\bU\in %\HH^0
L^2(\R/\Xi\Z)
\,;\;  \textstyle \int_{0}^{\Xi}\Impulse(\bU)\,\dif x=\int_{0}^{\Xi}\Impulse(\ubU)\,\dif x\,,
\int_{0}^{\Xi} \bU\,\dif x=\int_{0}^{\Xi} \ubU\, \dif x\}\,.$$
(The space $T_{\ubU}{\man}$ actually corresponds to what Kapitula and Promislow \cite{KapitulaPromislow} call an \emph{admissible space}.)
According to that relation between negative signatures, the fact that $\negsign(\Linvar)  = \negsign(\bC)$ implies that
the operator $\Linvar_{|T_{\ubU}{\man}}$ is nonnegative, and up to factoring out the null direction $\ubU_x$, this roughly shows that
the functional $\funct[\cdot;\mu,\blambda,\speed]$ has a local minimum at $\ubU$ and its translates $\ubU(\cdot+s)$ on  
$${\man}:=\{\bU\in %\HH^0
L^2(\R/\Xi\Z)
\,;\;  \textstyle \int_{0}^{\Xi}\Impulse(\bU)\,\dif x=\int_{0}^{\Xi}\Impulse(\ubU)\,\dif x\,,
\int_{0}^{\Xi} \bU\,\dif x=\int_{0}^{\Xi} \ubU\, \dif x\}\,.$$
Orbital stability can then be achieved by a contradiction argument as in \cite{BonaSouganidisStrauss,GrillakisShatahStrauss}, see \cite{BNR-GDR-AEDP}, or in a direct way as in \cite{Johnson,GallayHaragus} (see also \cite[\S 4.2 \& \S 7.3]{DebievreGenoudRotaNodari} for an interesting discussion of the pros and cons of these arguments). We choose the direct way for the proof of Theorem \ref{thm:orb} in Section \ref{s:orb}. 

In practice, we need to evaluate the \emph{orbital stability index} $\negsign(\Linvar)\,-\,\negsign(\bC)$. For the `concrete' systems we consider in Section \ref{s:ex}, we can infer  from a {Sturm}--{Liouville} argument that $\negsign(\Linvar) \,\in\,\{1,2\}$. In particular, extending results by Johnson \cite{Johnson} for (gKV), we show that $\Action_{\mu\mu}>0$ implies $\negsign(\Linvar) \,=\,1$. Furthermore, we see that when $N=1$, $\Action_{\mu\mu}\,\det(\bC)\,=\,\det(\Hess\Action)$. Hence a more explicit --- but partial --- version of the sufficient stability condition: 
$$\Action_{\mu\mu}>0\,,\;\det(\Hess\Action)<0\,,$$ which ensures indeed that 
$\negsign(\Linvar) =\negsign(\bC)$, and is of course consistent with Johnson's findings in the semilinear case. In fact, we recover in Section \ref{ss:qKdV} the more general sufficient condition for $\negsign(\Linvar)\,=\,\negsign(\bC)$ in (gKdV) that was later derived by Bronski, Johnson and Kapitula \cite{BronskiJohnsonKapitula} --- using in particular that $\Action_{\mu\mu}<0$ implies $\negsign(\Linvar) \,=\,2$ ---, namely
$$\Action_{\mu\mu}\neq 0\,,\; \left|\begin{array}{cc}
\Action_{\lambda\lambda} & \Action_{\mu\lambda} \\
\Action_{\lambda\mu}& \Action_{\mu\mu}\end{array}\right| \neq 0\,,\;\det(\Hess\Action)\neq0\,,\quad \negsign(\Hess\Action)=1\,.$$
More generally, for our concrete systems, we prove by related arguments --- Sturm-Liouville theory (see Lemma \ref{lem:Sturm}) and simple algebraic relations (see Proposition \ref{prop:ActionC}) --- the remarkable identity for the orbital stability index
$$\negsign(\Linvar)-\negsign(\bC)=\negsign(\Hess\Action)-N\,.$$ 
Therefore, our nonlinear stability result applies as soon as $\Action_{\mu\mu}\neq0$, $\det(\Hess\Action)\neq0$ and
$$
\negsign(\Hess\Action)\ =\ N\,.
$$

Our approach is applied to (qKdV) in Section \ref{ss:qKdV}, Theorem \ref{thm:qKdVorb}.

Details regarding the systems which have motivated this work, 
(EKL) and {\rm (EKE)}, are given in Section \ref{ss:EK}. A most important fact is that these systems share the very same abbreviated action, $\Action(\mu,\lambda,j,\sigma)$ defined in \eqref{eq:singleaction}, where $\sigma$ is the speed of EKE waves, $j$ is the `speed'\footnote{We use some quotes here because this is not a speed from the physical point of view, it actually corresponds to a mass transfer flux across the corresponding EKE waves.} of EKL waves, and
the roles of the parameters $\mu$ and $\lambda$ are exchanged when we go from {\rm (EKL)} to {\rm (EKE)} and vice versa:
$\mu$ is an \emph{energy level} for EKE waves and a \emph{Lagrange multiplier} for EKL waves, $\lambda$ is a Lagrange multiplier for EKE waves and an energy level for EKL waves\footnote{To avoid the introduction of too many notations, we have chosen to use the greek letters $\mu$ and $\lambda$ with a meaning in the  `concrete' examples {\rm (EKL)} and {\rm (EKE)} that is slightly different from their meaning in the abstract framework, see Table \ref{tb:notations}.}. As a consequence, the stability criteria which are expressed only in terms of $\det(\Hess\Action)$ and $\negsign(\Hess\Action)$ coincide for corresponding EKE waves and EKL waves. By contrast, the individual negative signatures $\negsign(\Linvar)$ and $\negsign(\bC)$ are in general not preserved by going from one formulation to the other. This explains why some simplified, partial criteria --- analogous to those in \cite{Johnson} for (gKdV) for instance --- are actually formulation-dependent.

The fact that waves should be simultaneously stable in both formulations seems very natural from a physical point of view. However, this is not that obvious to prove mathematically, because the mass Lagrangian coordinates are obtained from the Eulerian coordinates through a nonlinear and nonlocal change of variables that depends on the solution itself.

As far as spectral stability is concerned, and not only co-periodic stability actually, we prove in Theorem \ref{thm:conj} (also see Remark \ref{thm:conj}) that the corresponding linearized operators are indeed \emph{isospectral}. The proof is quite simple once we reformulate the nonlinear problems in a suitable way, but it reveals that the kind of conjugacy between those operators is not trivial. We are not aware of any earlier result of this type.

As regards the co-periodic orbital stability, its simultaneous occurrence in both formulations {\rm (EKL)} and {\rm (EKE)} is supported by the idea that corresponding EKE waves and EKL waves share the same constrained functional  $\funct[\cdot;\mu,\lambda,j,\sigma]$, and that the constraints are preserved by passing from Eulerian coordinates to mass Lagrangian coordinates\footnote{In fact, this is true provided that we also consider the period as a constraint, and thus prescribe the $(N+2)$ constraints $\textstyle\int_{0}^\Xi \dif x = \uuXi\,,\;\int_{0}^\Xi \bU(x) \dif x = \int_{0}^{\uuXi} \ubU(x) \dif x\,,\;
\int_{0}^\Xi \Impulse(\bU(x)) \dif x = \int_{0}^{\uuXi} \Impulse(\ubU(x) \dif x$. }. If the vanishing of our orbital stability index were exactly equivalent to the fact that the functional  $\funct[\cdot;\mu,\lambda,j,\sigma]$ has a local minimum on the constrained manifold ${\man}$ at the wave profile and its translates, this would show the equivalence of co-periodic orbital stability for {\rm (EKL)} and {\rm (EKE)}. We find out that the issue is a little more subtle by looking at our abstract result, Theorem \ref{thm:orb}. 
Recalling that the roles of the `concrete' parameters $\mu$ and $\lambda$ are exchanged when we go from {\rm (EKE)} to (EKL), we see that the main assumptions for applying Theorem \ref{thm:orb} to {\rm (EKE)} and {\rm (EKL)} are, 
$$\Action_{\mu\mu} \,\neq\,0\,,\quad \det(\Hess\Action)\,\neq\,0\,,\quad \negsign(\Hess\Action)\,=\,2\,,$$
for the former (see Theorem \ref{thm:EKEorb}), and 
$$ \Action_{\lambda\lambda} \,\neq\,0\,,\quad \det(\Hess\Action)\,\neq\,0\,,\quad \negsign(\Hess\Action)\,=\,2$$
for the latter (see Theorem \ref{thm:EKLorb}).
The slight discrepancy between these two sets of conditions obviously comes from the derivatives $\Action_{\mu\mu}$ and $\Action_{\lambda\lambda}$, which correspond respectively to the derivative with respect to $\mu$ of the wave period in Eulerian coordinates --- $\mu$ being the energy level in these coordinates --- and the derivative with respect to $\lambda$  of the wave period in mass Lagrangian coordinates --- $\lambda$ being the energy level in these coordinates. They are not a priori related to each other. However, as long as we regard the vanishing of either $\Action_{\mu\mu}$ or $\Action_{\lambda\lambda}$ as anomalous transitions, we can indeed think of the periodic waves as being simultaneously orbitally stable with respect to co-periodic perturbations in Eulerian coordinates and mass Lagrangian coordinates. Again, this is not that an obvious result, because the meaning of co-periodic perturbations is different from one formulation to the other\footnote{More precisely, the prescription of the period on one side corresponds to a `zero mass' perturbation on the other side, see Remark \ref{rem:stabind} for more details.}. 

Finally,  it turns out from a simple algebraic computation that
for {\rm (EKL)} the negative signature of the Hessian of the constrained Hamiltonian in mass Lagrangian coordinates coincides with the negative signature of the qKdV operator, $$\linvar:=\Hess(\en+\speed \imp)[\uvol]\,,$$  
where $\imp:=\frac12 \vol^2$ is the qKdV impulse, and $\speed=-j^2$ is prescribed by the speed of the EKL wave. Two ingredients then lead to a set of sufficient stability conditions for (EKL). The first one is that, as mentioned above,
$\negsign(\linvar)$ is known in terms of the sign of $\action_{\mu\mu}$, where $\action$ is an alternative notation for the abbreviated action associated with qKdV waves --- to avoid confusion with the one associated with EK waves, still denoted by $\Action$. The second ingredient follows from the observation that $\Action$ is explicitly related to $\action$, so that by the chain rule $\Hess\Action$ can be expressed in terms of $\Hess\action$ and $\nabla\action$.

\section{Co-periodic, spectral instability}\label{s:spectral}

\subsection{General setting}\label{ss:setting}

As in the introduction, we consider an abstract Hamiltonian system as in \eqref{eq:absHam}, with
$$\bU=\left(\begin{array}{c} \bv \\ \bu\end{array}\right)\,,\;\sJ=\bB\partial_x\,,\,\bB^{-1}=\,\left(\begin{array}{c|c}a& b \\ \hline b &0\end{array}\right)\,,\;b\neq 0\,\,,$$
$$\Ham=\Ham[\bU]=\Ham(\bU,\bU_x)=\Ec(\bv,\bu)+\En(\bv,\bv_x)\,,$$
and denote by $\Impulse$ the impulse --- or momentum --- defined by
$$\Impulse(\bU)=\frac12 \bU\cdot \bB^{-1} \bU\,.$$

Furthermore, let us assume that $\Ec$ is strongly convex with respect to $\bu$, and that $\En$ is strongly convex with respect to $\bv_x$. Both {\rm (EKE)} and {\rm (EKL)} fit this abstract framework, with
$$\bU=\left(\begin{array}{c} \rho \\ \vits \end{array}\right)\,,\;\Ec(\rho,\vits)=\frac12 \rho \vits^2\,,\;\En(\rho,\rho_x)=\frac12 \Cap(\vol)\rho_x^2+F(\rho)$$ 
for the former, and
$$\bU=\left(\begin{array}{c} \vol \\ \vitsL \end{array}\right)\,,\;\Ec(\vol,\vitsL)=\frac12  \vitsL^2\,,\;\En(\vol,\vol_y)=\frac12 \cap(\vol)\vol_y^2+f(\vol)$$
for the latter, as long as $\cap$, $\Cap$, and $\rho$ take positive values.

Recall that profiles $\ubU$ of periodic wave solutions to \eqref{eq:absHam} are characterized by the algebro-differential system made of \eqref{eq:EL}, \eqref{eq:ELham}, which depends on the parameters $(\mu,\blambda,\speed)\in \R^4$. Equivalently, by using \eqref{eq:translx},  we can view the profile $\ubU$ of waves of speed $\speed$ as a spatially periodic, and steady solution to the system \eqref{eq:absHam} rewritten in a mobile frame, which reads
\begin{equation}\label{eq:absHamc}\partial_t\bU = \bB\,\partial_x (\Euler (\Ham+\speed \Impulse)[\bU])\,.\end{equation}
For later use, let us note that \eqref{eq:absHamc} admits the conservation law
\begin{equation}
\label{eq:impulseclc}
\partial_t\Impulse(\bU) = \partial_x( \bU \cdot \Euler (\Ham+\speed \Impulse)[\bU]\,+\,\Legendre (\Ham+\speed \Impulse)[\bU] )\,,
\end{equation}
which is just  \eqref{eq:impulsecl} written in the mobile frame. In a similar way,  let us note that as soon as Eq.~\eqref{eq:EL} holds true, Eq.~\eqref{eq:ELham} equivalently reads 
\begin{equation}
\label{eq:ELhamc}
\ubU \cdot \Euler (\Ham+\speed \Impulse)[\ubU]\,+\,\Legendre (\Ham+\speed \Impulse)[\ubU] \,=\,\mu\,,
\end{equation}
We now fix such a periodic profile $\ubU$, say of period $\Xi$, and assume without loss of generality that
$\uvol_x$ vanishes at $x=0$, which will simplify a little bit our computations.
Linearizing \eqref{eq:absHamc} about $\ubU$, we receive the following system, in which the same notation $\bU$ now stands for the variation of the original $\bU$ around $\ubU$,
\begin{equation}\label{eq:absHamlin}\partial_t\bU = \bB\,\partial_x (\Linvar \bU)\,,\end{equation}
where 
$$\Linvar\ :=\ \Hess(\Ham+c\Impulse)[\ubU]\,.$$
In general, $\Hess \Ham[\ubU]$ is the differential operator defined by $$(\Hess \Ham[\ubU] \bU)_\alpha=\, \frac{\partial^2 \Ham[\ubU]}{\partial U_\alpha \partial U_\beta} U_\beta \,+\,
\frac{\partial^2 \Ham[\ubU]}{\partial U_\alpha \partial U_{\beta,x}} U_{\beta,x}
\,-\,\Dif_x\left(\frac{\partial^2 \Ham[\ubU]}{\partial U_{\alpha,x}\partial U_\beta} U_\beta\,+\,\frac{\partial^2 \Ham[\ubU]}{\partial U_{\alpha,x}\partial U_{\beta,x}} U_{\beta,x}\right)\,,$$
and similarly for $\Impulse$. However, $\Hess\Impulse[\ubU]$ happens to merely coincide here with the matrix $\bB^{-1}$, and 
$\Linvar$ is  like $\Hess \Ham[\ubU]$ a self-adjoint differential operator, of second order in $\bv$, with periodic coefficients of period $\Xi$. Our main purpose here is to derive a criterion ensuring that the composite differential operator
$$\Lin:= \bB\,\partial_x \Linvar$$
does not have any spectrum in the complex, right-half plane, when the eigenfunctions are sought $\Xi$-periodic.

Let us recall the classical observations that the profile equation in \eqref{eq:EL} implies, by differentiation in $x$, that
$\Linvar \ubU_x=0$, and by differentiation with respect to parameters $\mu$, $\blambda$, $\speed$,
we see that\footnote{Throughout the paper, subscripts $\mu$, $\lambda_1$, $\lambda_2$ --- or simply $\lambda$ when $N=1$ ---, and $\speed$ stand for partial derivatives with respect to the parameters $\mu$, $\blambda$, $\speed$.} 
$$\Lin \ubU_\mu=0\,,\;\Lin \ubU_{\lambda_{1}}=0\,,\;\Lin \ubU_{\lambda_{2}}=0\,,\;\Lin \ubU_\speed= - \ubU_x\,.$$

Another useful relation, which follows from \eqref{eq:absHamlin} but is more easily derived by linearizing \eqref{eq:impulseclc} about $\ubU$, is the following
$$\partial_t(\nabla\Impulse(\ubU) \cdot \bU) = \begin{array}[t]{l}\partial_x( \bU \cdot \Euler (\Ham+\speed \Impulse)[\ubU]\,+\,\ubU\cdot \Linvar \bU-\bU\cdot \nabla_{\bU}(\Ham+\speed \Impulse)[\ubU])\\ [3pt]
+\,\partial_x\left(\partial_{\bv_x\bv }\En[\ubv]\, \ubv_x \,\bv  
+\partial^2_{\bv_x}\En[\ubv]\,\ubv_x \,\bv_x \right)\,,\end{array}
$$
which can be simplified into
\begin{equation}
\label{eq:impulseclclin}
\partial_t(\nabla\Impulse(\ubU) \cdot \bU) = 
\partial_x(\ubU\cdot \Linvar \bU\,- \,\ucap \,\ubv_{xx} \bv \,+\,\ucap\,\ubv_x\,\bv_x)
\end{equation}
where $\ucap:=\partial^2_{\bv_x} \En[\uvol] >0$ by assumption.

For all $\ev\in \C$, let us consider the spectral problem associated with \eqref{eq:absHamlin}, 
\begin{equation}\label{Evans_eq}
\ev\,\bU\ =\ \Lin\bU
\end{equation}
which amounts to looking for solutions of \eqref{eq:absHamlin} of the form
$\ee^{\ev t} \bU(x)$. For such a solution, \eqref{eq:impulseclclin} readily implies by integration the following relation
$$\ev \int_{0}^{\Xi} \nabla\Impulse(\ubU(x)) \cdot \bU(x)\,\dif x = 
\ubU(0)\cdot [\Linvar \bU]\,-\,\ucap(0)\, \ubv_{xx}(0)\, [\bv]\,,$$
where we have used the shortcut $[f]$ for expressions of the form 
$f(\Xi)-f(0)$, and the fact that $\ubU$ is $\Xi$-periodic and chosen such that $\ubv_x(0)=0$. Furthermore, by also integrating \eqref{Evans_eq} over a period, we obtain
\begin{equation}
\label{eq:direct}\ev\,\int_{0}^{\Xi} \bU(x)\,\dif x\,=\,\bB \,[\Linvar \bU]\,.
\end{equation}
Therefore, we have
\begin{equation}
\label{eq:crucial}
\ev \int_{0}^{\Xi} \nabla\Impulse(\ubU(x)) \cdot \bU(x)\,\dif x = 
\ev\, \bB^{-1}\ubU(0)\cdot \int_{0}^{\Xi} \bU(x)\,\dif x\,-\,\ucap(0)\, \ubv_{xx}(0)\, [\bv]\,.
\end{equation}
This relation will be used in a crucial way below, given that $\ucap(0)>0$ and $\ubv_{xx}(0)\neq 0$.

\subsection{Evans function}\label{ss:Evans}

The eigenvalue equations in \eqref{Evans_eq} consist of a system of ODEs, which is of first order in $\bu$ and third order in $\bv$. Rewriting this system as a first-order system of four ODEs --- linear ODEs with $\Xi$-periodic coefficients,  and denoting by $\Monod(\cdot;\ev)$ its fundamental solution, we see that the existence of a $\Xi$-periodic, nontrivial solution to \eqref{Evans_eq} is equivalent to $\Evans(\ev)=0$, where
$$
\Evans(\ev)\ :=\ \det (\Monod(\Xi;\ev)-\Monod(0;\ev)).
$$
The function $\Evans:\C\to\C$ is called an Evans function\footnote{Here with Floquet exponent equal to zero, since we search for co-periodic eigenfunctions only; a more general Evans function would be $\Evans(\ev,\alpha)\ :=\ \det (\Monod(\Xi;\ev)-\ee^{i\alpha}\Monod(0;\ev))$, defined for all Floquet exponents $\alpha \in \R$.}.

\begin{theorem}\label{thm:EvansEK}
In the framework of Sections \ref{s:not} and \ref{ss:setting}, we make the `generic' assumption
\begin{itemize}
\item[{\bf (H0)}] There exists an open set $\Omega$ of $\R^{N+2}$ and a family of periodic traveling profiles $\ubU$ smoothly parametrized by   $(\mu,\blambda,\speed)\in \Omega$ such that 
\eqref{eq:EL}-\eqref{eq:ELham} hold true, that is
$$\Euler (\Ham+\speed\Impulse)[\ubU] \,=\,\blambda\,,\qquad \ubU\cdot \Euler\Ham[\ubU] + \Legendre\Ham[\ubU]+ \speed \Impulse(\ubU)\,=\,\mu\,.$$
\end{itemize} 
If $N=2$ then the Evans function $\Evans$ defined above has the following asymptotic behaviors
\begin{equation}
D(r)\stackrel{r\to0}{=} r^{4}\ (-\det(\bB^{-1}))\ \det  (\Hess\Action)\,+\,o(r^{4})\,,\qquad D(r)>0 \mbox{ for } r\gg 1\,.
\end{equation}
If  $\det  (\Hess\Action)<0$, then the corresponding wave is spectrally unstable.
\end{theorem}

Notice that $\det(\bB^{-1})=-b^2$ so that $(-\det(\bB^{-1}))=|\det(\bB^{-1})|$.

\begin{proof}
We begin by observing that $\ubU_x,\ubU_\mu,\ubU_{\lambda_1},\ubU_{\lambda_2}$ are solutions to the ODEs in \eqref{Evans_eq} with $z=0$, as a consequence --- by differentiation in $x$ --- of the relations
\begin{equation}\label{eq:Linvarderprof}
\Linvar \ubU_x=0\,,\;\Linvar \ubU_\mu=0\,,\;\Linvar\ubU_{\lambda_1}=\left(\begin{array}{c}1\\0\end{array}\right)\,,\;
\Linvar\ubU_{\lambda_2}=\left(\begin{array}{c}0\\1\end{array}\right)\,,
\end{equation}
which themselves come from the differentiation of \eqref{eq:EL}.
Furthermore, $(\ubU_x,\ubU_\mu,\ubU_{\lambda_1},\ubU_{\lambda_2})$ is an independent family. Indeed, would a linear combination $\beta_0\ubU_x+\beta_3\ubU_\mu+\beta_1\ubU_{\lambda_1}+\beta_2\ubU_{\lambda_2}$ be zero, the last two equations in \eqref{eq:Linvarderprof} above would imply $\beta_1=\beta_2=0$, so that $\beta_0\ubU_x+\beta_3\ubU_\mu=0$, which turns out to be impossible unless $\beta_0=\beta_3=0$. Indeed, by differentiation of \eqref{eq:ELham} we see that
$$- \,\ucap \,\ubv_{xx} \ubv_\mu \,+\,\ucap\,\ubv_x\,\ubv_{\mu,x}\,=\,1\,,\quad - \,\ucap \,\ubv_{xx} \ubv_x \,+\,\ucap\,\ubv_x\,\ubv_{xx}\,=\,0\,,$$
so that $\ubv_\mu$ and $\ubv_x$ cannot be colinear.
(In the computation here above, we actually have differentiated \eqref{eq:ELham} under the form \eqref{eq:ELhamc}, and used the same simplification as in the derivation of \eqref{eq:impulseclclin}, as well as
the first two 
equations 
in \eqref{eq:Linvarderprof} to cancel out the terms 
$\ubU\cdot \Linvar \ubU_\mu$, and $\ubU\cdot \Linvar \ubU_x$, 
even though these 
simplifications are 
not necessary to show that $\ubU_x$ and $\ubU_\mu$ cannot be colinear. If the second equation obtained in this way is trivial, this is not the case for the first one.)  Since we have enforced $\ubv_x(0)=0$, the above relations imply in particular
\begin{equation}\label{e:keyrel}
\ucap(0) \,\ubv_{xx}(0) \ubv_\mu(0)\ =\ -1\,.
\end{equation}

This preliminary observation that $(\ubU_x,\ubU_\mu,\ubU_{\lambda_1},\ubU_{\lambda_2})$ is an independent family allows us to consider the family of independent solutions $\bU^j(\cdot;z)$, $j=1,2,3,4$, to  \eqref{Evans_eq} defined by the initial conditions
$$
(\bv^j(0;z),\bv^j_x(0;z),\bv^j_{xx}(0;z),\bu^j(0;z))^{\sf T}
\ =\ \left\{\begin{array}{rcl} 
(\ubv_x(0),\ubv_{xx}(0),\ubv_{xxx}(0),\ubu_x(0))^{\sf T}&\textrm{ if }&j=1,\\
(\ubv_\mu(0),\ubv_{\mu,x}(0),\ubv_{\mu,xx}(0),\ubu_\mu(0))^{\sf T}&\textrm{ if }&j=2,\\
(\ubv_{\lambda_1}(0),\ubv_{\lambda_1,x}(0),\ubv_{\lambda_1,xx}(0),\ubu_{\lambda_1}(0))^{\sf T}&\textrm{ if }&j=3,\\
(\ubv_{\lambda_2}(0),\ubv_{\lambda_2,x}(0),\ubv_{\lambda_2,xx}(0),\ubu_{\lambda_2}(0))^{\sf T}&\textrm{ if }&j=4.\\
\end{array}\right.
$$

We see that the Evans function equivalently reads $\Evans(\ev)=E(\ev)/\Delta$, where
$$
E(z)\ :=\ \left|\begin{array}{rrlcrl}
&[\!\!\!\!&\bv^1]& \cdots &[\!\!\!\!&\bv^{4}]\\ 
&[\!\!\!\!&\bv^1_x]& \cdots &[\!\!\!\!&\bv^{4}_x]\\ 
&[\!\!\!\!&\bv^1_{xx}]& \cdots &[\!\!\!\!&\bv^{4}_{xx}]\\ 
&[\!\!\!\!&\bu^1]&\cdots &[\!\!\!\!&\bu^{4}]\\ 
\end{array}\right|\,,\;
\Delta:= \left|\begin{array}{rcl}
\bv^1(0)& \cdots &\bv^{4}(0)\\ 
\bv^1_x(0)& \cdots &\bv^{4}_x(0)\\ 
\bv^1_{xx}(0)& \cdots &\bv^{4}_{xx}(0)\\ 
\bu^1(0)&\cdots &\bu^{4}(0)\\ 
\end{array}\right|
\ .
$$
Here above and in the remaining part of this proof, the notation $[\cdot]$ is reserved for differences of values between $x=\Xi$ and $x=0$, and has nothing to do with the evaluation of functionals anymore.
 
\paragraph{Low frequency expansion.} This is basically a variation on the computation made in \cite[Section B.2]{BNR-JNLS14}, with a few more details for the reader's convenience. We begin by observing that  
$$[\Linvar \bU]\,=\,\left(\begin{array}{c} - \ucap(0)\,[\bv_{xx}] \\ \ualpha(0)\,[\bu] \end{array}\right)\,+\,\left(\begin{array}{ccc}
* & * & * \\
* & 0 & 0 \end{array}\right) \,\left(\begin{array}{c} \,[\bv] \\ \,[\bv_x] \\ \,[\bu] \end{array}\right)$$
where $\ualpha:=\partial^2_{\bu} \Ec(\ubv,\ubu) >0$ by assumption, and $*$ stand for immaterial real numbers coming from the evaluation of $\ubU$ and its derivatives at $0$. Therefore, we can make some row combinations in the determinant defining $E(\ev)$ by using \eqref{eq:direct}, and obtain
$$
E(\ev)\ =\ -\ev^2\det(\bJ^{-1})
\,\ucap(0)^{-1}\,\ualpha(0)^{-1}\,
\left|\begin{array}{rrlcrl}
&[\!\!\!\!&\bv^1]& \cdots &[\!\!\!\!&\bv^4]\\ 
&[\!\!\!\!&\bv^1_x]& \cdots &[\!\!\!\!&\bv^4_x]\\ 
&\!\!\!\!&\int_0^{\uXi} \bU^1 &\cdots&\!\!\!\!&\int_0^{\uXi} \bU^4
\end{array}\right|.
$$
We can proceed in a similar way by using \eqref{eq:crucial}, which yields
$$
\begin{array}{rcl}
E(\ev)&=&\ev^3\det(\bJ^{-1})
\,\ucap(0)^{-2}\,\ualpha(0)^{-1}\,
\ubv_{xx}(0)^{-1}\\[1em]
&&\times\left|\begin{array}{rrlcrl}
&\!\!\!\!&\int_0^{\uXi}
\nabla\Impulse(\ubU)\cdot \bU^1
& \cdots &\!\!\!\!&\int_0^{\uXi}
\nabla\Impulse(\ubU)\cdot \bU^4\\[1em]
&[\!\!\!\!&\bv^1_x]& \cdots &[\!\!\!\!&\bv^4_x]\\[1em]
&\!\!\!\!&\int_0^{\uXi} \bU^1 &\cdots&\!\!\!\!&\int_0^{\uXi} \bU^4
\end{array}\right|.
\end{array}
$$

Now, we observe that $(\bU^1)_{|\ev=0}=\ubU_x$, and by differentiating $\ev \bU^1=\Lin\bU^1$ we see that $\ubU_x= \Lin (\partial_\ev\bU^1)_{|\ev=0}$. Therefore, $\bU^1+\ubU_c$ is in the kernel of $\Lin$, which is spanned by $(\ubU_x,\ubU_\mu,\ubU_{\lambda_1},\ubU_{\lambda_2})$ since this is an independent family of solutions to $\Lin \bU=0$, which is equivalent to a first-order system of four ODEs. We also have that $(\bU^2)_{|\ev=0}=\ubU_\mu$, $(\bU^1)_{|\ev=0}=\ubU_{\lambda_1}$, $(\bU^1)_{|\ev=0}=\ubU_{\lambda_2}$, by the choice of initial conditions.
We thus find that
$$
\begin{array}{l}
E(\ev)=-\ev^4\det(\bJ^{-1})
\,\ucap(0)^{-2}\,\ualpha(0)^{-1}\,
\ubv_{xx}(0)^{-1}\\[1em]
\;\times\left|\begin{array}{cccc}
\int_0^{\uXi}
\nabla\Impulse(\ubU)\cdot \ubU_c& 
\int_0^{\uXi}
\nabla\Impulse(\ubU)\cdot \ubU_\mu&
\int_0^{\uXi}
\nabla\Impulse(\ubU)\cdot \ubU_{\lambda_1}&
\int_0^{\uXi}
\nabla\Impulse(\ubU)\cdot \ubU_{\lambda_2}\\[1em]
[\ubv_{c,x}]&[\ubv_{\mu,x}]&[\ubv_{\lambda_1,x}]&[\ubv_{\lambda_2,x}]\\[1em]
\int_0^{\uXi} \ubU_c&\int_0^{\uXi} \ubU_\mu&\int_0^{\uXi} \ubU_{\lambda_1}&\int_0^{\uXi} \ubU_{\lambda_2}
\end{array}\right|+\cO(\ev^5)\,,\\[4em]
\end{array}$$
which equivalently reads, according to \eqref{eq:deraction},
$$
\begin{array}{l}
E(\ev)=-\ev^4\det(\bJ^{-1})
\,\ucap(0)^{-2}\,\ualpha(0)^{-1}\,
\ubv_{xx}(0)^{-1}
\;\times\left|\begin{array}{cccc}
\Action_{\speed \speed}& 
\Action_{\speed \mu}&
\Action_{\speed \lambda_1}&
\Action_{\speed \lambda_2}\\[1em]
[\ubv_{c,x}]&[\ubv_{\mu,x}]&[\ubv_{\lambda_1,x}]&[\ubv_{\lambda_2,x}]\\[1em]
\Action_{\lambda_1 \speed}&
\Action_{\lambda_1 \mu}&\Action_{\lambda_1 \lambda_1}&\Action_{\lambda_1 \lambda_2}
\\[1em]
\Action_{\lambda_2 \speed}&
\Action_{\lambda_2 \mu}&\Action_{\lambda_2 \lambda_1}&\Action_{\lambda_2 \lambda_2}
\end{array}\right|+\cO(\ev^5)\,.\\[4em]
\end{array}$$
Finally, using that $[\ubv_{a,x}]=- \Xi_a \ubv_{xx}(0)$, for $a=\speed,\mu,\lambda_1,\lambda_2$ --- which merely comes from the differentiation of the relation $\ubv_x(\Xi)=\ubv_x(0)$ with respect to those parameters --- we obtain
$$
\begin{array}{l}
E(\ev)=\ev^4\det(\bJ^{-1})
\,\ucap(0)^{-2}\,\ualpha(0)^{-1}
\;\times\left|\begin{array}{cccc}
\Action_{\speed \speed}& 
\Action_{\speed \mu}&
\Action_{\speed \lambda_1}&
\Action_{\speed \lambda_2}\\[1em]
\uXi_c&\uXi_\mu&\uXi_{\lambda_1}&\uXi_{\lambda_2}\\[1em]
\Action_{\lambda_1 \speed}&
\Action_{\lambda_1 \mu}&\Action_{\lambda_1 \lambda_1}&\Action_{\lambda_1 \lambda_2}
\\[1em]
\Action_{\lambda_2 \speed}&
\Action_{\lambda_2 \mu}&\Action_{\lambda_2 \lambda_1}&\Action_{\lambda_2 \lambda_2}
\end{array}\right|+\cO(\ev^5),
\end{array}
$$
that is, recalling also from \eqref{eq:deraction} that $\Xi=\Action_\mu$, 
$$E(\ev)=\ev^4\det(\bJ^{-1})
\,\ucap(0)^{-2}\,\ualpha(0)^{-1}\,\det \Hess\Action+\cO(\ev^5)\,.$$
On the other hand, using that 
$$\Linvar \bU(0)\,=\,\left(\begin{array}{c} - \ucap(0)\,\bv_{xx}(0) \\ \ualpha(0)\,\bu(0) \end{array}\right)\,+\,\left(\begin{array}{ccc}
* & * & * \\
* & 0 & 0 \end{array}\right) \,\left(\begin{array}{c} \,\bv(0) \\ \,\bv_x(0) \\ \,\bu(0)\end{array}\right)$$
in general, and Eqs in \eqref{eq:Linvarderprof}, 
we can compute explicitly $\Delta$ by making row combinations again. By using \eqref{e:keyrel} in the last step, we thus find that
$$
\begin{array}{rcl}
\Delta
&=&
\left|\begin{array}{cccc}
\ubv_x(0)&\ubv_\mu(0)&\ubv_{\lambda_1}(0)&\ubv_{\lambda_2}(0)\\ 
\ubv_{xx}(0)&\ubv_{\mu,x}(0)&\ubv_{\lambda_1,x}(0)&\ubv_{\lambda_2,x}(0)\\ 
\ubv_{xxx}(0)&\ubv_{\mu,xx}(0)&\ubv_{\lambda_1,xx}(0)&\ubv_{\lambda_2,xx}(0)\\ 
\ubu_x(0)&\ubu_\mu(0)&\ubu_{\lambda_1}(0)&\ubu_{\lambda_2}(0)\\ 
\end{array}\right|\\
&=&-
\,\ucap(0)^{-1}\,\ualpha(0)^{-1}\,
\left|\begin{array}{cccc}
0&\ubv_\mu(0)&\ubv_{\lambda_1}(0)&\ubv_{\lambda_2}(0)\\ 
\ubv_{xx}(0)&\ubv_{\mu,x}(0)&\ubv_{\lambda_1,x}(0)&\ubv_{\lambda_2,x}(0)\\ 
0&0&1&0\\ 
0&0&0&1\\ 
\end{array}\right|\\[1em]
&=&
\,\ucap(0)^{-1}\,\ualpha(0)^{-1}\,
\,\ubv_{xx}(0)\,\ubv_\mu(0)\\[1em]
&=&\,-\,\ucap(0)^{-2}\,\ualpha(0)^{-1}.
\end{array}
$$
Altogether with the expansion of $E(\ev)$, this gives
$$
\begin{array}{l}
\Evans(\ev)\stackrel{\ev\to0}{=}-\ev^4\det(\bJ^{-1})
\,\det \Hess\Action\,+\,\cO(\ev^5)\,.
\end{array}
$$

\paragraph{High frequency expansion.}

In order to find the sign of $\Evans(\ev)$ for $\ev\in\R_+$ large enough,
we can invoke a \emph{homotopy argument}\footnote{Usual in many related computations, but apparently used for the first time in a periodic context ; see Remark~\ref{r:Mat}.}. For this argument to work out, we must check  
that there exists $R>0$ so that $\Evans(r)\neq 0$ for $r\geq R$, and moreover that this $R$ can be found to be uniform along the family of Evans functions associated with a continuous path going from the operator $\Lin$ to a simpler operator, say $\tLin$, for which we can compute the Evans function $\tEvans$ explicitly. We choose
$$\tLin\,:=\,\tbJ\, \partial_x \,\Hess \tHam[\ubU]\,,\;\tbJ^{-1}:=\left(\begin{array}{c|c}0&b\\ \hline b&0\end{array}\right)\,,\;\tHam[\bU]:= \frac12 \bu^2+\frac12 \bv_x^2\,,$$
and postpone the search for $R$ to the next paragraph. 

Once we have this $R$, we know that
$\Evans(r)$ and $\tEvans(r)$ have the same sign for $r\geq R$. Let us compute $\tEvans(r)$.
The eigenvalue equations $\tLin \bU=\ev \bU$ equivalently read
$$\left\{\begin{array}{rcl}
-\partial_x^3\bv &=& \ev b\,\bu\,,\\
\partial_x \bu &=& \ev b\,\bv\,,
\end{array}\right.$$
or
$$
\partial_x\left(\begin{array}{c} \bv \\ \bv_x \\ \bv_{xx} \\ \bu\end{array}\right)=
{\mathbb A}(\ev) \left(\begin{array}{c} \bv \\ \bv_x \\ \bv_{xx} \\ \bu\end{array}\right)\,,\;
{\mathbb A}(\ev) := \left(\begin{array}{cccc}0&1&0&0\\0&0&1&0\\0&0&0&-zb\\zb&0&0&0\end{array}\right)\,.
$$
It is then a simple exercise to compute $
\tEvans(z)\ = \ \det(\ee^{\Xi {\mathbb A}(\ev)}-1)$. In particular, for $\ev=r\in (0,+\infty)$, ${\mathbb A}(r)$ has four distinct eigenvalues, $\pm (1\pm i) \sqrt{r |b|/2}$. Therefore, it is diagonalizable and
$$\det(\ee^{\Xi {\mathbb A}(r)}-1)= (e^{\Xi\,(1+i)\sqrt{ r |b|/2}}-1) \,(e^{\Xi\,(1-i)\sqrt{ r |b|/2}}-1)\, (e^{\Xi\,(-1+i)\sqrt{ r |b|/2}}-1)\, (e^{\Xi\,(-1-i)\sqrt{ r |b|/2}}-1)$$
$$\,= \ \left|e^{\Xi\,(1+i)\sqrt{ r |b|/2}}-1\right|^2\ \left|e^{%i
\Xi\,
(-1+i)\sqrt{r |b|/2}}-1\right|^2> 0.
$$

\paragraph{High enough frequencies are not eigenvalues.} For any $\theta\in [0,1]$, we set 
$$\Lin_\theta:=\tbJ_\theta\, \partial_x \,\Hess \tHam_\theta[\ubU]\quad\textrm{where}\ \tHam_\theta=\theta\Ham+(1-\theta)\tHam
\quad\textrm{and}\quad
\tbJ_\theta^{-1}=\theta\,\bJ^{-1}+(1-\theta)\tbJ^{-1}.$$
Notice that this does define a $\tbJ_\theta$ since the formula for $\tbJ_\theta^{-1}$ defines a matrix with determinant $\det(\bJ^{-1})$. The aim is to find $R\in (0,+\infty)$ such that, for all $\theta\in [0,1]$, the operator 
$\Lin_\theta$ does not have any $\Xi$-periodic eigenfunction associated with a real eigenvalue $r\geq R$. Observing that $\Lin_\theta$ is --- on purpose --- exactly of the same form as $\Lin$, by just replacing $\bJ^{-1}$, $\Ec$, $\En$ by
$$\tbJ^{-1}_\theta:=\left(\begin{array}{c|c}\theta a &b\\ \hline b&0\end{array}\right)\,,\; \Ec_\theta=\theta \Ec+ (1-\theta) (\tfrac12 \bu^2)\,,\;\En_\theta=\theta \En+ (1-\theta) (\tfrac12 \bv_x^2)\,,$$
we can drop $\theta$ and seek $R$ for $\Lin$ under the only assumptions that $b$ is fixed and positive, $a$ may vary while staying bounded, $\ucap=\partial^2_{\bv_x} \En[\ubv]$ and $\ualpha=\partial^2_{\bu} \Ec[\ubv,\ubu]$ are bounded, positive and bounded away from zero. Such an $R$ can be derived from some rough \emph{a priori} estimates. A similar computation was made in \cite[Section B.1]{BNR-JNLS14}, with a slight mistake which can be fixed by modifying the high order estimate accordingly with what follows.

\noindent
Let us write the eigenvalue equations \eqref{Evans_eq} in a more explicit way,
$$\left\{\begin{array}{l}
\partial_x(\uM \bv + \ubeta \bv +\ugamma \bu)\ +\ \speed\ \partial_x(a \bv + b \bu)\ =\ r\ ( a \bv + b \bu)\,,\\
\partial_x(\ugamma \bv + \ualpha \bu)\ +\ \speed\ \partial_x(b \bv)\ =\ r \ b  \bv\,,
\end{array}\right.$$
where $\uM:=\Hess \En[\ubv]=- \partial_x \ucap \partial_x + \uq$ with $\uq$ bounded, and 
$\ubeta:=\partial^2_{\bv} \Ec[\ubv,\ubu]$, $\ugamma:=\partial_{\bu\bv} \Ec[\ubv,\ubu]$ bounded too.
On the one hand, taking the inner product of the system above with 
$$\bJ \left(\begin{array}{c}\bv\\ \bu\end{array}\right)= \left(\begin{array}{c}\bu/b\\ \bv/b-a\bu/b\end{array}\right)$$ in $L^2(\R/\Xi\Z)$ and integrating by parts, we find that
$$r (\|\bv\|_{L^2}^2+\|\bu\|_{L^2}^2)\begin{array}[t]{l}= \int_{0}^{\Xi} 
\big(\  \bu_x\, \partial_x(\ucap\, \bv_x) \, + \, \bu \, \partial_x ( (\uq+\ubeta) \bv + \ugamma  \bu )
 + (\bv-a\bu)\, \partial_x(\ugamma \bv +\ualpha \bu) 
 \big)/b \ \dif x
 \\[5pt]
  \leq C\, (\|\bv\|_{H^2}^2+\|\bu\|_{H^1}^2)
   \end{array}
$$
for some constant $C$ depending only the bounds on $a$, $\ucap$, $\uq$, $\ualpha$, $\ubeta$, $\ugamma$.
On the other hand, taking the inner product of the system above with $(\bv_x,\bu_x)^{{\sf T}}$,
and integrating by parts again, we obtain
$$0\begin{array}[t]{l}=  \int_{0}^{\Xi} \begin{array}[t]{r}
\!\!\big(\  \bv_{xx}\, \partial_x(\ucap\, \bv_x) \, + \, \bv_x \, \partial_x ( (\uq+\ubeta) \bv + \ugamma  \bu )
 + \bu_x\, \partial_x(\ugamma \bv +\ualpha \bu) \\
 \quad + \speed\, \bv_x\,\partial_x(a\bv+b\bu)\,+\, \speed \,\bu_x\,\partial_x(b\bv)
 \big) \ \dif x
 \end{array}
 \\[5pt]
 =  \int_{0}^{\Xi}
\big( \ucap\,\bv_{xx}^2\,+\,\ualpha\,\bu_x^2\big)\,\dif x\,+\,
\begin{array}[t]{r}  \int_{0}^{\Xi}
\big(\ \ucap_{x} \bv_x \bv_{xx} - \ugamma\, \bv_{xx}\,\bu\,+ \bv_x\,  \partial_x((\uq+\ubeta+a\speed)\bv) \\
+\, \bu_x\, \partial_x((\ugamma\,+\,2b\speed)\, \bv
 \big) \ \dif x
 \end{array}
 \\[5pt]
  \geq \frac12\,\int_{0}^{\Xi}
\big( \ucap\,\bv_{xx}^2\,+\,\ualpha\,\bu_x^2\big)\,\dif x -\,C'\, (\|\bv\|_{H^1}^2+\|\bu\|_{L^2}^2)
   \end{array}
$$
for some other constant $C'$ depending only the bounds on $a$, $\ucap$, $\ucap_x$, $\uq$, $\ualpha$, $\ubeta$, $\ugamma$, $\ugammax$. Using once more that $\ucap$ and $\ualpha$ are positive and bounded away from zero,
we thus find a constant $C''$ such that
$$\|\bv\|_{H^2}^2+\|\bu\|_{H^1}^2\,\leq\,C'' (\|\bv\|_{L^2}^2+\|\bu\|_{L^2}^2)\,.$$
(Note that $\|\bv_x\|_{L^2}^2$ has been absorbed in the left-hand side.) Therefore, we have 
$$r (\|\bv\|_{L^2}^2+\|\bu\|_{L^2}^2)\leq CC''(\|\bv\|_{L^2}^2+\|\bu\|_{L^2}^2)\,,$$
which implies that $\bv$ and $\bu$ must be zero if $r> CC''$.

\paragraph{Conclusion.}
By the mean value theorem, a necessary condition for stability of the wave is that $D$ does not change sign on $(0,+\infty)$. 
Combining the low frequency expansion with the fact that $\Evans(r)$ is positive for large $r$, we obtain the necessary condition for co-periodic stability
$$
-\ \det (\bJ^{-1})\ \det \Hess \Action
\ \geq\ 0.
$$
which requires that $\det \Hess \Action$ be nonnegative. 
\end{proof}

We can show a similar result in the case $N=1$, which corresponds to (qKdV), or equivalently
$\bU=\bv, \;\Ham[\bU]=\en(\bv,\bv_x)\,,\;\bJ=1\,,\;\Impulse=\imp=\tfrac12 \bv^2$
in the abstract form \eqref{eq:absHamb}. In this case, the profile equation \eqref{eq:EL} reduces to the scalar Euler--Lagrange equation 
$$\Euler (\en
+\speed
\imp)[\ubv] \,=\,\lambda\,,$$
and the abbreviated action is defined as in \eqref{eq:defTheta} by just dropping the $\bu$-component and $\Ec$:
$$\action(\mu,\lambda,\speed):= 
\int_{0}^{\Upsilon} (\en(\ubv,\ubv_x)+\speed \imp(\ubv) - \lambda  \ubv +\mu)\,\dif x\,.$$
For a reason that will be clarified in Section \ref{ss:EK}, we have substituted $\Upsilon$ for $\Xi$ as the period of the wave.
The Evans function is also defined as before by
$$
\evans(\ev)\ :=\ \det (\monod(\Upsilon;\ev)-\monod(0;\ev)),
$$
where $\monod(\cdot;\ev)$ is now the fundamental solution of the third order ODE 
\begin{equation}
\label{Evans_eq-scalar}
\partial_x (\Hess(\en 
+\speed 
\imp)[\ubv] \bv) =\ev \ \bv
\end{equation}
viewed as a first-order system.

\begin{theorem}\label{thm:EvansKdV} Under assumption {\bf (H0)}, in the case $N=1$, the Evans function defined above has the following asymptotic behaviors
\begin{equation}
\evans(r)\stackrel{r\to0}{=} r^{3} \det  (\Hess\action)\,+\,o(r^{3})\,,\quad\evans(r)<0 \mbox{ for } r\gg 1\,.
\end{equation}
If  $\det  (\Hess\action)>0$, then the corresponding wave is spectrally unstable.
\end{theorem}

For consistency, notice that here $\det(\bJ^{-1})=1$.

\begin{proof}
We can basically copy-paste computations from the proof of Theorem \ref{thm:EvansEK}, by taking $\bJ=1$, and dropping
the $\bu$-components, $\Ec$, and all terms in $\lambda_2$. We just have to pay attention to where the signs change.
A first, obvious one is that $\det \bJ$ is now positive. There is another change of sign in the computation of $E(\ev)$ near zero, because there is now only one row where we find a minus sign by writing $\int_{0}^{\Xi} \ubv\,\dif x=-\action_\lambda$. So these two changes of sign give the claimed asymptotic expansion at zero.
There is a third, and last change of sign in the computation of $\evans(r)$ for large $r$. Indeed, the ODE to solve is now
$$-\partial_x^3\bv = r \,\bv\,,
$$
which has the three wavenumbers $k_0=-\sqrt[3]{r}$, $k_1=\ee^{i\pi/3} \sqrt[3]{r}$, and $\overline{k_1}$ for $r>0$. Therefore,
$\tevans(r)=(\ee^{k_0}-1)\ |\ee^{k_1}-1|^2<0$.
\end{proof} 

\begin{remark}\label{r:Frederic}
\textup{
Our main implicit restriction --- even at the abstract level --- is that we only consider systems for which, by a suitable number of integrations, the original traveling-wave profile system may be converted in a planar Hamiltonian, reduced, profile equation. Otherwise one would expect, as in the well-studied --- and algebraically much simpler --- case of solitary waves, some orientation index to enter in formulas for stability indices. For solitary waves, the computation of the necessary orientation index from geometric invariants is still the object of intense research ; see for instance \cite{Chardard-Bridges} and references therein.}
\end{remark}

\begin{remark}\label{r:Mat}
\textup{
As said in the introduction, this result confirms earlier findings by Johnson \cite{Johnson} in the special case when $\cap$ is constant. Though our presentation for (qKdV) does not follow the one by Johnson for (gKdV), some of our steps do not differ significantly. By contrast, some others are fundamentally different, as required by the quasilinear nature of our problem. For instance, in the foregoing proof, the purpose of the homotopy argument and the auxiliary resolvent estimates is precisely to reduce computations to a semilinear case. When equations are already in semilinear form, those techniques are not needed, and a readily regular limit $\lambda\to\infty$ leads to a constant-coefficient problem. Likewise, the local well-posedness results invoked in Section \ref{s:ex}  for applying our abstract orbital stability to actual PDEs is dramatically improved for semilinear versions of those equations. These observations are instances of the usual rule of thumb that departures of quasilinear strategies from semilinear ones are only required when some high-frequency control is needed.}
\end{remark}

\begin{remark}\label{r:Krein}
\textup{
It is instructive to seek parallels of our results in the classical stability theory for steady states of finite-dimensional Hamiltonian systems of ordinary differential equations. Indeed, up to replacing Evans' functions with characteristic polynomials, the foregoing proofs echo the classical proof that steady states at which the Hessian of the Hamiltonian is nonsingular and has an odd number of negative directions are spectrally instable. We claim that the analogy goes further. On the one hand, it follows from our proof that the sign of $\det  (\Hess\Action)$ provides us with the parity of the number of eigenvalues of $\Lin$ on $(0,+\infty)$. On the other hand, as we shall see in the next section, the relevant Hessian there is the constrained Hessian $\Linvar_{|T_{\ubU}{\man}}$. Even though we do not endeavor to prove it in the present paper, we do expect that the parity of the negative signature of $\Linvar_{|T_{\ubU}{\man}}$ and of the number of eigenvalues of $\Lin$ on $(0,+\infty)$ coincide, so that the results of the current section could be thought of as a direct analogue of the finite-dimensional case, with  $\Linvar_{|T_{\ubU}{\man}}$ in place of the classical Hessian of the Hamiltonian. The deepest way to prove our claim regarding the agreement of those parities consists in examining the \emph{Krein signature} of eigenvalues. Indeed, building on the fact that eigenvectors of $\Lin$ are orthogonal for the quadratic form associated with $\Linvar$, one may expect to prove that the negative signature of $\Linvar_{|T_{\ubU}{\man}}$ is the number of eigenvalues $\ev$, with $\Real(\ev)\geq0$ and negative Krein signature, and our claim on parity would then follow from the fact that eigenvalues with $\Real(\ev)\geq0$ but $\Imag(\ev)\neq0$ come in pairs. See detailed discussions, precise statements and proofs of similar results in \cite{KapitulaKevrekidisSandstede,BronskiJohnsonKapitula,BronskiJohnsonKapitulaII} and \cite[Chapter~7]{KapitulaPromislow}.}
\end{remark}

\section{Co-periodic, orbital stability}\label{s:orb}

\subsection{Abstract setting}\label{ss:abs}

We still consider a Hamiltonian system of the form \eqref{eq:absHamb}, which we relabel here for the reader's convenience:
\begin{equation}\label{eq:absHamB}
\partial_t\bU =   \partial_x(\bJ \Euler \Ham[\bU])\,,
\end{equation}
with $\bU$ taking values in $\R^N$, $\bJ$ a nonsingular, symmetric $N\times N$ matrix, $\Ham=\Ham(\bU,\bU_x)$, 
and denote by $\Impulse$ the impulse --- or momentum --- defined by
$$\Impulse(\bU)=\frac12 \bU\cdot \bB^{-1} \bU\,.$$

Eq.~\eqref{eq:absHamB} is obviously a system of $N$ (local) conservation laws of order at most three in the spatial variable $x$.  
Of course we have in mind the more specific forms of $\bJ$ and $\Ham$ that are described in Section \ref{ss:setting}, and correspond to either (qKdV) in the case $N=1$, or to (EK) in the case $N=2$. In the latter case, the first conservation law is of order one, and the second one is of order three as regards the first dependent variable, and of order one for the second dependent variable. Section \ref{s:ex} is devoted to a detailed investigation of those `examples'. Here, we refrain from restricting to any specific form of $\bJ$ and $\Ham$, in order to emphasize the crucial ingredients in the proof of co-periodic, orbital stability. The reader is referred to Sections \ref{ss:qKdV} and \ref{ss:EK} for an application of our abstract result (Theorem \ref{thm:orb} below) to  respectively (qKdV) and (EK).

As far as smooth solutions of \eqref{eq:absHamB} are concerned, they satisfy at least two additional, local conservation laws. One is the conservation of the impulse $\Impulse$, Eq.~\eqref{eq:impulsecl}, and the other one is the conservation of the Hamiltonian $\Ham$, which explicitly reads
$$\partial_t\Ham[\bU] = \partial_x(\tfrac12 \Euler\Ham[\bU]\cdot \bB\, \Euler\Ham[\bU]  \, + \,\nabla_{\bU_x}\Ham[\bU] \cdot \bB\, \partial_x \Euler\Ham[\bU])\,.$$
All these local conservation laws have the most important consequence that, along smooth periodic solutions to \eqref{eq:absHamB}, we have
\begin{equation}
\label{eq:globalCL}
\frac{\dif }{\dif t} \int_{0}^\Xi \bU \dif x = 0\,,\quad \frac{\dif }{\dif t} \int_{0}^\Xi \Impulse(\bU) \dif x = 0\,,\quad \frac{\dif }{\dif t} \int_{0}^\Xi \Ham[\bU] \dif x = 0\,,
\end{equation}
if $\Xi$ denotes the period of those solutions. For a given $\Xi$, we call \emph{energy space}, and denote by $\HH_\Xi$ a dense subspace of $(L^2(\R/\Xi\Z))^N$ on which the functional
$$\bU\mapsto \int_{0}^\Xi \Ham[\bU] \dif x$$
is (at least) ${\class}^2$. Behind this loose definition, we merely have in mind $\HH_\Xi=H^1(\R/\Xi\Z)$ for (qKdV), and 
$\HH_\Xi=H^1(\R/\Xi\Z)\times L^2(\R/\Xi\Z)$ for (EK). Note that the linear functional
$$\bU\mapsto \int_{0}^\Xi \bU \dif x$$
is automatically ${\class}^\infty$ on $(L^2(\R/\Xi\Z))^N$, by the embedding $L^1(\R/\Xi\Z)\hookrightarrow L^2(\R/\Xi\Z)$,
and that the quadratic functional
$$\bU\mapsto \int_{0}^\Xi \Impulse(\bU) \dif x$$
is also ${\class}^\infty$ on $(L^2(\R/\Xi\Z))^N$ by the Cauchy--Schwarz inequality. Therefore, whatever the constants 
$(\mu,\blambda,\speed)\in \R^{N+2}$, the functional
$$\bU\mapsto\funct[\bU;\mu,\blambda,\speed]\,:=\, \int_{0}^{\Xi} (\Ham(\bU,\bU_x)+\speed\Impulse(\bU)-\blambda\cdot \bU+\mu)\,\dif x$$
is ${\class}^2$ on the {energy space} $\HH_\Xi$. Let us point out that these functionals --- in particular $\funct[\bU;\mu,\blambda,\speed]$ --- are invariant under the action of spatial translations $\bU\mapsto \bU(\cdot+s)$ on $\Xi$-periodic functions $\bU$, so that they are indeed well defined for $\bU$ viewed as a function on  the circle $\R/\Xi\Z$. Furthermore, 
\eqref{eq:globalCL} implies that $\funct[\cdot;\mu,\blambda,\speed]$ is preserved along smooth, $\Xi$-periodic solutions of 
\eqref{eq:absHamB}.

Our main assumptions are the following.

\begin{itemize}
\item[{\bf (H0)}] There exists an open set $\Omega$ of $\R^{N+2}$ and a family of periodic traveling profiles $\ubU$ parametrized by   $(\mu,\blambda,\speed)\in \Omega$ such that 
the profile equations in \eqref{eq:EL}-\eqref{eq:ELham}, 
$$\Euler (\Ham+\speed\Impulse)[\ubU] \,=\,\blambda\,,\; \Legendre\Ham[\ubU]- \speed \Impulse(\ubU)+\blambda\cdot\ubU \,=\,\mu\,,$$
hold true, and the mapping $(\mu,\blambda,\speed)\in \Omega \mapsto (\ubU,\Xi)\in {\class}^2_b(\R)\times \R$ is continuously differentiable, where $\Xi$ denotes the period of the profile $\ubU$.
\item[{\bf (H1)}] The derivative of the period $\Xi$ with respect to the energy level $\mu$, denoted by $\Xi_\mu$, does not vanish on $\Omega$, and the abbreviated action integral
$$\Action(\mu,\blambda,\speed)
=\,\int_{0}^{\Xi} (\Ham(\ubU,\ubU_x)+\speed \Impulse(\ubU) - \blambda\cdot \ubU +\mu)\,\dif x$$
is such that the matrix 
$$\bC:=
\frac{\check\nabla \Action_{\mu}\otimes \check\nabla \Action_{\mu}}{\Action_{\mu\mu}}\,-\,
\check\nabla^2 \Action\,,$$
is nonsingular for $(\mu,\blambda,\speed)\in \Omega$, with $\check\nabla=\Big(\begin{array}{c}\nabla_{\blambda} \\ \partial_\speed\end{array}\Big)$.
\item[{\bf (H2)}] For all $\Xi$ in the set of periods achieved on $\Omega$, there exists a dense subspace $\HH_\Xi$ of $(L^2(\R/\Xi\Z))^N$, and an open subset of $\HH_\Xi$ containing all the profiles $\ubU$ on which the functional
$$\bU\mapsto \int_{0}^\Xi \Ham[\bU] \dif x$$
is ${\class}^2$, and if we denote by $\langle \cdot,\cdot\rangle$ the dual product between $\HH_\Xi'$ and $\HH_\Xi$, there exists a positive number $\alpha$ such that
$$\|\bU\|_{\HH_\Xi}^2=\langle \Hess\Ham[\ubU] \bU, \bU\rangle\,+\,\alpha\,\|\bU\|^2_{L^2}$$ defines an equivalent norm on $\HH_\Xi$, uniformly in the parameters defining the  $\Xi$-periodic profile $\ubU$.
\item[{\bf (H3)}] For all $\Xi$ in the set of periods achieved on $\Omega$, there exists a dense subspace $\W_\Xi$ of the energy space $\HH_\Xi$ on which the Cauchy problem for \eqref{eq:absHamb} is locally well-posed.
\end{itemize}

These assumptions are discussed for (qKdV) and (EK) in Section \ref{s:ex}.

\subsection{An index for co-periodic, orbital stability}\label{ss:indexorb}

\begin{theorem}\label{thm:orb}
Under the assumptions {\bf (H0)}-{\bf (H1)}-{\bf (H2)}-{\bf (H3)}, for all $(\mu,\blambda,\speed)\in \Omega$ such that
\begin{itemize}
\item the negative signature of $\bC$ equals the one of the operator $\Linvar=\Hess(\Ham+\speed \Impulse)[\ubU]$,
\item the kernel of $\Linvar$ is spanned by $\ubU_x$,
\end{itemize}
the periodic wave of profile $\ubU$ is conditionally, orbitally stable in the following sense.

\noindent
For all $\varepsilon>0$, there exists $\eta>0$ so that, for all 
$\bU_0\in \W_\Xi$ such that $\|\bU_0-\ubU\|_{\HH_\Xi}\leq \eta$,
if $T$ is the maximal time of existence of the solution $\bU:t\mapsto \bU(\cdot,t)\in \W_\Xi$ to \eqref{eq:absHamB} such that $\bU(0)=\bU_0$, then
$$\inf_{s\in \R}\|\bU(\cdot,t)-\ubU(\cdot+s)\|_{\HH_\Xi}\leq \varepsilon\,,\quad \forall t\in [0,T)\,.$$
\end{theorem}

In practice, for the cases discussed in Section \ref{s:ex}, the fact that the kernel of $\Linvar$ is spanned by $\ubU_x$ is a consequence of the assumption in {\bf (H1)} that $\Xi_\mu$ is nonzero. This is nevertheless a crucial point in the proof of Theorem \ref{thm:orb}, that is why we state it explicitly. Otherwise, the most important and nontrivial assumption is 
$$\negsign(\Linvar)-\negsign(\bC)=0\,.$$
In this respect, the integer $\negsign(\Linvar)-\negsign(\bC)$ may be called an \emph{orbital stability index}. It is investigated in more details in Section \ref{s:ex}, where we show in particular its connection with the negative signature of $\Hess\Action$ itself, through the remarkable formula
\begin{equation}\label{eq:indicevssignAction}
\negsign(\Linvar)-\negsign(\bC)=\negsign(\Hess\Action)-N\,.
\end{equation}

\begin{remark}
\textup{
At the abstract level of Theorem \ref{thm:orb}, it is not obvious that the stability criterion $\negsign(\Linvar)=\negsign(\bC)$ contains the necessary condition $(-1)^{N}\det\Hess\Action\geq 0$ derived in Section \ref{s:spectral} (Theorems \ref{thm:EvansEK} and \ref{thm:EvansKdV}) for spectral stability. However, if we admit \eqref{eq:indicevssignAction} for a while, we readily see that a null orbital stability index means that $\negsign(\Hess\Action)=N$, which implies that
$(-1)^N \,\det \Hess\Action\geq 0$. For an alternative connection, see Remarks~\ref{r:Mat} \& \ref{r:sw}.}
\end{remark}

\begin{remark}
\textup{
In special cases, genuine orbital stability can be inferred from conditional orbital stability. This was done\footnote{In the sense that it is proved there that solutions starting sufficiently close to the background wave are global in time. However the result is still conditional in the sense that $\bU_0$ is required to have higher regularity --- $\bU_0\in \W_\Xi$ ---  than afforded by the energy norm $\|\,\cdot\,\|_{\HH_\Xi}$.} for example by Bona and Sachs \cite[Theorem 4]{BonaSachs} regarding the stability of solitary waves in {\rm (EKL)} with a constant $\cap$. Indeed, in this case {\rm (EKL)} is a \emph{semilinear} system of PDEs, and a bound on the low order derivatives of the energy space yields a bound on higher order derivatives, by differentiation of the PDEs and by commutator estimates for the lower order terms.}
\end{remark}

\begin{remark}\label{r:sw}
\textup{
Going on with our analogy with the finite-dimensional, ODE case initiated in Remark~\ref{r:Mat}, we observe that the above theorem essentially shows that, if the negative signature of $\Linvar_{|T_{\ubU}{\man}}$ is \emph{zero} (see Eq.~\eqref{eq:negsign} below) then the wave is nonlinearly stable, while the results of the previous section show that, if that negative signature is \emph{odd} then the wave is spectrally unstable. Even in the finite-dimensional case, this offers a genuine dichotomy\footnote{Up to considering also the opposite of the `natural' Hamiltonian, if needed.} only for \emph{planar} ODEs. For periodic waves, by contrast with what happens for the effectively lower-dimensional solitary waves or kinks, it turns out that we \emph{never} have a genuine dichotomy. Nevertheless, by transferring \emph{infinite-dimensional} conditions on $\Linvar_{|T_{\ubU}{\man}}$ to
\emph{finite-dimensional} ones on $\Hess\Action$, we have come up with the neat following criteria: 
\begin{itemize}
\item if $\negsign(\Hess\Action)-N$ is zero then the wave is nonlinearly stable (by Theorem \ref{thm:orb} plus \eqref{eq:indicevssignAction});
\item if $\negsign(\Hess\Action)-N$ is odd then the wave is spectrally unstable (by Theorems \ref{thm:EvansEK} \& \ref{thm:EvansKdV}).
\end{itemize}
}
\end{remark}

\begin{proof}[Proof of Theorem \ref{thm:orb}]
It heavily relies on \cite[Theorem 3]{BNR-GDR-AEDP}, in which we proved that
\begin{equation}\label{eq:negsign}
\negsign(\Linvar) = \negsign(\Linvar_{|T_{\ubU}{\man}}) \,+\,\negsign(\bC)\,,
\end{equation}
with
$$T_{\ubU}{\man}:= \{\bU\in (L^2(\R/\Xi\Z))^N\,;\;\textstyle\int_{0}^{\Xi} \bU\,\dif x=0\,,\;\int_{0}^{\Xi} \bU\cdot \nabla\Impulse(\ubU)\,\dif x =\,0\}\,,$$
and is a revisited --- expanded and more accurate --- version of the proof of Corollary 2 in \cite{BNR-GDR-AEDP}.
We proceed in the same spirit as in \cite[Theorem 7.3]{DebievreGenoudRotaNodari}.

\paragraph{Step 1.} Orbital stability within the constraint manifold
$${\man}\,=\,\{\bU\in \HH_\Xi\,;\;  \textstyle 
\int_{0}^{\Xi} \bU\,\dif x=\int_{0}^{\Xi} \ubU\, \dif x\,,\;\int_{0}^{\Xi}\Impulse(\bU)\,\dif x=\int_{0}^{\Xi}\Impulse(\ubU)\,\dif x \,\}\,.$$
The assumption $\negsign(\Linvar) = \negsign(\bC)$ implies $\negsign(\Linvar_{|T_{\ubU}{\man}})=0$ by the formula in \eqref{eq:negsign} recalled above.
Since the kernel of $\Linvar$ is spanned by $\ubU_x$, this altogether implies the existence of $C>0$ so that 
$$\langle\Linvar \bV,\bV\rangle\,\geq \,C\,
\|\bV\|^2_{L^2}$$
for all $\bV\in \HH_\Xi$ such that
\begin{equation}\label{eq:allconstraints}
\textstyle\int_{0}^{\Xi} \bV\,\dif x=0\,,\quad\int_{0}^{\Xi} \bV\cdot \nabla\Impulse(\ubU)\,\dif x =\,0\,,\quad\int_{0}^{\Xi} \bV\cdot \ubU_x\,\dif x=0\,.
\end{equation}
In addition, $C$ is bounded by below by a uniform positive constant when the parameters $(\blambda,\speed)$ vary in a (small) compact subset of $\Lambda$, the projection of $\Omega$ onto $\R^{N+1}$, and $\mu$ is implicitly defined as a function of $(\blambda,\speed)$ by the fixed period $\Xi$ (which is made possible by the assumption $\Xi_\mu\neq 0$ in {\bf (H1)}).
Therefore, by {\bf (H2)}, there exists another positive constant $\widetilde{C}$ such that
$$\langle\Linvar \bV,\bV\rangle\,\geq \,\widetilde{C}\,\|\bV\|_{\HH_\Xi}^2\,,$$
for all $\bV\in \HH_\Xi$ satisfying \eqref{eq:allconstraints}. Indeed, using the equivalent norm on $\HH_\Xi$ given in {\bf (H2)},
up to augmenting $\alpha$ in such a way that
$$- \speed \,\bV\cdot \bB^{-1} \bV \leq {\alpha} \,\bV\cdot\bV\,,\;\forall \bV\in\R^N\,,$$
(and $\speed$ possibly varying in a compact set of possible wave velocities), and recalling that
$\Linvar= \Hess \Ham[\ubU] +\speed \bB^{-1}$, 
we see that 
$$\|\bV\|_{\HH_\Xi}^2 =\langle\Linvar \bV,\bV\rangle\,-\,\speed\, \int_{0}^\Xi \bV\cdot \bB^{-1} \bV\,\dif x\,+\,
\alpha \int_{0}^{\Xi} \bV\cdot \bV\, \dif x\,\leq \,\left(1+\frac{2\alpha}{C}\right)\, \langle\Linvar \bV,\bV\rangle$$
for $\bV\in \HH_\Xi$ satisfying \eqref{eq:allconstraints}.

Now, by Taylor expansion we have 
$$\funct[\bU;\mu,\blambda,\speed]-\funct[\ubU;\mu,\blambda,\speed]\,=\,
\tfrac12\,\langle \Linvar (\bU-\ubU) , \bU-\ubU\rangle\,+\,
o(\|\bU-\ubU\|_{\HH_\Xi}^2)
$$
for $\bU\in \HH_\Xi$ close to $\ubU$.
In the expansion here above, the first order term has vanished because of the profile equation $\Euler (\Ham+\speed\Impulse)[\ubU] \,=\,\blambda$, which means that $\ubU$ is a critical point of the functional 
$\funct[\cdot;\mu,\blambda,\speed]$, and the second order term comes from the fact that the operator 
$\Linvar=\Hess(\Ham+\speed\Impulse)$ is precisely the second variational derivative of the functional 
$\funct[\cdot;\mu,\blambda,\speed]$. For those $\bU\in \HH_\Xi$ close to $\ubU$ that in addition belong to 
${\man}$, we have
$$\bU-\ubU\,=\,\bV+o(\|\bU-\ubU\|_{\HH_\Xi})\,,\; \bV\in T_{\ubU}{\man}\,,$$
which means that $\bV$ satisfies the constraints in \eqref{eq:allconstraints} except for the last one. A nowadays well-known trick to enforce this constraint is to use translation invariance and the implicit function theorem to prove the following.
\begin{lemma}\label{lem:transl}
 For all $\epsilon>0$ we define
$$\UU_\epsilon\,=\,\{\bU\in \HH_\Xi\,;\;\inf_{s\in \R} \| \bU-\ubU(\cdot+s)\|_{\HH_\Xi}\,\leq\,\epsilon\}\,.$$
There exists $\epsilon_0>0$ and a ${\class}^1$ function $\tau: \UU_{\epsilon_0}\to \R$ such that
for all $\bU\in \UU_{\epsilon_0}$, 
$$\int_{0}^{\Xi} (\bU(x+ \tau(\bU))-\ubU(x))\cdot \ubU_x(x)\,\dif x=0\,,\quad \|\bU(\cdot+ \tau(\bU))-\ubU\|_{\HH_\Xi}\leq\epsilon_0\,,$$
and
$$ \|\bU(\cdot+ \tau(\bU))-\ubU\|_{\HH_\Xi} \rightarrow 0\quad\mbox{when}\quad 
\inf_{s\in \R} \| \bU-\ubU(\cdot+s)\|_{\HH_\Xi} \rightarrow 0\,.$$
\end{lemma}
As a consequence, for $\bU\in \UU_{\epsilon_0}\bigcap {\man}$, we have by the invariance of $\funct[\cdot;\mu,\blambda,\speed]$ under spatial translations,
$$\funct[\bU;\mu,\blambda,\speed]-\funct[\ubU;\mu,\blambda,\speed]\,\begin{array}[t]{l}=\,\funct[\widetilde{\bU};\mu,\blambda,\speed]-\funct[{\ubU};\mu,\blambda,\speed]\\ [5pt]
=\,\tfrac12\,\langle \Linvar \widetilde{\bV},\widetilde{\bV}\rangle\,\dif x \,+\,
o(\|\widetilde{\bU}-{\ubU}\|_{\HH_\Xi}^2)\,,\end{array}
$$
where we have denoted by $\widetilde{\bU}$ the translate $\bU(\cdot+\tau(\bU))$, so that
$$\widetilde{\bU}-{\ubU}\,=\,\widetilde{\bV}+o(\|\widetilde{\bU}-{\ubU}\|_{\HH_\Xi})\,,\; \widetilde{\bV}\in T_{\ubU}{\man}\,,\;\textstyle
\int_{0}^{\Xi} \widetilde{\bV}\cdot \ubU_x\,\dif x=0\,.$$
Therefore, we have the lower bound
$$\funct[\bU;\mu,\blambda,\speed]-\funct[\ubU;\mu,\blambda,\speed]\,\begin{array}[t]{l}\geq\,\frac{\widetilde{C}}{2}\,\|\widetilde{\bV}\|^2_{\HH_\Xi}\,+\,o(\|\widetilde{\bU}-\ubU\|_{\HH_\Xi}^2)\\ [5pt]
\geq \frac{\widetilde{C}}{4}\,\|\widetilde{\bU}-\ubU\|^2_{\HH_\Xi}\,,\end{array}$$
for $\bU\in \UU_{\epsilon_1}\bigcap {\man}$ and some $\epsilon_1\in (0,\epsilon_0]$.

This is all what we need to use $\funct[\cdot;\mu,\blambda,\speed]$ as a Lyapunov function to show orbital stability within 
${\man}$. Even though this is a classical reasoning, we give it for completeness.
For all $\varepsilon\in (0,\epsilon_1]$, for all $\bU\in \UU_{\epsilon_1}\bigcap {\man}$ such that
$\|\widetilde{\bU}-\ubU\|_{\HH_\Xi}=\varepsilon$, we have
$$\funct[\bU;\mu,\blambda,\speed]\geq \funct[\ubU;\mu,\blambda,\speed]+ \frac{\widetilde{C}}{4}\,\varepsilon^2=:m(\varepsilon)\,>\,\funct[\ubU;\mu,\blambda,\speed]\,.$$
By continuity of $\funct[\cdot;\mu,\blambda,\speed]$ at $\ubU$ and its invariance under spatial translations, there exists $\eta\in (0,\varepsilon]$ such that 
$$\funct[\bU;\mu,\blambda,\speed]<m(\varepsilon)$$
for all $\bU\in \UU_{\eta}$.
Therefore, if we take $\bU_0\in \W_\Xi \bigcap\UU_{\eta}\bigcap {\man}$, and denote by $\bU(t)$ the solution at time $t\in [0,T)$ of \eqref{eq:absHamB} such that $\bU(0)=\bU_0$,
we have 
$$\funct[\bU(t);\mu,\blambda,\speed]=\funct[\bU_0;\mu,\blambda,\speed]<m(\varepsilon)$$
for all $t\in [0,T)$. Since $\bU(t)$ belongs to ${\man}$ by the conservation of $\int_{0}^\Xi \bU\dif x$ and 
$\int_{0}^\Xi \Impulse(\bU)\dif x$ in \eqref{eq:globalCL}, the definition of $m(\varepsilon)$ and the mean value theorem 
prevent $\|\widetilde{\bU}(t)-\ubU\|_{\HH_\Xi}$ from growing larger than $\varepsilon$. Indeed, if this happened, 
there should exist a time $t$ such that $\|\widetilde{\bU}(t)-\ubU\|_{\HH_\Xi}=\varepsilon$, hence in particular 
$\bU(t)\in \UU_{\epsilon_1}\bigcap {\man}$, and therefore
$$\funct[\bU(t);\mu,\blambda,\speed]\geq m(\varepsilon)$$
whereas we know that
$$\funct[\bU(t);\mu,\blambda,\speed]<m(\varepsilon)\,.$$

This proves that $\bU(t)$ belongs to $\UU_\varepsilon$ whenever $\bU(0)\in \UU_{\eta}\bigcap {\man}$.
In addition, this $\eta$ --- as well as the $\epsilon_0$, $\epsilon_1$ invoked in the derivation of $\eta$ --- can be chosen to be the same for parameters $(\mu,\blambda,\speed)$ varying in a compact subset of $\Omega$ on which the period $\Xi$ remains constant, because $\funct$ is uniformly continuous on compact sets of $\HH_\Xi \times \Omega$, the profile $\ubU$ depends continuously on the parameters $(\mu,\blambda,\speed)$, as well as the lower bound $\widetilde{C}$, as already mentioned. This uniformity will be used in a crucial way in the final argument.

\paragraph{Step 2.} Find a way out of the constraint manifold.

For fixed $(\umu,\ublambda,\uspeed)\in \Omega$, let us denote by $\uubU$ the associated profile, of period $\Xi$,
$$\uUU_\epsilon\,=\,\{\bU\in \HH_\Xi\,;\;\inf_{s\in \R} \| \bU-\uubU(\cdot+s)\|_{\HH_\Xi}\,\leq\,\epsilon\}\,$$
for all $\epsilon>0$,
We still denote  by $\ubU$ profiles associated with `generic' parameters $(\mu,\blambda,\speed)\in\Omega$. There exists a neighbordhood ${\mathbb B}_\epsilon$ of $(\ublambda,\uspeed)$ in $\Lambda$ such that
for all $(\blambda,\speed)\in {\mathbb B}_\epsilon$ and $\mu=\mu(\blambda,\speed)$  --- prescribed  by the fixed period $\Xi$ ---, the corresponding profile $\ubU$ belongs to $\uUU_\epsilon$.

Next, we claim that for $(\bM,P)$ close enough to $(\int_{0}^\Xi \uubU\,\dif x,\int_{0}^\Xi \Impulse(\uubU)\,\dif x)$,
say
$$\textstyle\left\|\bM \,-\,\int_{0}^\Xi \uubU\,\dif x \right\|_{\R^N}\leq \delta(\epsilon)\,,\qquad\left|P \,-\,\int_{0}^\Xi \Impulse(\uubU)\,\dif x \right|\leq \delta(\epsilon)$$
for some positive $\delta(\epsilon)$, there exists $(\blambda,\speed)\in {\mathbb B}_\epsilon$
such that the traveling profile $\ubU$ associated with $(\mu(\blambda,\speed),\blambda,\speed)$ satisfies
$$\textstyle\int_{0}^\Xi \ubU\,\dif x\,=\,\bM\,,\qquad\int_{0}^\Xi \Impulse(\ubU)\,\dif x\,=\,P\,.$$
This follows from the inverse mapping theorem. As a matter of fact, the Jacobian matrix of the mapping
$$\textstyle(\blambda,\speed)\mapsto (\int_{0}^\Xi \ubU\,\dif x,\int_{0}^\Xi \Impulse(\ubU)\,\dif x)$$
turns out to be $-\bC$. This is precisely the reason why we have called $\bC$ the \emph{constraint} matrix. For more details, see \cite[Theorem~3]{BNR-GDR-AEDP} where we use in a crucial way relations in \eqref{eq:deraction},
$$\textstyle\Action_\mu=\Xi\,,\;\nabla_{\blambda} \Action = -\int_{0}^{\Xi} \ubU\,\dif x\,,\;\Action_\speed=\int_{0}^{\Xi} \Impulse(\ubU)\,\dif x\,.$$
Since by {\bf (H1)} $\bC$ is assumed to be nonsingular, the inverse mapping theorem does apply and this proves our claim.

\paragraph{Conclusion.} We now have all the ingredients to complete the proof of Theorem \ref{thm:orb}. Let us take $\varepsilon \in (0,\epsilon_1]$, where $\epsilon_1$ is introduced as in Step 1 for $(\umu,\ublambda,\uspeed)$ \emph{and its neighbors}. This is where we use uniformity: by Step 1, there exists $\eta\in (0,\varepsilon]$ so that, for all $(\blambda,\speed)\in {\mathbb B}_{\epsilon_1}$, if $\bU_0\in \W_\Xi$ is such that
\begin{equation}\label{eq:small+constraints}
\inf_{s\in \R} \|\bU_0-\ubU(\cdot+s)\|_{\HH_\Xi}\leq \eta\,,\;\textstyle\int_{0}^\Xi \bU_0\,\dif x\,=\,\int_{0}^\Xi \ubU\,\dif x\,,\,\int_{0}^\Xi \Impulse(\bU_0)\,\dif x\,=\,\int_{0}^\Xi \Impulse(\ubU)\,\dif x\,,
\end{equation}
with $\ubU$ being the traveling profile associated with $(\mu(\blambda,\speed),\blambda,\speed)$, then
$\bU(t)\in \UU_{\varepsilon/2}$ for all $t\in [0,T)$, where $T$ is the maximal time of existence of the solution $\bU(t)$ of \eqref{eq:absHamB}
such that $\bU(0)=\bU_0$.

By continuity of the mapping
$$\textstyle\bU \in \HH_\Xi\mapsto (\int_{0}^\Xi \bU\,\dif x,\int_{0}^\Xi \Impulse(\bU)\,\dif x)$$ 
there exists $\zeta \in (0,\eta/2]$ such that $\|\bU_0-\uubU\|_{\HH_\Xi}\leq \zeta$ implies
$$\textstyle\left\|\int_{0}^\Xi \bU_0\,\dif x \,-\,\int_{0}^\Xi \uubU\,\dif x \right\|_{\R^N}\leq \delta(\eta/2)\,,\;\left|\int_{0}^\Xi \Impulse(\bU_0)\,\dif x \,-\,\int_{0}^\Xi \Impulse(\uubU)\,\dif x \right|\leq \delta(\eta/2)\,,$$
where $\delta$ is the function involved in Step 2.
This implies by Step 2 the existence of $(\blambda,\speed)\in {\mathbb B}_{\eta/2}$, its associated profile being $\ubU\in \uUU_{\eta/2}$ such that
$$\textstyle\int_{0}^\Xi \ubU\,\dif x\,=\,\int_{0}^\Xi \bU_0\,\dif x\,,\qquad\int_{0}^\Xi \Impulse(\ubU)\,\dif x\,=\,\int_{0}^\Xi \Impulse(\bU_0)\,\dif x\,.$$

Therefore, if $\bU_0\in \W_\Xi$ is such that $\|\bU_0-\uubU\|_{\HH_\Xi}\leq \zeta$, there exists a profile $\ubU$ such that we have \eqref{eq:small+constraints} --- by using the triangle inequality to achieve the first condition. By Step 1 this implies that $\bU(t)\in \UU_{\varepsilon/2}$ for all $t\in [0,T)$, and thus $\bU(t)\in \uUU_{\varepsilon}$ by the triangle inequality again.
This proves the orbital stability of the traveling wave $(x,t)\mapsto \uubU(x-\speed t)$.
\end{proof}

Before concentrating separately on the cases $N=1$ and $N=2$, let us point out a general identity on the constraint matrix and the Hessian of the abbreviated action integral.

\begin{proposition}\label{prop:ActionC}
Let $\Action:\Omega\to \R$ be a ${\mathcal C}^2$ function of $(\mu,\blambda,\speed)\in \Omega$, an open subset of $\R\times\R^N\times\R$, such that 
the second order derivative $\Action_{\mu\mu}$ does not vanish, and define the continuous function $\bC:\Omega\to \R^{(N+1)\times(N+1)}$ by
$$\bC= \frac{\check\nabla \Action_{\mu}\otimes \check\nabla \Action_{\mu}}{\Action_{\mu\mu}}\,-\,
\check\nabla^2 \Action\,,$$
where $\check\nabla$ stands for the partial gradient with respect to all but the first independent variable $\mu$.
Then there exists a continuous mapping $\bP: \Omega\to {\sf SL}_{N+2}(\R)$ such that
$$ {\bP} \,(\Hess\Action)\, \transp{\bP}= \left( \begin{array}{c|c} \Action_{\mu\mu} & 0 \\ \hline
0 & -\bC\end{array}\right) \,.$$
In particular, we have
\begin{equation}
\label{eq:detActionC}
\det(\Hess\Action)\,=\,(-1)^{N+1}\,\Action_{\mu\mu}\,\det\bC\,.
\end{equation}
\end{proposition}
\begin{proof}
The matrix relation is a matter of elementary operations on rows and columns, which give the result with 
 $$\arraycolsep=1.5pt\def\arraystretch{1.5}
\bP\,=\,\left( \begin{array}{c|c}1 & 0\\  \hline
\frac{1}{\Action_{\mu\mu}}\check\nabla\Action_\mu & I_{N+1}\end{array}\right)\,.$$
Eq.~\eqref{eq:detActionC} is a straightforward consequence of that relation.
\end{proof}

As a consequence, we see that $\negsign(\Hess\Action)=\negsign(-\bC)$ if $\Action_{\mu\mu}>0$ and 
$\negsign(\Hess\Action)=\negsign(-\bC)+1$ if $\Action_{\mu\mu}<0$. This is the key to the proof of the identity in \eqref{eq:indicevssignAction}, together with a Sturm--Liouville argument as we explain in the next section.

\section{Examples}\label{s:ex}

\subsection{Quasilinear KdV}\label{ss:qKdV}

We use here the notational convention introduced before Theorem \ref{thm:EvansKdV}, namely, for a given periodic wave solution to (qKdV) of spatial period $\Upsilon$ and profile $\ubv$,
$$\action(\mu,\lambda,\speed)=\int_{0}^{\Upsilon} (\en(\ubv,\ubv_x)+\speed \imp(\ubv) - \lambda  \ubv +\mu)\,\dif x\,,$$
$$\linvar=\Hess(\en+\speed\imp)[\ubv]\,,\qquad\qquad \imp(\bv)=\tfrac12   \bv^2\,.$$
Accordingly, we denote by $\lag$ the Lagrangian such that $$\action(\mu,\lambda,\speed)
=\int_{0}^{\Upsilon} (\lag[\ubv;\lambda,\speed]+\mu)\,\dif x\,,$$
and whose second variational derivative is $\linvar$, that is
$$\lag[\bv;\lambda,\speed]:= \en(\bv,\bv_x)+ \tfrac12   \speed \bv^2 - \lambda\,\bv\,.$$

The assumption in {\bf (H0)} regarding the existence and parametrization of periodic wave profiles is easily met when the energy is of the same form as in \eqref{eq:en},
$$\en(\bv,\bv_x)=f(\bv)+\tfrac12 \cap(\bv) \bv_x^2\,,$$
with smooth functions $f$ and $\cap$ such that $\cap(\bv)>0$. Indeed, the profile equation $\Legendre \lag[\ubv;\lambda,\speed]=\mu$ then reads
$$\tfrac12 \cap(\ubv) \ubv_x^2-f(\ubv)-\tfrac12   \speed \ubv^2+ \lambda\,\ubv=\mu\,.$$
Remarkably enough, the associated phase portrait in the plane $\{(\bv,\dot\bv)\,;\,\dot\bv:=\bv_x\sqrt{\cap(\bv)}\}$ does not depend on $\cap$, and consists of the level sets 
$$\{(\ubv,\dot\ubv)\,;\tfrac12 \dot\ubv^2-f(\ubv)-\tfrac12   \speed \ubv^2+ \lambda\,\ubv=\mu\,\}\,.$$
When $\mu$ is varied, these level sets
exhibit saddle points $(\bv,0)$ where the potential $$\Potential(\bv;\lambda,\speed):=-f(\bv)-\tfrac12   \speed \bv^2+ \lambda\,\bv$$
achieves a local maximum, and center points $(\bv,0)$ where $\Potential(\cdot;\lambda,\speed))$ has a local minimum.
In other words,  since $f'=-\press$, $(\bv,0)$ is a saddle point if $\press(\bv)-\speed\bv +\lambda=0$, $\press'(\bv)<\speed$,
and a center point if $\press(\bv)-\speed\bv +\lambda=0$, $\press'(\bv)>\speed$. In particular, if $\press$ is convex, 
a family of periodic wave profiles is found as soon as we have a pair made of a saddle point $(\bv_s,0)$ and a center point $(\bv_0,0)$ such that there exists $\bv^s$ with $\bv_s<\bv_0<\bv^s$ and 
$$\Potential(\bv_s;\lambda,\speed)=\Potential(\bv^s;\lambda,\speed)\,,$$
which amounts to an equal area rule on the graph of $p$
(that is, the areas in between the graph of $p$ and that of the affine function $\bv\mapsto \speed \bv -\lambda$, 
for $\bv\in [\bv_s,\bv_0]$ and for $\bv\in [\bv_0,\bv^s]$ are equal; this can be observed on
Figures \ref{fig:portrait-KdV} and \ref{fig:portrait-SV} hereafter).

By the implicit function theorem, the roots $\bv$ of $\press(\bv)-\speed\bv +\lambda$ are smoothly parametrized by 
$(\lambda,\speed)$ as long as $\press'-\speed$ does not vanish, which means that saddle points and center points are smoothly parametrized by $(\lambda,\speed)$, and the zeroes of $\Potential(\cdot;\lambda,\speed)-\mu$ are smoothly parametrized by $(\mu,\lambda,\speed)$ away from critical points $\bv_s(\lambda,\speed)$ and $\bv_0(\lambda,\speed)$. By smoothness of the flow of ODEs, this shows that periodic wave orbits found inside homoclinic loops are smoothly parametrized by $(\mu,\lambda,\speed)$. 

Let us give a more precise situation in which {\bf (H0)} is satisfied. It is chosen in order to include many of the examples we have in mind, and in particular power laws $\press(\bv)=\bv^\gamma$, $\gamma>1$, and $\press(\bv)=\bv^{-\gamma}$, $\gamma>0$. (The more complicated van der Waals law has been investigated in earlier work, see \cite{BNR-JNLS14}.) 
Let us point out in passing that, as far as profile equations are concerned, the special case $\press(\bv)=\bv^2$, which corresponds to the `standard' KdV equation, and
$\press(\bv)=\bv^{-2}$, which corresponds to a shallow-water type of pressure law\footnote{See  \S\ref{ss:orbEK} for an explanation.}, are closely related. Indeed, the profile equation for $\press(\bv)=\bv^2$ readily amounts to a cubic potential $\Potential$,
while the profile equation for $\press(\bv)=\bv^{-2}$ also amounts to a cubic potential after multiplying it by $\ubv$ --- and modifying $\cap$ accordingly (one should however pay attention to the fact that this operation alters the status of the parameters $(\mu,\lambda,\speed)$, and in particular that of $\mu$, which becomes like a Lagrange multiplier instead of being an energy level). Even without this trick, the phase portraits look similar, see Figures \ref{fig:portrait-KdV} and \ref{fig:portrait-SV} to compare the two situations.

\newpage
\begin{figure}[H]
\begin{center}
\includegraphics[width=85mm]{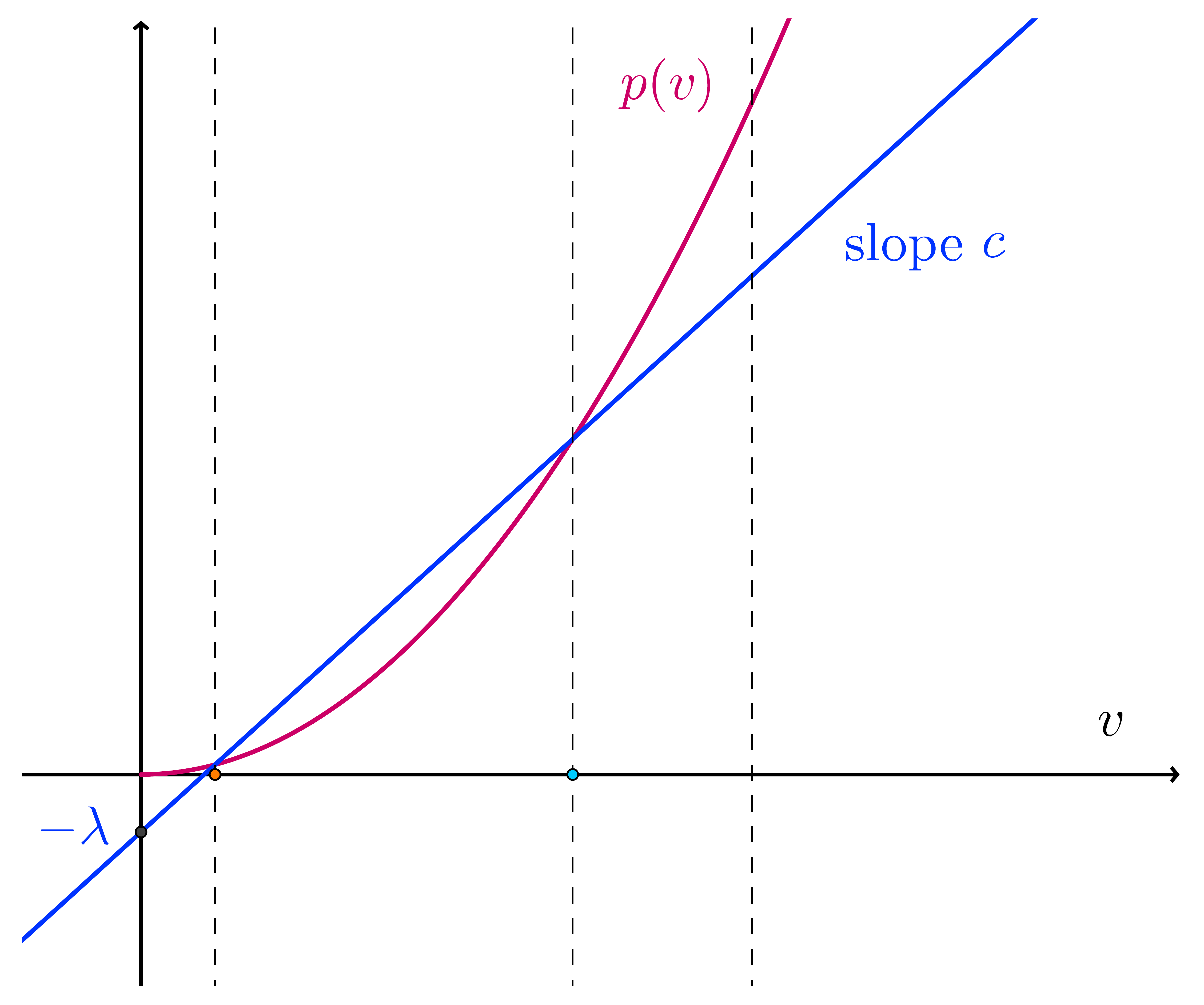} 

\vspace{2mm}
\includegraphics[width=85mm]{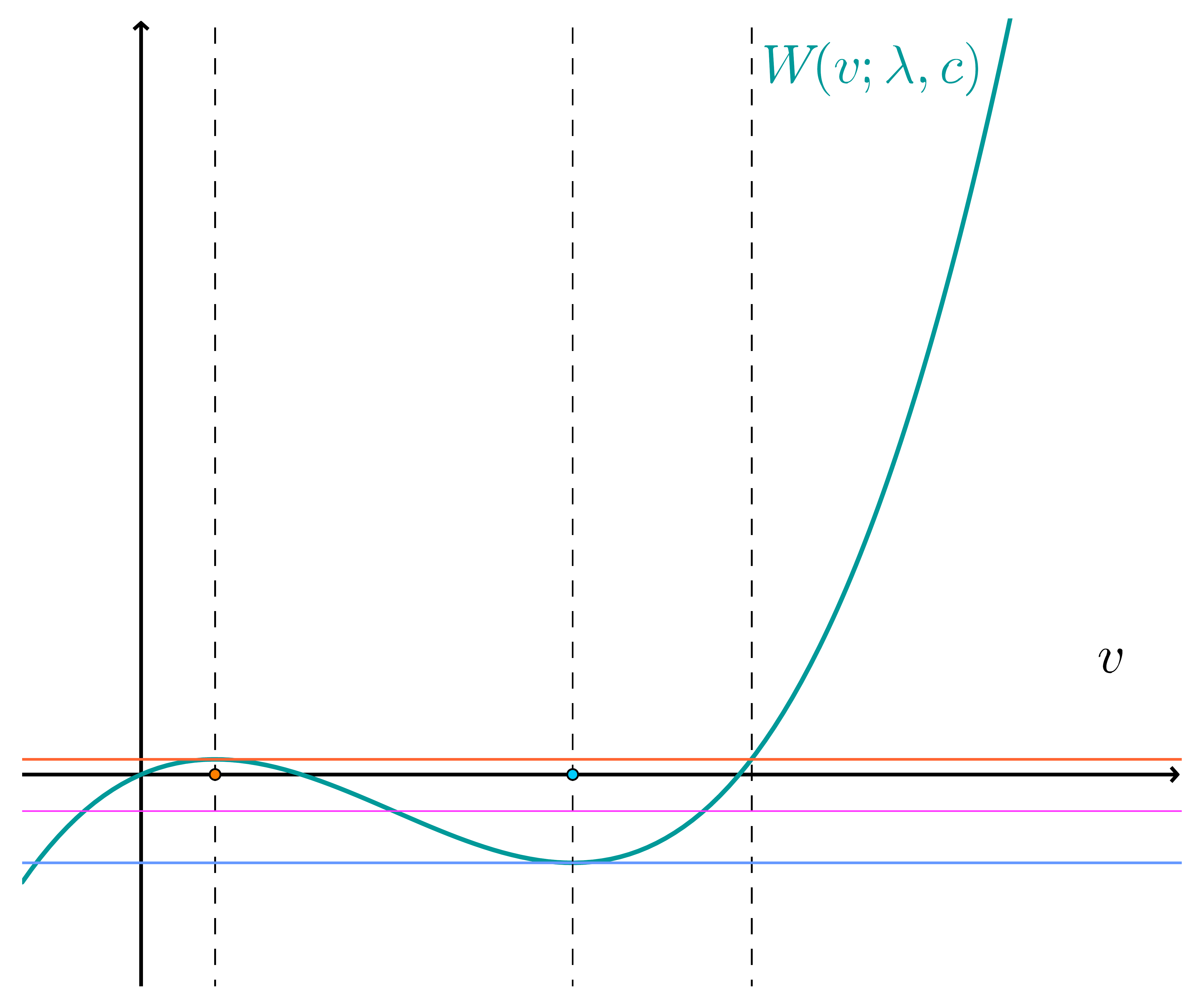}

\vspace{2mm}
\includegraphics[width=85mm]{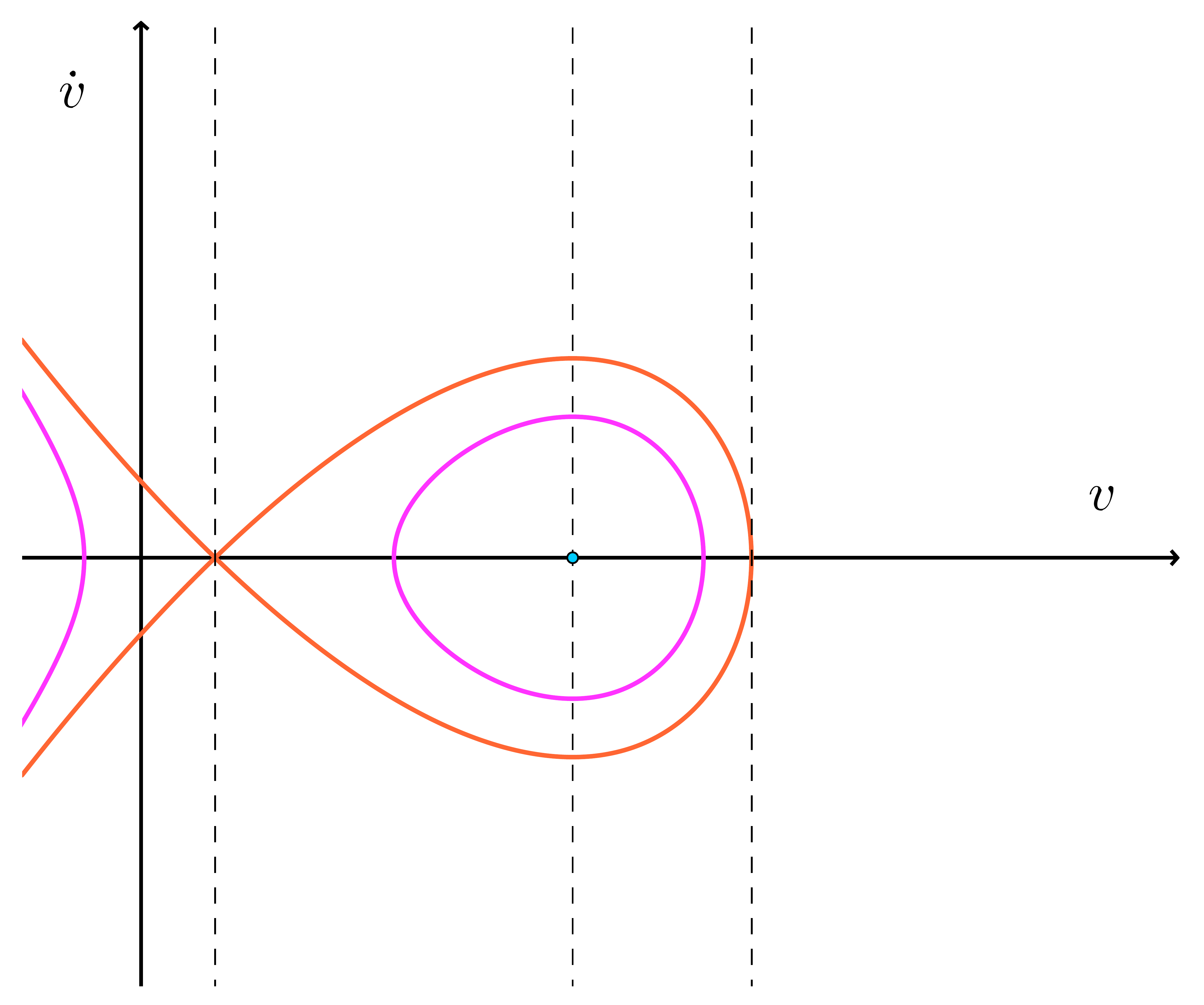}

\caption{Quadratic nonlinearity,  
associated potential and phase portrait.}\label{fig:portrait-KdV}
\end{center}
\end{figure}

\newpage
\begin{figure}[H]
\begin{center}
\includegraphics[width=85mm]{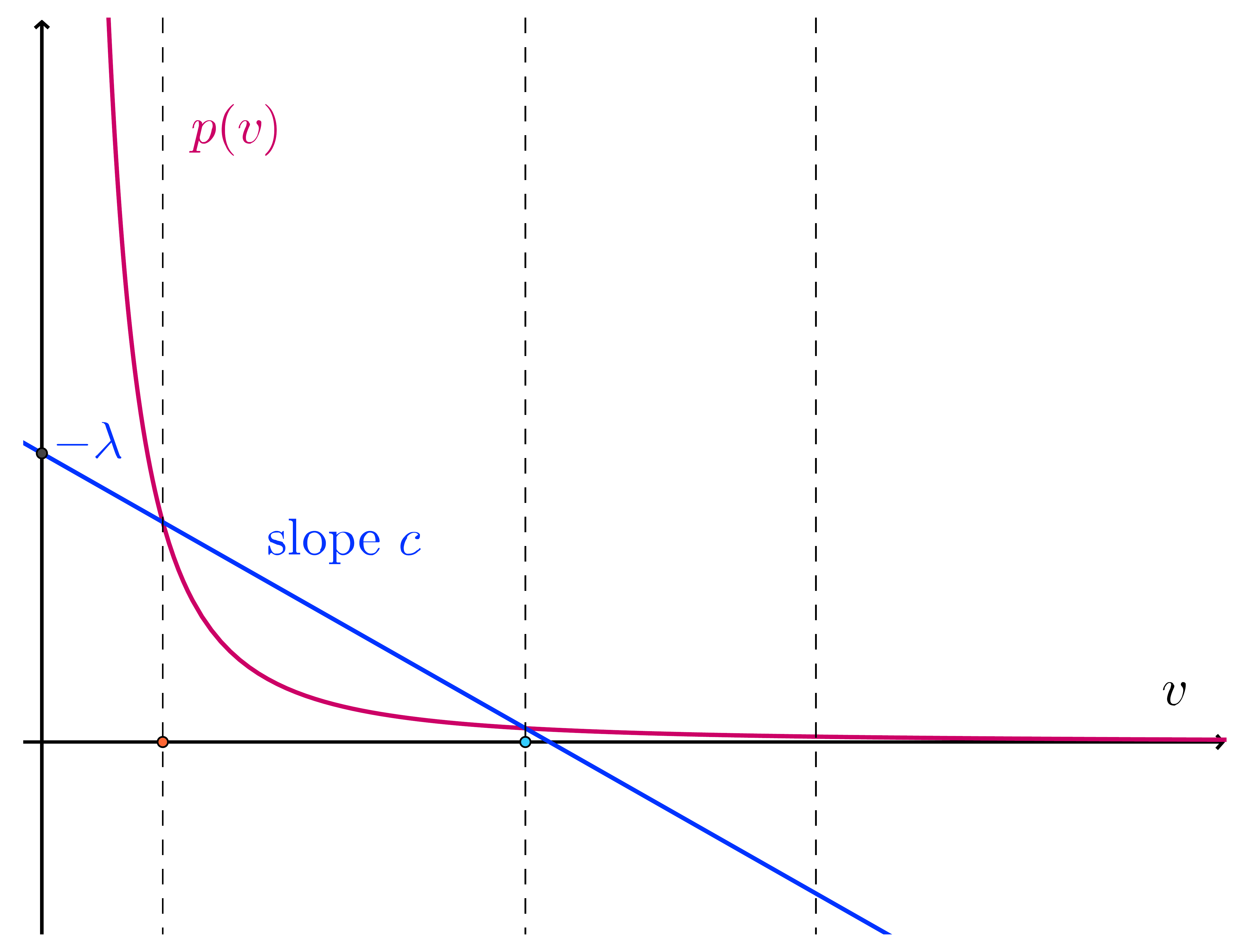}

\vspace{2mm}
\includegraphics[width=85mm]{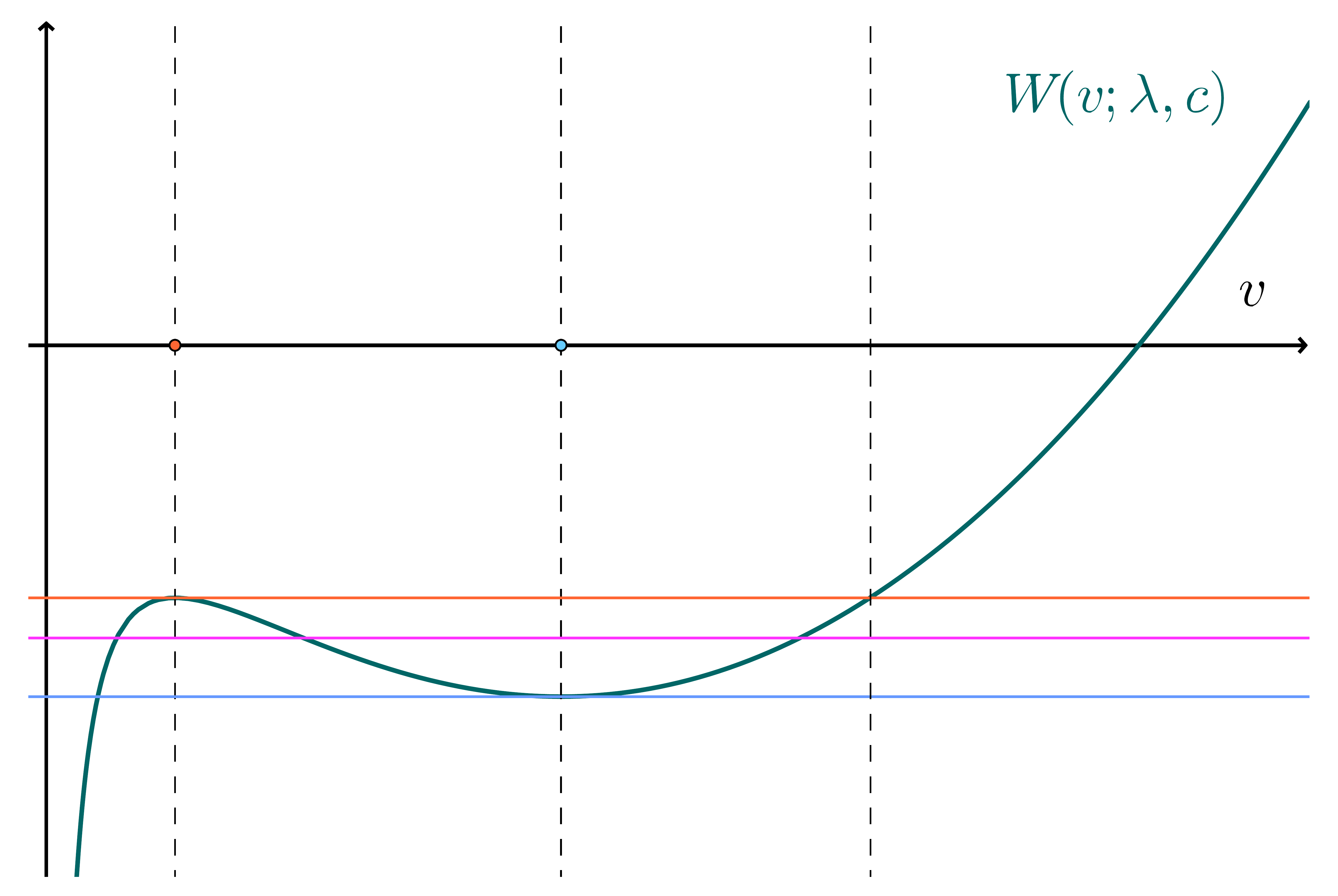}

\vspace{2mm}
\includegraphics[width=85mm]{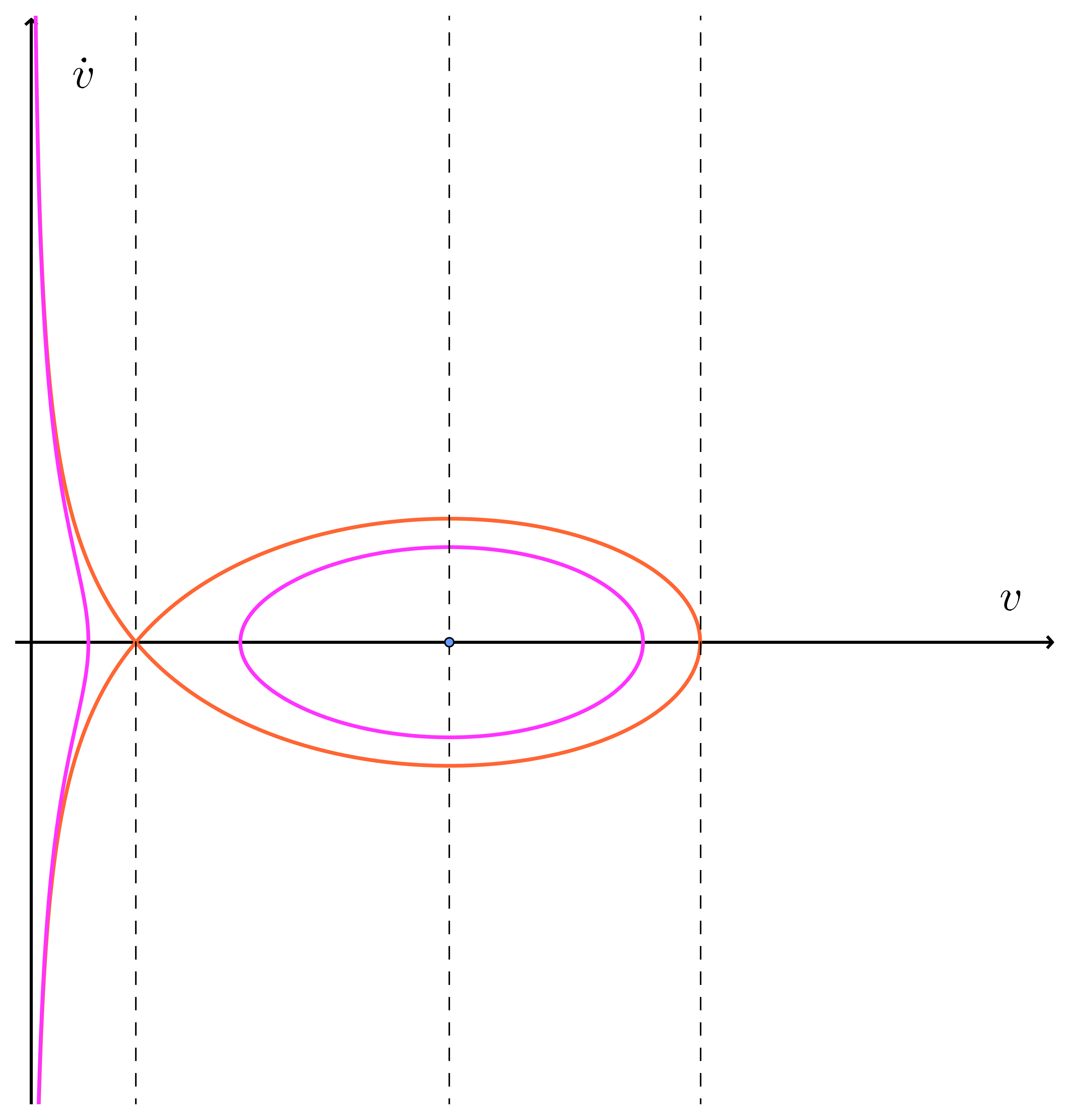}

\caption{Shallow-water type nonlinearity,  
associated potential and phase portrait.}\label{fig:portrait-SV}
\end{center}
\end{figure}

\begin{proposition}\label{prop:H0}
Assume that $\cap:I\to (0,+\infty)$ is ${\class}^2$, and that $\press=-f': I \to \R$ is ${\class}^2$ with $p''>0$ on the open interval $I\subset\R$.
We denote by $V:J:=\press'(I)\to I$ the inverse mapping of $\press'$, 
and 
$$l: \speed\in J\mapsto \speed V(\speed)-p(V(\speed))\,,$$
$$l_-: \speed\in J\mapsto \lim_{\bv \searrow \inf I} (\speed \bv-p(\bv))\,,$$
$$l_+: \speed\in J\mapsto \lim_{\bv \nearrow \sup I} (\speed \bv-p(\bv))\,,$$
and consider the open sets
$$\Gamma:=\{\speed\in J\,;\; 
l_-(\speed) <\speed\,,\;l_+(\speed) <l(\speed)\}\,,$$
$$\Lambda:=\{(\lambda,\speed)\in \R\times \Gamma\,;\; \max(l_-(\speed),l_+(\speed))<\lambda<l(\speed)\}\,.$$
Then there exist ${\class}^2$ mappings
$$\bv_s:\Lambda\to \{\bv\in I\,;\;\bv<V(\speed)\}\quad \mbox{and}\quad \bv_0:\Lambda\to \{\bv\in I\,;\;\bv>V(\speed)\}$$ such that for all $(\mu,\lambda,\speed)\in \Omega$,
$$\Omega:=\{(\mu,\lambda,\speed)\in \R\times \R\times \Gamma\,;\;
(\lambda,\speed)\in \tilde{\Lambda}\,,\;\mu \in (\Potential(\bv_0(\lambda,\speed);\lambda,\speed),\Potential(\bv_s(\lambda,\speed);\lambda,\speed))) \}\,,$$
$$\tilde{\Lambda}:=\{(\lambda,\speed)\in \Lambda\,;\;\lim_{\bv \nearrow \sup I} \Potential(\bv;\lambda,\speed) > \Potential(\bv_s(\lambda,\speed);\lambda,\speed) \}\,,$$
there is a unique solution $\ubv$ to
$$\tfrac12 \cap(\ubv) \ubv_x^2-f(\ubv)-\tfrac12   \speed \ubv^2+ \lambda\,\ubv=\mu\,,\; \ubv_x(0)=0\,,$$
that is periodic, and it is ${\class}^2$ in parameters $(\mu,\lambda,\speed)$.
\end{proposition}
The proof is based on the remarks made above and some elementary analysis of the variations of $\Potential(\cdot;\lambda,\speed)$ for $(\lambda,\speed)\in \tilde{\Lambda}$, see Table \ref{tb:var} below. Details are left to the reader.

\begin{table}[hb]
$$\begin{array}{c|cccccccccccccc}
\bv &  & & v_1 & & v_s & & v_2  &  & v_0 & & v_3 & & v^s & \\ \hline 
\Potential' & & & + & & 0 & & - & & 0 & & + & &   \\  \hline
& & &  & & \mu_s & & & & & &  & &   \mu_s & \\
 & &  & & \nearrow &  & \searrow  & & & & & & \nearrow &   \\ [-5pt]
\Potential & & & \mu & & & & \mu & & & & \mu & &   \\
 & &  \nearrow &  &  &  &  &  & \searrow  & &  \nearrow & &  &   \\
  & &   &  &  &  &  &   &  & \mu_0 & & &  &   \\
\end{array}$$
\caption{Variations of potential}\label{tb:var}
\end{table}

Proposition \ref{prop:H0} applies in particular to
\begin{itemize}
\item $\press(\bv)=\bv^\gamma$, $\gamma>1$, $I=(0,+\infty)$, $J=(-\infty,0)$, $V(\speed)=(\speed/\gamma)^{1/(\gamma-1)}$, $$l_-(\speed)=0\,,\;l_+(\speed)=-\infty\,,\;l(\speed)=(1/\gamma^{1/(\gamma-1)}-1/\gamma^{\gamma/(\gamma-1)})\speed^{\gamma/(\gamma-1)}\,,$$
\item $\press(\bv)=\bv^{-\gamma}$, $\gamma>0$, $I=(0,+\infty)$, $J=(0,+\infty)$, $V(\speed)=(-\gamma/\speed)^{1/(\gamma+1)}$, $$l_-(\speed)=-\infty\,,\;l_+(\speed)=-\infty\,,\;l(\speed)=(\gamma^{1/(\gamma+1)}-1/\gamma^{\gamma/(\gamma+1)})(-\speed)^{\gamma/(\gamma+1)}\,.$$
\end{itemize}

This is what we may say on {\bf (H0)}. Regarding {\bf (H1)}, it turns out to be equivalent to
$$\action_{\mu\mu}\neq 0\,,\qquad\det(\Hess\action)\,\neq\,0\,.$$
Indeed, $\action_{\mu\mu}\neq 0$ is exactly the first condition in {\bf (H1)}, and as soon as it is satisfied, the equivalence between $\det \bc\neq 0$ and $\det(\Hess\action)\neq 0$ readily follows from Proposition \ref{prop:ActionC}, in which Eq.~\eqref{eq:detActionC} reads, with our present notation (and $N=1$),
$$ \det (\Hess\action) = \action_{\mu\mu} \det \bc\,.$$
We may also derive this relation by writing explicitly 
$$\bc= -\frac{1}{\action_{\mu\mu}}\;\left(\begin{array}{cc} 
\left|\begin{array}{cc}
\action_{\lambda\lambda} & \action_{\mu\lambda} \\
\action_{\lambda\mu} & \action_{\mu\mu}
\end{array}\right| & 
\left|\begin{array}{cc}
\action_{\lambda \speed} & \action_{\mu \speed} \\
\action_{\lambda\mu} & \action_{\mu\mu}
\end{array}\right| 
\\ [15pt]
\left|\begin{array}{cc}
\action_{\speed \lambda} & \action_{\mu\lambda} \\
\action_{\speed \mu} & \action_{\mu\mu}
\end{array}\right| & 
\left|\begin{array}{cc}
\action_{\speed  \speed} & \action_{\mu \speed} \\
\action_{\speed \mu} & \action_{\mu\mu}
\end{array}\right|
\end{array}\right)\,,
$$
and by computing
$$\left|\begin{array}{cc} 
\left|\begin{array}{cc}
\action_{\lambda\lambda} & \action_{\mu\lambda} \\
\action_{\lambda\mu} & \action_{\mu\mu}
\end{array}\right| & 
\left|\begin{array}{cc}
\action_{\lambda \speed} & \action_{\mu \speed} \\
\action_{\lambda\mu} & \action_{\mu\mu}
\end{array}\right| 
\\ [15pt]
\left|\begin{array}{cc}
\action_{\speed \lambda} & \action_{\mu\lambda} \\
\action_{\speed \mu} & \action_{\mu\mu}
\end{array}\right| & 
\left|\begin{array}{cc}
\action_{\speed  \speed} & \action_{\mu \speed} \\
\action_{\speed \mu} & \action_{\mu\mu}
\end{array}\right|
\end{array}\right| \,=\,\action_{\mu\mu}\,\det(\Hess\action)\,.$$

The investigation of whether $\action_{\mu\mu}$, the derivative of the period with respect to the value of the ODE Hamiltonian vanishes is a classical topic in Hamiltonian dynamics, see for instance the recent paper \cite{GarijoVilladelprat} and references therein.
A criterion ensuring that this derivative is positive was given in particular by Chicone \cite{Chicone}. It merely reads  $(\Potential/(\Potential')^2)''>0$,
for a Hamiltonian of the form $\tfrac12  
\bv_x^2+\Potential(\bv)$. Despite its simple form, it is not easy to check analytically. We have chosen to rely on numerical experiments in Section \ref{s:num} to rule out the critical cases in which $\action_{\mu\mu}$ would be zero.
Regarding the zeroes of the determinant of $\Hess\action$, far from distinguished limits --- see \cite{BMR2} for a discussion of those ---, we are not aware of any general analytical result. This is also investigated numerically in Section \ref{s:num}.

Concerning assumptions in {\bf (H2)}, we claim --- recall that the period of waves is denoted by $\Upsilon$ instead of $\Xi$ here --- that the space $\HH_\Upsilon=H^1(\R/\Upsilon\Z)$ is a convenient choice as soon as 
$$\en(\bv,\bv_x)=f(\bv)+\tfrac12 \cap(\bv) \bv_x^2$$
with $f:I\to \R$ and $\cap:I\to (0,+\infty)$ of class ${\mathcal C}^2$ on an open interval $I$. Indeed, in this case,
the mapping 
$$E:\bv \mapsto \int_{0}^{\Upsilon} \en(\bv,\bv_x) \,\dif x$$
is  well-defined on the open subset of $H^1(\R/\Upsilon\Z)$ made of $\bv$ with values in $I$ --- thanks to the embedding $H^1(\R/\Upsilon\Z)\hookrightarrow {\mathcal C}^0_b$ --- and twice differentiable with
$$\dif E (\bv) \cdot h= \int_{0}^{\Upsilon} (f'(\bv) h +\tfrac12 \cap'(\bv) h \bv_x^2 + \cap(\bv)  \bv_x h_x)\dif x\,, \; \forall h\in H^1(\R/\Upsilon\Z)\,,$$
$$\dif^2 E (\bv) \cdot (h,k)= \int_{0}^{\Upsilon} (f''(\bv) h k +\tfrac12 \cap''(\bv) h k \bv_x^2 +  \cap'(\bv) (h k_x+k h_x) + \cap(\bv)  h_x k_x)\dif x\,, \; \forall h,k\in H^1(\R/\Upsilon\Z)\,,$$
for all $\bv\in H^1(\R/\Upsilon\Z)$ with image in the domain of definition of $f$ and $\cap$. If in addition $\bv\in {\mathcal C}^2_b$, as is the case for traveling profiles, we may integrate by parts in the formulas above, and recognize the variational derivatives of $\en$. As a matter of fact, if $\bv\in {\mathcal C}^2_b$ we find that
$$\dif E (\bv) \cdot h= \int_{0}^{\Upsilon} h\,\Euler\en[\bv] \dif x\,, \; \forall h\in H^1(\R/\Upsilon\Z)\,,$$
$$\dif^2 E (\bv) \cdot (h,k)= \langle \Hess\en[\bv] h, k\rangle\,, \; \forall h,k\in H^1(\R/\Upsilon\Z)\,,$$
where 
$$\Euler \en[\bv]:= \partial_\bv \en - \partial_x(\partial_{\bv_x}\en)= f'(\bv) +\tfrac12 \cap'(\bv) \bv_x^2 - \partial_x(\cap(\bv)  \bv_x)\in {\mathcal C}^0_b\,,$$
and
$$\Hess\en[\bv] h:= (\partial_{\bv}^2\en - \partial_x(\partial_{\bv\bv_x}\en)) \,h - \partial_x(h\,\partial_{\bv_x}^2\en)=
 (f''(\bv)   + \tfrac{1}{2} \cap''(\bv) \bv_x^2   - \partial_x(\cap'(\bv) \bv_{x}) )\,h-\partial_x (\cap(\bv) h_x) $$
belongs to $H^{-1}(\R/\Upsilon\Z) = (H^1(\R/\Upsilon\Z))'$ for all $h\in H^1(\R/\Upsilon\Z)$, and 
$\langle \cdot ,\cdot \rangle$ denotes the dual product. In particular, $\Hess\en[\bv]$ is a \emph{Sturm-Liouville} operator, of the form $-\partial_x K \partial_x + q$, with  $\Upsilon$-periodic coefficients,
$$K= \cap(\bv) \in {\mathcal C}^2_b\,,\qquad0<K_0\leq K\leq K_1\,,$$
$$q=f''(\bv)   + \tfrac{1}{2} \cap''(\bv) \bv_x^2   - \partial_x(\cap'(\bv) \bv_{x}) \in {\mathcal C}^0_b\,, \quad\|q\|_{L^\infty}\leq \alpha_0\,.$$
Therefore, for all $h\in H^1(\R/\Upsilon\Z)$, we have
$$\langle \Hess\en[\bv] h, h\rangle + \alpha \|h\|_{L^2}^2= \int_{0}^{\Upsilon} (K h_x^2 + (q+\alpha) h^2) \dif x\,,$$
and if we choose for instance $\alpha= \alpha_0+1$, we have
$$\min(K_0,1)  \int_{0}^{\Upsilon} (h_x^2 + h^2) \dif x \leq \langle \Hess\en[\bv] h, h\rangle + \alpha \|h\|_{L^2}^2 \leq 
\max(K_1,2\alpha)  \int_{0}^{\Upsilon} (h_x^2 + h^2) \dif x\,.$$
This shows the equivalence of norms requested in {\bf (H2)}.

Finally, the main assumption in {\bf (H3)} is satisfied at least in $H^s(\R/\Upsilon\Z)$ for $s>7/2$. Indeed,  
the following theorem is proved in a forthcoming paper \cite{Mietka}.
\begin{theorem}\label{thm:qKdV}
Let $s$ be an integer such that $s>7/2$. If $p=-f': I \to \R$  is ${\class}^{s+1}$ and $\cap:I \to (0,+\infty)$ is ${\class}^{s+2}$, then for all $\Upsilon>0$, 
$\bv_0\in H^s(\R/\Upsilon\Z)$, the image of $\bv_0$ being in $I$, there exists $T>0$ and a unique $\bv\in {\class}([0,T);H^s(\R/\Upsilon\Z))$ solution to {\rm (qKdV)} with $\en=f(v)+\frac12 \cap(v)v_x^2$, and $\bv_0\mapsto \bv$ is continuous.
\end{theorem}
In the special case when $\cap$ is constant, well-posedness is known to hold true with much lower regularity, for example $s > -1/2$ is sufficient for the classical KdV \cite{KPV96}.
However low-regularity results rely on dipsersive effects to control nonlinear terms and are thus strongly model-dependent.
This is notoriously still a field of intense research\footnote{Incidentally we point to the attention of the reader the attempt of keeping track of latest known results --- including local-well posedness proofs --- for various dispersive equations on {\tt http://wiki.math.toronto.edu/DispersiveWiki/}}.

\begin{theorem}\label{thm:qKdVorb}
Under the assumptions of Theorem \ref{thm:qKdV}, 
if ${\bf (H0)}$ is satisfied (in particular if the assumptions of --- and thus --- Proposition \ref{prop:H0} hold true), and if, for some $(\mu,\lambda,\speed)\in \Omega$, 
$${\bf (s)}\quad \action_{\mu\mu}\neq 0\,,\quad \det(\Hess\action)\neq 0\,,\quad \negsign(\Hess\action)\,=\,1\,,$$
is satisfied, then the periodic wave associated with $(\mu,\lambda,\speed)$ is conditionally, orbitally stable in $H^1(\R/\Upsilon\Z)$.
\end{theorem}
\begin{proof}
The assumptions {\bf (H2)}-{\bf (H3)} are implied by those of Theorem \ref{thm:qKdV}, 
while, as explained before,  
{\bf (H1)} is equivalent to
$$\action_{\mu\mu} \,\det(\Hess\action)\,\neq \,0\,.$$
So our main assumptions
{\bf (H0)}-{\bf (H1)}-{\bf (H2)}-{\bf (H3)} are met here.

Furthermore, as announced before, the assumption that the kernel of $\linvar$ is spanned by $\ubv_x$ is automatically satisfied. In fact, Lemma \ref{lem:Sturm} stated below applies to the present operator $\linvar$, by using the potential $\Potential=\Potential(\bv;\lambda,\speed)$ introduced at the beginning of this section. We thus infer at the same time that 
the kernel of $\linvar$ in $L^2(\R/\Upsilon\Z)$ is spanned by $\ubv_x$, and how to compute the signature of $\linvar$.
Therefore, to apply Theorem \ref{thm:orb} and thus complete the proof of Theorem \ref{thm:qKdVorb} it suffices to check 
that our assumptions imply $\negsign(\linvar)=\negsign(\bc)$.

To do so, we can actually bypass the computation of $\negsign(\bc)$ by using Proposition \ref{prop:ActionC}. With our current notations, this algebraic proposition shows indeed that 
$$\negsign(\Hess\action)=\negsign(-\bc)\;\mbox{ if }\,\action_{\mu\mu}>0\,,\quad
\negsign(\Hess\action)=\negsign(-\bc)\,+\,1\;\mbox{ if }\,\action_{\mu\mu}<0\,,$$
while we know from Lemma \ref{lem:Sturm} that
$$\negsign(\linvar)= 1\;\mbox{ if }\action_{\mu\mu}=\Upsilon_\mu>0\,,\quad 
\negsign(\linvar)= 2\;\mbox{ if }\action_{\mu\mu}=\Upsilon_\mu<0\,.$$
Therefore, regardless of the sign of $\action_{\mu\mu}$ --- as long as this number is nonzero ---, we have
$$\negsign(\Hess\action)=\negsign(-\bc)\,+\,\negsign(\linvar)\,-\,1\,.$$
Furthermore, since the $2\times 2$ matrix $\bc$ is nonsingular --- as already justified ---, 
$$\negsign(-\bc)=2-\negsign(\bc)\,,$$
so that the previous formula equivalently reads
$$\negsign(\Hess\action)\,=\,1\,+\,\negsign(\linvar)\,-\,\negsign(\bc)\,.$$
This is the announced identity in \eqref{eq:indicevssignAction} in the special case $N=1$.

As a consequence, the stability index $\negsign(\linvar)\,-\,\negsign(\bc)$ vanishes if and only if 
$$\negsign(\Hess\action)\,=\,1\,.$$
In this respect, the set of conditions in ${\bf (s)}$ are the optimal ones enabling us to apply Theorem \ref{thm:orb} to {\rm (qKdV)}.
\end{proof}

\begin{lemma}
\label{lem:Sturm} Assume that $\cap:I\to (0,+\infty)$ and
$\Potential:I\to \R$ are ${\class}^2$ on some open interval $I$ and such that the \emph{Euler--Lagrange} equation
$\Euler \enred  [\bv] =0$ associated with the energy
$$\enred: (\bv,\bv_x)\mapsto \tfrac{1}{2} \cap(\bv) \,\bv_x^2\,+\,\Potential (\bv)$$
admits a family of periodic solutions $\ubv$ taking values in $I$, parametrized by the energy level
$\mu \,=\,\Legendre \enred[\ubv]$, for $\mu\in J$, another open interval. If we denote by $\Upsilon$ the period of  $\ubv$, and assume that $\Upsilon_\mu$, its derivative with respect to $\mu$, does not vanish, then the self-adjoint differential operator $\linvar\,:=\,\Hess \enred[\ubv]$ has the following properties:
\begin{itemize}
\item the kernel of $\linvar$ on $L^2(\R/\Upsilon\Z)$ is the line spanned by $\ubv_x$;
\item the negative  signature $\negsign(\linvar)$ of $\linvar$ is given by the following rule:
\begin{itemize}
\item[$*$] if $\Upsilon_\mu>0$ then $\negsign(\linvar)=1$,
\item[$*$] if $\Upsilon_\mu<0$ then $\negsign(\linvar)=2$.
\end{itemize}
\end{itemize}
\end{lemma}

Since $$\linvar=- \partial_x \cap(\ubv) \partial_x + \Potential''(\ubv)   + \tfrac{1}{2} \cap''(\ubv) \ubv_x^2   - \partial_x(\cap'(\ubv) \ubv_{x})\,,$$
is a \emph{Sturm--Liouville} operator with periodic coefficients, the main argument in the proof of Lemma \ref{lem:Sturm} relies on Sturm's oscillation theorem, as for instance in Lemma 1 in \cite{BronskiJohnsonKapitula}, which concerns the case of a constant $\cap$. The detailed proof is postponed to Appendix \ref{app:Sturm}.

Note that since $\Hess\action$ is a $3\times 3$ matrix, its determinant cannot be positive if $\negsign(\Hess\action)=1$. So we could equivalently replace the second condition in {\bf (s)} by
$\det(\Hess\action)<0$.

If the first two conditions in {\bf (s)} are readily amenable to computations, the verification of the third one, $\negsign(\Hess\action)=1$ demands some slightly more sophisticated algebraic work. It can be interesting to have more explicit conditions, in particular to compare with earlier results. If we assume moreover that $2\times 2$ determinant $\action_{\lambda\lambda} \action_{\mu\mu}-\action_{\lambda\mu}^2$ does not vanish, we can see by an elementary count of sign changes in the principal minors of $\Hess \action$  that {\bf (s)} stems from having either one of the following sets of conditions
\begin{itemize}
\item[{\bf (s1)}]
$\action_{\mu\mu}>0\,,\quad \left|\begin{array}{cc}
\action_{\lambda\lambda} & \action_{\mu\lambda} \\
\action_{\lambda\mu} & \action_{\mu\mu} \end{array}\right|\neq 0\,,\quad \det(\Hess\action)<0\,,$
\item[{\bf (s2)}]
$\action_{\mu\mu}<0\,,\quad \left|\begin{array}{cc}
\action_{\lambda\lambda} & \action_{\mu\lambda} \\
\action_{\lambda\mu} & \action_{\mu\mu}
\end{array}\right| <0\,,\quad \det(\Hess\action)<0\,.$
\end{itemize}
For convenience, the reader may refer to  Table \ref{tb:signs} in Appendix \ref{app:signs}, in which 
{\bf (s1)} corresponds to the second and third rows, and {\bf (s2)} to the $7$th.
From the same table we see that, in terms of the constraint matrix $\bc$,
{\bf (s1)} corresponds to $\negsign(\bc)=1$, and {\bf (s2)} to $\negsign(\bc)=2$. 
Note however that our conditions in {\bf (s)} are slightly more general than the prescription of ({\bf (s1)} or {\bf (s2)}), in that they
do not require $\action_{\lambda\lambda} \action_{\mu\mu}-\action_{\lambda\mu}^2\neq 0$.

In the special case when $\cap$ is constant, Theorem \ref{thm:qKdVorb} was essentially already known in the case $\action_{\lambda\lambda} \action_{\mu\mu}-\action_{\lambda\mu}^2\neq 0$, and proved in a slightly different manner in \cite{BronskiJohnsonKapitula}. Indeed, orbital stability with respect to co-periodic perturbations is essentially a consequence of \cite[Theorem 1]{BronskiJohnsonKapitula}, under the assumption
$$\action_{\mu\mu}\neq 0\,,\; \left|\begin{array}{cc}
\action_{\lambda\lambda} & \action_{\mu\lambda} \\
\action_{\lambda\mu}& \action_{\mu\mu}\end{array}\right| \neq 0\,,
\quad \det(\Hess\action)\neq 0\,,
\quad \negsign(\Hess\action)=1\,,$$
which is  equivalent to ({\bf (s1)} or {\bf (s2)}).
An earlier, similar result was shown by Johnson \cite{Johnson}, under the more restrictive assumption
\begin{equation}
\label{eq:Johnson}
\action_{\mu\mu}>0\,,\quad\left|\begin{array}{cc}
\action_{\lambda\lambda} & \action_{\mu\lambda} \\
\action_{\lambda\mu} & \action_{\mu\mu}
\end{array}\right| <0\,,\quad \det(\Hess\action)<0\,.
\end{equation}
We should mention that these results by Bronski \emph{et al} dealing with (gKdV) --- and not its quasilinear version (qKdV) --- yield genuine orbital stability, and not only conditional orbital stability. This is because the Cauchy problem for (gKdV) is much better understood that for (qKdV) with a nonconstant $\cap$.

Remarkably enough, Table \ref{tb:signs} in Appendix \ref{app:signs} shows that the matrix $\Hess\action$ cannot be nonnegative. Indeed,
by our identity on negative signatures (see Eq.~\eqref{eq:negsign}), we must have $\negsign(\linvar)\geq \negsign(\bc)$, so that the --- purely algebraic --- situation in the first row of Table \ref{tb:signs} cannot occur.

\subsection{Euler--Korteweg}\label{ss:EK}
\subsubsection{Eulerian coordinates \emph{vs} mass Lagrangian coordinates}\label{ss:EKLEKE}
Before investigating how to apply Theorem \ref{thm:orb} to the Euler--Korteweg system,
let us come back to the equivalence between its formulation in Eulerian coordinates {\rm (EKE)}, and its formulation in mass Lagrangian coordinates (EKL). This equivalence works as long as we deal with states away from vacuum, and more precisely, with densities $\rho$ that are bounded and bounded by below by some positive constant. It is based on the fact that the continuity equation 
\begin{equation}
\label{eq:cont}
\partial_t\rho+\partial_x(\rho\vits)=0
\end{equation}
is equivalent --- for $x\in \R$, and $t$ in some interval too ---  to the existence of a function $Y=Y(x,t)$ such that $\partial_xY=\rho$ and $\partial_tY=-\rho\vits$. Denoting by $(y=Y(x,t),s=t)$ the new coordinates obtained this way,
and introducing 
$$\vol(y,s)=1/\rho(x,t)\,, \; \vitsL(y,s)=\vits(x,t)\,,$$
we see that {\rm (EKE)} and {\rm (EKL)} are equivalent provided that 
$\En(\rho,\rho_x)=\rho\en(\vol,\vol_y)$. (Here we focus on smooth solutions, but this equivalence is also known to hold true for weak solutions when $\En$ depends only on $\rho$, that is, for the usual Euler equations.)
The change of coordinates $(x,t,\rho,\vits)\mapsto (y,s,\vol,\vitsL)$ is clearly nonlinear, and also nonlocal since the new `independent variable' $y$ actually depends on the integration of the dependent variables $\rho$ and $\rho\vits$. Nevertheless, there is also an equivalence between travelling wave solutions of {\rm (EKE)} and (EKL).
As shown in \cite{SBG-DIE}, $(\rho,\vits)=(\urho,\uvits)(x-\sigma t)$ is a travelling wave solution to {\rm (EKE)} if and only if $(\vol,\vitsL)=(\uvol,\uvitsL)(y+j t)$ is a travelling wave solution to {\rm (EKE)}, along with
$$\urho(\uvits-\sigma)\equiv j\,,\quad \urho(x)=1/\uvol(\uY(x)))\,,\quad\uvits(x)=\uvitsL(\uY(x))\,,\quad\uY'(x)=\urho(x)\,,$$
or, equivalently,
$$\uvitsL-j \uvol\equiv \sigma\,,\quad \uvol(y)=1/\urho(\uX(y)))\,,\quad\uvitsL(y)=\uvits(\uX(y))\,,\quad\uX'(y)=\uvol(y)\,.$$
In particular, if $(\urho,\uvits)$ is $\Xi$-periodic then $(\uvol,\uvitsL)$ is $\Upsilon$-periodic with
$$\Upsilon=\int_{0}^{\Xi}\urho(x) \dif x\,,\qquad \Xi=\int_{0}^{\Upsilon}\uvol(y) \dif y\,.$$
Furthermore, we can see that the abbreviated action integral does not depend on the chosen formulation. Before checking this, let us perform preliminary work on profile equations.
The system {\rm (EKE)} fits our abstract framework with $$\bU=\left(\begin{array}{c} \rho \\ \vits \end{array}\right)\,,\;\Ham[\bU]= \En(\rho,{\rho}_x) + \tfrac{1}{2} \rho\vits^2\,,\;\Impulse(\bU)=-\rho\vits\,,$$
so that the profile equations in \eqref{eq:EL}-\eqref{eq:ELham} can be written as
\begin{equation}\label{eq:profEKE}
\left\{\begin{array}{l}(\partial_\rho\En)(\urho,\urhox) - ((\partial_{\rho_x}\En)(\urho,\urhox))_x + \tfrac{1}{2} \uvits^2 -\sigma\uvits \,=\,-\lambda\,,\\
\urho\uvits-\sigma \urho\,=\,j\,,\\
\urhox\,(\partial_{\rho_x}\En)(\urho,\urhox)\,-\,\En(\urho,\urhox)\,-\,\tfrac{1}{2} \urho\uvits^2\,+\,\sigma\,\urho\,\uvits\,-\,\lambda\,\urho\,+\,j\,\uvits\,=\,\mu\,.
\end{array}\right.
\end{equation}
The notational choice for the Lagrange multiplier in the right-hand side of the second equation here above is dictated by the relation we already have in mind, and we denote by $-\lambda$ the first Lagrange multiplier for convenience, because it is going to be identified with the energy level in the EKL profile equations.  As a matter of fact,
it results from Theorem 1 in \cite{SBG-DIE} that the profile equations in Eulerian coordinates, as written above, are equivalent to
\begin{equation}\label{eq:profEKL}
\left\{\begin{array}{l}(\partial_\vol\en)(\uvol,\uvoly) - ((\partial_{\vol_x}\en)(\uvol,\uvoly))_y- j \uvitsL \,=\,-\mu\,,\\
\uvitsL- j  \uvol\,=\,\sigma\,,\\
\uvoly\,(\partial_{\vol_y}\en)(\uvol,\uvoly)\,-\,\en(\uvol,\uvoly)\,-\,\tfrac{1}{2} \uvitsL^2\,+\,j\,\uvol\,\uvitsL\,-\,\mu\,\uvol\,+\,\sigma\,\uvitsL\,=\,\lambda\,.
\end{array}\right.
\end{equation}
We recognize here the profile equations for (EKL), which fits our abstract framework with $$\bU=\left(\begin{array}{c} \vol \\ \vitsL \end{array}\right)\,,\qquad\Ham[\bU]= \en(\vol,\vol_y) + \tfrac{1}{2} \vitsL^2\,,\qquad\Impulse(\bU)=\vol\vitsL\,.$$
So there is an almost perfect symmetry in these profile equations, where we can see that 
$j$ and $\sigma$ exchange their roles, as well as $\lambda$ and $\mu$, when we go from {\rm (EKE)} to {\rm (EKL)} or \emph{vice versa}. This is why, even though this might seem confusing at first glance, we avoid introducing any additional piece of notation.
The fact that we have a single abbreviated action integral for both {\rm (EKE)} and {\rm (EKL)} is now clear because 
the change of variables is such that $\dif y=\urho \dif x$, or equivalently $\dif x=\uvol \dif y$, so that
\begin{equation}\label{eq:singleaction}
\Action(\mu,\lambda,j,\sigma)\begin{array}[t]{l}\,:=\,\displaystyle\int_{0}^{\Xi} 
(\En(\urho,\urhox)+\tfrac{1}{2} \urho\uvits^2 - \sigma \urho\uvits + \lambda \urho - j \uvits + \mu)\,\dif x\\ [8pt]
\,=\,\displaystyle\int_{0}^{\Upsilon} 
(\en(\uvol,\uvoly)+\tfrac{1}{2} \uvitsL^2 - j \uvol\uvitsL + \mu \uvol - \sigma \uvitsL + \lambda)\,\dif y\,.\end{array}
\end{equation}
The third of profile equations in \eqref{eq:profEKE} and \eqref{eq:profEKL} show that $\Action$ is indeed the abbreviated action for both of them.
In addition, we have the following correspondence between the `abstract parameters' $(\mu,\blambda,\speed)$, as introduced in profile equations \eqref{eq:EL}-\eqref{eq:ELham} and used in the abbreviated action \eqref{eq:Theta}, and the `practical parameters' $(\mu,\lambda,j,\sigma)$ for {\rm (EKE)} and (EKL).

\begin{table}[H]
\begin{center}\begin{tabular}{|c|c|c|c|c|} \hline
abstract & $\mu$ & $\lambda_1$ & $\lambda_2$ & $\speed$ \\ \hline
EKE & $\mu$ & $-\lambda$ & $j$ & $\sigma$ \\ \hline
EKL & $\lambda$ & $-\mu$ & $\sigma$  & $- j$ \\ \hline
\end{tabular}
\end{center}
\caption{Notation for wave parameters}
\label{tb:notations}
\end{table}

\begin{remark}
\label{rem:parameters}
\textup{
On the one hand, the occurrence of minus signs in Table \ref{tb:notations}
is not a real issue for applying our theory to {\rm (EKE)} and (EKL),  because both the quadratic form associated with $\Hess\Action$ and the one associated with the constraint matrix $\bC$ are invariant under the symmetry $\nu\mapsto -\nu$ in any dependent variable $\nu$ of $\Action$, and these quadratic forms are, together with the derivatives $
\Xi_\mu$ and $
\Upsilon_\lambda$, the only objects which govern our stability criteria.
On the other hand, the fact that the roles of parameters change when we go from {\rm (EKL)} to {\rm (EKE)} will have to be addressed carefully.}
\end{remark}

The invariance of spectral stability properties when going from Eulerian coordinates to mass Lagrangian coordinates, even though it seems very natural, is not obvious either as the original conjugacy occurs through the change of coordinates $(x,t,\rho,\vits)\mapsto (y,s,\vol,\vitsL)$ that is \emph{nonlinear} and \emph{nonlocal}.
Nonetheless, there is a kind of `conjugacy' between systems that are obtained by linearizing {\rm (EKE)} and (EKL),  in moving frames, about $(\urho,\uvits)$ and $(\uvol,\uvitsL)$ respectively. This in turn enables us to prove that the existence of an unstable mode in either one of these systems implies so for the other. In order to point out this conjugacy, we first reformulate {\rm (EKE)} and {\rm (EKL)} in moving frames associated with the waves. Let $\sigma$ and $j$ be fixed, and consider the new dependent variables
$$q:=\rho(\vits-\sigma)-j\,,\quad z:=\vitsL-j\vol-\sigma\,,$$
and the new independent variables
$$\xi:=x-\sigma t\,,\quad\zeta:=y+js\,.$$
Then the system {\rm (EKE)} is equivalent to
\begin{equation*}
\mbox{(EKE$j$)}\quad \left\{\begin{array}{l}\partial_t\rho +\partial_x q\,=\,0\,,\\ [5pt]
\partial_t q +\partial_x \Phi \,=\,0\,,\quad \Phi=\Phi(\rho,\rho_x,\rho_{xx},q;j):=(q+j)^2/\rho+\rho \Euler\En[\rho]+\Legendre\En[\rho]\,,
\end{array}\right.
\end{equation*}
where $\partial_t$ now denotes the partial derivative at constant $\xi$, and for convenience we have substituted again $x$ for $\xi$.
The second equation in {\rm (EKE$j$)} here above comes from the conservation law for the impulse, as in \eqref{eq:impulsecl}, which reduces here to
$$\partial_t(\rho\vits)+\partial_x(\rho \vits^2+ \rho \,\Euler\En[\rho]+  \Legendre \En[\rho])\,=\,0\,,$$
from which we have subtracted $\sigma$ times the continuity equation \eqref{eq:cont}.
Similarly, the system {\rm (EKL)} is found to be equivalent to
\begin{equation*}
\mbox{(EKL$j$)}\quad \left\{\begin{array}{l}\partial_s{\vol} \,=\,\partial_y z\,,\\ [5pt]
\partial_s z \,=\, \partial_y\Psi\,,\quad \Psi=\Psi(\vol,\vol_y,\vol_{yy},z;j):= \Euler \en [\vol] - 2 j z - j^2 \vol\,.
\end{array}\right.
\end{equation*}
Here above, $\partial_s$ stands for the partial derivative at constant $\zeta$, and for convenience we have substituted again $y$ for $\zeta$. Of course, as is the case for {\rm (EKE)} and (EKL), systems {\rm (EKE$j$)} and {\rm (EKL$j$)} are equivalent, as long as smooth solutions with positive and bounded densities $\rho$ and volumes $\vol$ are concerned and provided that 
$\En(\rho,\rho_x)=\rho\en(\vol,\vol_y)$. More precisely, the change of coordinates $(x,t,\rho,q)\mapsto (y,s,\vol,\vitsL)$ is given by
$$\dif y = \rho \dif x - q \dif t\,,\quad s=t\,,\quad\vol(y,s)=1/\rho(x,t)\,, \quad z(y,s)=q(x,t)/\rho(x,t)\,,$$
or equivalently by
$$\dif x = \vol \dif y + z \dif s\,,\quad t=s\,,\quad\rho(x,t)=1/\vol(y,s)\,, \quad q(x,t)=z(y,s)/\vol(y,s)\,.$$
By construction of {\rm (EKE$j$)} and {\rm (EKL$j$)}, the travelling wave solutions to {\rm (EKE)} and {\rm (EKL)} considered above become \emph{stationary} solutions  to {\rm (EKE$j$)} and {\rm (EKL$j$)}, which read $(\urho,0)$ and $(\uvol,0)$ respectively.
We can now show the following.

\begin{theorem}\label{thm:conj}
Assume that $\urho$ and $\uvol$ are smooth functions on $\R$, bounded and bounded by below by positive constants, related by
$$\urho(x)=1/\uvol(\uY(x)))=\uY'(x)\,,$$
and such that $(\urho,0)$ and  $(\uvol,0)$ are stationary solutions to {\rm (EKE$j$)} and {\rm (EKL$j$)} respectively, for some real number $j$. Then by linearizing {\rm (EKE$j$)} and {\rm (EKL$j$)} about $(\urho,0)$ and  $(\uvol,0)$ respectively, we receive systems whose spectra are identical.
\end{theorem}

\begin{proof}
Let us simply call {\rm ($\ell$E$j$)} and {\rm ($\ell$L$j$)} the linearized systems of {\rm (EKE$j$)} and {\rm (EKL$j$)} about $(\urho,0)$ and  $(\uvol,0)$ respectively, and denote them in abstract form as
$$\mbox{\rm ($\ell$E$j$)} \quad \partial_t \left(\begin{array}{c}\dot\rho \\ \dot{q} \end{array}\right)= \Linb_E \left(\begin{array}{c}\dot\rho \\ \dot{q} \end{array}\right)\,,$$
$$\mbox{\rm ($\ell$L$j$)} \quad \partial_s \left(\begin{array}{c}\dot\vol \\ \dot{z} \end{array}\right)= \Linb_L \left(\begin{array}{c}\dot\vol \\ \dot{z} \end{array}\right)\,.$$
Our aim is to show that 
 the differential operators $\Linb_E$ and $\Linb_L$ are  isospectral.
By translation invariance of  {\rm ($\ell$E$j$)} and {\rm ($\ell$L$j$)}, we already know that $\nu=0$ is an eigenvalue of both $\Linb_E$ and $\Linb_L$, associated eigenvectors being
$\transp{(\urhox,0)}$ and $\transp{(\uvoly,0)}$ respectively. 
From now on, we take a nonzero complex number $\nu$, and aim at showing that it belongs to the spectrum of  
$\Linb_E$ if and only if it belongs to the spectrum of $\Linb_L$. By spectrum of the operator $\Linb_E$, which is a differential operator in $x$ with $\Xi-$periodic coefficients, we mean the whole spectrum in the space of square integrable functions, which is known to be the collection of complex numbers $\nu$ such that there is a nontrivial $(\dot\rho,\dot{q})$ satisfying
\begin{equation}\label{eq:spe}
\Linb_E \left(\begin{array}{c}\dot\rho \\ \dot{q} \end{array}\right) = \nu \left(\begin{array}{c}\dot\rho \\ \dot{q} \end{array}\right)\,,\quad\exists \alpha\in \R\,,\;(\dot\rho,\dot{q})(x+\Xi) = \ee^{i \alpha} (\dot\rho,\dot{q})(x)\,,\;\forall x\in \R\,.
\end{equation}
Here above, $\alpha$ is called a Floquet exponent. Note that the spectrum of $\Linb_E$ in the space of square integrable $\Xi$-periodic functions corresponds to those $\nu$ for which $\alpha=0$. This is the case for $\nu=0$, since $(\urhox,0)$ is $\Xi$-periodic. Of course, the spectrum of $\Linb_L$ enjoys a similar characterization, which is the existence of a nontrivial $(\dot\vol,\dot{z})$ such that
\begin{equation}\label{eq:spl}
\Linb_L \left(\begin{array}{c}\dot\vol \\ \dot{z} \end{array}\right) = \nu \left(\begin{array}{c}\dot\vol\\ \dot{z} \end{array}\right)\,,\quad\exists \alpha\in \R\,,\;(\dot\vol,\dot{z})(y+\Upsilon) = \ee^{i \alpha} (\dot\vol,\dot{z})(y)\,,\;\forall y\in \R\,.
\end{equation}
The idea is to show a one-to-one correspondence between nontrivial $(\dot\rho,\dot{q})$ satisfying \eqref{eq:spe}
and nontrivial $(\dot{\vol},\dot{z})$ satisfying \eqref{eq:spl}, for the very same values of $(\nu,\alpha)$ with $\nu\neq 0$. This can be done by first returning to nonlinear systems. 

Solving  the original, nonlinear system {\rm (EKE$j$)} for some perturbation of $\urho$ that is parametrized by say $\epsilon$ as initial data, we receive a family of solutions $(\rho,q)=(\rho,q)(x,t;\epsilon)$ of {\rm (EKE$j$)} parametrized by $\epsilon$ such that
$$(\rho,q)(x,t;0)= (\urho(x),0)\,,$$
and that $(\dot\rho,\dot{q}):(x,t)\mapsto (\rho_\epsilon,q_\epsilon)(x,t;0)$ solves {\rm ($\ell$E$j$)}.
Furthermore, introducing $Y=Y(x,t;\epsilon)$ such that $Y(x,t;0)=\uY(x)$ and 
$\partial_xY=\rho$, $\partial_tY=-q$, we have 
\begin{equation}\label{eq:ELcoord} 
\rho(x,t;\epsilon)=1/\vol(Y(x,t;\epsilon),t;\epsilon)\,, \; q(x,t)=z(Y(x,t;\epsilon),t;\epsilon)/\vol(Y(x,t;\epsilon),t;\epsilon)
\end{equation}
where $(\vol,z)=(\vol,z)(y,s;\omega)$ is a family of solutions to  {\rm (EKL$j$)} parametrized by $\omega$. (We use here different notations for parameters $\epsilon$ and $\omega$ for the same reason as for times $t$ and $s$, that is, in order to avoid confusion about partial derivatives.) This implies that 
$$(\vol,z)(y,s;0)= (\uvol(y),0)\,,$$
and that $(\dot{\vol},\dot{q}):(y,s)\mapsto (\vol_\omega,q_\omega)(y,s;0)$ solves {\rm ($\ell$L$j$)}. 
Assume moreover that the family $(\rho,q)(\cdot,0;\epsilon)$ is chosen such that 
$(\dot{\rho}_0,\dot{q}_0)={(\rho_\epsilon,q_\epsilon)}(\cdot,0;0)$  
satisfies \eqref{eq:spe}.
Then
$$(\dot\rho,\dot{q})=\ee^{\nu t} (\dot{\rho}_0,\dot{q}_0)\,,$$
and we claim that, similarly,
$$(\dot\vol,\dot{z})=\ee^{\nu s} (\dot{\vol}_0,\dot{z}_0)\,,$$
with $(\dot{\vol}_0,\dot{z}_0)={(\vol_\omega,z_\omega)}(\cdot,0;0)$ satisfying \eqref{eq:spl}.

In order to prove that claim, let us first note that, by the chain rule applied to \eqref{eq:ELcoord},
\begin{equation}\label{eq:ELdiff}\dot\rho(x,t)=-(\dot\vol(\uY(x),t)+\dot{Y}(x,t)\, \uvoly(\uY(x)))/\uvol(\uY(x))^2\,,\dot{q}(x,t)=\dot{z}(\uY(x),t)/\uvol(\uY(x))\,,\end{equation}
where $\dot{Y}(x,t):=Y_\epsilon(x,t;0)$. Now, by differentiating $\partial_xY=\rho$, $\partial_tY=-q$, we get that
$\dot{Y}_x=\dot{\rho}$, $\dot{Y}_t=-\dot{q}$. By the first row in \mbox{\rm ($\ell$E$j$)} and the fact that $(\dot\rho,\dot{q})$ depends on $t$ as a linear function of $\ee^{\nu t}$, we have $\nu \dot{\rho} = - \dot{q}_x$, and therefore, we find that $\dot{Y}=-\dot{q}/\nu$. This relation is the key to the claimed conjugacy, because it enables us to rewrite \eqref{eq:ELdiff} as
\begin{equation}\label{eq:ELdiffr}\dot{\vol}(y,s)=-\uvol(y)^2 \dot{\rho}(\uX(y),s) + \dot{q}(\uX(y),s)\, \uvoly(y)/\nu\,,\;
\dot{z}(y,s)= \uvol(y) \dot{q}(\uX(y),s)\,,\end{equation}
with $\uX=\uY^{-1}$.
Equation \eqref{eq:ELdiffr} obviously implies that
$$\dot{\vol}(y,s)=\ee^{\nu s}\,(-\uvol(y) \dot{\rho}_0(\uX(y)) + \dot{q}_0(\uX(y))\, \uvoly(y)/\nu) = \ee^{\nu s} \dot{\vol}_0(y)\,,$$
$$\dot{z}(y,s)= \ee^{\nu s}( \uvol(y) \dot{q}_0(\uX(y)) = \ee^{\nu s} \dot{z}_0(y)\,.$$
In addition, we observe that
$\dot{\vol}$, $\dot{z}$ given by \eqref{eq:ELdiffr} in terms of $\uvol$, $\uX$, $\dot{\rho}$, $\dot{q}$, and $\nu$,
are bounded functions of $y$ if $\dot{\rho}$, $\dot{q}$ are bounded functions of $x$ --- because $\uvol$ and its derivatives are bounded. Furthermore,  $\dot{\vol}$, $\dot{z}$ are $\Upsilon$-periodic in $y$ if $\dot{\rho}$, $\dot{q}$ are $\Xi$-periodic in $x$ --- because $\uX(y+\Upsilon)-\uX(y)=\Xi$, and more generally, 
if $(\dot\rho,\dot{q})(x+\Xi,t) = \ee^{i \alpha} (\dot\rho,\dot{q})(x,t)$ for all $x,t$,
then 
$(\dot\vol,\dot{z})(y+\Upsilon,s) = \ee^{i \alpha} (\dot\vol,\dot{z})(y,s)$ for all $y,s$.

Similar computations can be performed in the other way round. Indeed, we can find $(\dot\vol,\dot{z})$ solving {\rm (EKL$j$)} from $(\dot\rho,\dot{q})$ solving {\rm (EKE$j$)} through the following formula, analogous to \eqref{eq:ELdiff},
\begin{equation}\label{eq:ELdiffi}\dot\vol(y,s)=-(\dot\rho(\uX(y),s)+\dot{X}(y,s)\, \urhox(\uX(y)))/\urho(\uX(y))^2\,,\dot{z}(y,s)=\dot{q}(\uX(y),s)/\urho(\uX(y))\,,\end{equation}
and the key relation is $\dot{X}=\dot{z}/\nu$. Using in addition that $\urho=1/\uvol$, $\urhox/\urho^2=-\uvoly/\uvol$, we thus see that, as expected, \eqref{eq:ELdiffi} is equivalent to \eqref{eq:ELdiffr}. 
\end{proof}

\begin{remark}\label{rem:conj}
\textup{
Theorem \ref{thm:conj} shows in particular that the operators just obtained by linearizing {\rm (EKE)} and {\rm (EKL)} in moving frames, but in the `original' dependent variables $(\rho,\vits)$ and $(\vol,\vitsL)$,
$$\Lin_E=-\Dif_x\left(\begin{array}{cc} j/\urho & \urho \\
\Hess \En[\urho] & j / \urho\end{array}\right)\quad \mbox{and}\quad \Lin_L=\Dif_y\left(\begin{array}{cc} -j & 1 \\
\Hess \en[\uvol] & -j\end{array}\right)\,,$$
are \emph{isospectral}.
Furthermore, we can infer from its proof the relationship between the eigenfunctions of $\Lin_E$ and $\Lin_L$ associated with nonzero eigenvalues.
Indeed, recalling that 
$$q:=\rho(\vits-\sigma)-j\,,\;z:=\vitsL-j\vol-\sigma\,,$$
we have 
$$\dot{q}= (j  / \urho)\dot{\rho} + \urho \dot{\vits}\,,\;\dot{z}=\dot{\vitsL}- j\dot{\vol}\,,$$
which yields, by substitution in \eqref{eq:ELdiffr},
$$\dot{\vol}=\uvol^2 (-1 + j \uL) \dot{\rho} +  \uL \dot{\vits}\,,\qquad
\dot{\vitsL}= j^2 \uvol^2\uL \dot{\rho} + (1+j\uL) \dot{\vits}\,,\qquad\uL:= \frac{\uvoly}{\nu \uvol}$$
(where we have omitted to write the independent variables for simplicity).
The practical `conjugacy' between $\Lin_E$ and $\Lin_L$ is thus far from being trivial.}
\end{remark}

Another natural question is the relationship between the Hessians of the constrained energies
$$\Linvar_E=\left(\begin{array}{cc} 
\Hess \En[\urho] & j / \urho\\
j/\urho & \urho \end{array}\right)\quad \mbox{and}\quad 
\Linvar_L=\left(\begin{array}{cc}
\Hess \en[\uvol] & -j \\
-j & 1 \end{array}\right)\,.$$
Interestingly, both of these matrix-valued operators are linked in a rather simple manner to \emph{scalar} operators.
Indeed, using notation as above,
$$\dot{q}= (j  / \urho)\dot{\rho} + \urho \dot{\vits}\,,\quad\dot{z}=\dot{\vitsL}- j\dot{\vol}\,,$$
we see that
$$\left(\begin{array}{c} 
\dot\rho \\
\dot\vits\end{array}\right)\cdot \Linvar_E \left(\begin{array}{c} 
\dot\rho \\
\dot\vits\end{array}\right)\,=\,
\dot\rho\,\linvar_E\dot\rho\,+\,\dot{q}^2/\urho\,,\quad\linvar_E:= \Hess \En[\urho] \,-\,j^2/\urho^3\,,$$
$$\left(\begin{array}{c} 
\dot\vol \\
\dot\vitsL\end{array}\right)\cdot \Linvar_L \left(\begin{array}{c} 
\dot\vol \\
\dot\vitsL\end{array}\right)\,=\,
\dot\vol\,\linvar_L\dot\vol\,+\,\dot{z}^2\,,\qquad\linvar_L:= \Hess \en[\uvol] \,-\,j^2\,.$$
These expressions have the striking consequence in terms of negative signatures that
\begin{equation}
\label{eq:signatures}
\negsign(\Linvar_E)=\negsign(\linvar_E)\,,\qquad\negsign(\Linvar_L)=\negsign(\linvar_L)\,.
\end{equation}
Both $\linvar_E$ and $\linvar_L$ are \emph{scalar} Sturm--Liouville operators with periodic coefficients. We thus have a simple criterion to compute their negative signature (see Lemma \ref{lem:Sturm} stated in Section \ref{ss:qKdV}).

Regarding the relationship between the scalar operators $\linvar_E$ and $\linvar_L$, one may check that 
$$\urho\,\linvar_E\dot{\rho} \,+\,\uCap\,(\urhox\,\dot\rho_x\,-\,\urhoxx \,\dot\rho)\,=\,-\,\linvar_L \dot\vol\,,\qquad  \uCap:=\partial^2_{\rho_x} \En[\urho]\,,$$
or equivalently
$$\uvol\,\linvar_L\dot{\vol} \,+\,\ucap\,(\uvoly\,\dot\vol_y\,-\,\uvolyy \,\dot\vol)\,=\,-\,\linvar_E \dot\rho\,,\qquad  \ucap:=\partial^2_{\vol_x} \en[\uvol]\,,$$
for
$$\dot{\vol}(\uY(x))=-\dot{\rho}(x)/\urho(x)^2\,,$$
which amounts to substituting $0$ 
for $\dot{X}$ in \eqref{eq:ELdiffi}. 
These relations are --- fortunately --- consistent with the fact that 
$\linvar_E\urhox$ and $\linvar_L\uvoly$ vanish simultaneously, 
 but they do not show a relationship between the signatures of $\linvar_E$ and $\linvar_L$.
 This is actually no surprise because we expect that it is the stability indices 
 $$\ind_E:=\negsign(\Linvar_E)-\negsign(\bC_E)=\negsign(\linvar_E)-\negsign(\bC_E)\quad \mbox{and}\quad \ind_L:=\negsign(\Linvar_L)-\negsign(\bC_L)=\negsign(\linvar_L)-\negsign(\bC_L)$$ which vanish simultaneously (see Remark \ref{rem:stabind} here after for a variational point of view on this question), and the negative signatures of the constraint matrices $\bC_E$ and $\bC_L$ have no a priori reason to coincide. Indeed, recalling from \eqref{eq:singleaction} the definition of $\Action(\mu,\lambda,j,\sigma)$,
 \begin{comment}
\begin{equation}
\label{def:ActionEK}
\Action(\mu,\lambda,j,\sigma)\begin{array}[t]{l}\,=\,\displaystyle\int_{0}^{\Xi} 
(\En(\urho,\urhox)+\tfrac{1}{2} \urho\uvits^2 - \sigma \urho\uvits + \lambda \urho - j \uvits + \mu)\,\dif x\\ [8pt]
\,=\,\displaystyle\int_{0}^{\Upsilon} 
(\en(\uvol,\uvoly)+\tfrac{1}{2} \uvitsL^2 - j \uvol\uvitsL + \mu \uvol - \sigma \uvitsL + \lambda)\,\dif y\end{array}\,,
\end{equation}
\end{comment}
and interpreting the abstract definition of the constraint matrix $\bC$ in {\bf (H1)} with the present notation (see Remark \ref{rem:parameters}), we find that
\begin{equation}\label{eq:CE}
\bC_E=-\frac{1}{\Action_{\mu\mu}}  
\left(\begin{array}{ccc} 
\left|\begin{array}{cc}
\Action_{\lambda\lambda} & \Action_{\mu\lambda} \\
\Action_{\lambda\mu} & \Action_{\mu\mu}
\end{array}\right| & 
\left|\begin{array}{cc}
\Action_{\lambda j} & \Action_{\mu j} \\
\Action_{\lambda\mu} & \Action_{\mu\mu}
\end{array}\right|
& \left|\begin{array}{cc}
\Action_{\lambda \sigma} & \Action_{\mu \sigma} \\
\Action_{\lambda\mu} & \Action_{\mu\mu}
\end{array}\right| \\ [15pt]
\left|\begin{array}{cc}
\Action_{j \lambda} & \Action_{\mu\lambda} \\
\Action_{j \mu} & \Action_{\mu\mu}
\end{array}\right| & 
\left|\begin{array}{cc}
\Action_{j  j} & \Action_{\mu j} \\
\Action_{j \mu} & \Action_{\mu\mu}
\end{array}\right|
& \left|\begin{array}{cc}
\Action_{j  \sigma} & \Action_{\mu \sigma} \\
\Action_{j \mu} & \Action_{\mu\mu}
\end{array}\right| \\ [15pt]
\left|\begin{array}{cc}
\Action_{\sigma \lambda} & \Action_{\mu\lambda} \\
\Action_{\sigma \mu} & \Action_{\mu\mu}
\end{array}\right| & 
\left|\begin{array}{cc}
\Action_{\sigma  j} & \Action_{\mu j} \\
\Action_{\sigma \mu} & \Action_{\mu\mu}
\end{array}\right|
& \left|\begin{array}{cc}
\Action_{\sigma  \sigma} & \Action_{\mu \sigma} \\
\Action_{\sigma \mu} & \Action_{\mu\mu}
\end{array}\right|
\end{array}\right)
\,,\end{equation}
\begin{equation}\label{eq:CL}\bC_L=- \frac{1}{\Action_{\lambda\lambda}}
\left(\begin{array}{ccc} 
\left|\begin{array}{cc}
\Action_{\mu\mu} & \Action_{\lambda\mu} \\
\Action_{\mu\lambda} & \Action_{\lambda\lambda}
\end{array}\right| & 
\left|\begin{array}{cc}
\Action_{\mu\sigma} & \Action_{\lambda\sigma} \\
\Action_{\mu\lambda} & \Action_{\lambda\lambda}
\end{array}\right|
& \left|\begin{array}{cc}
\Action_{\mu j} & \Action_{\lambda j} \\
\Action_{\mu\lambda} & \Action_{\lambda\lambda}
\end{array}\right| \\ [15pt]
\left|\begin{array}{cc}
\Action_{\sigma\mu} & \Action_{\lambda\mu} \\
\Action_{\sigma\lambda} & \Action_{\lambda\lambda}
\end{array}\right| & 
\left|\begin{array}{cc}
\Action_{\sigma\sigma} & \Action_{\lambda\sigma} \\
\Action_{\sigma\lambda} & \Action_{\lambda\lambda}
\end{array}\right|
& \left|\begin{array}{cc}
\Action_{\sigma  j} & \Action_{\lambda j} \\
\Action_{\sigma\lambda} & \Action_{\lambda\lambda}
\end{array}\right| \\ [15pt]
\left|\begin{array}{cc}
\Action_{j \mu} & \Action_{\lambda\mu} \\
\Action_{j \lambda} & \Action_{\lambda\lambda}
\end{array}\right| & 
\left|\begin{array}{cc}
\Action_{j \sigma} & \Action_{\lambda\sigma} \\
\Action_{j \lambda} & \Action_{\lambda\lambda}
\end{array}\right|
& \left|\begin{array}{cc}
\Action_{j j} & \Action_{\lambda j} \\
\Action_{j \lambda} & \Action_{\lambda\lambda}
\end{array}\right|
\end{array}\right)\,.
\end{equation}
Obviously, the matrices $\bC_E$ and $\bC_L$ do not involve  exactly the same minors of $\Hess\Action$, and may therefore have different negative signatures --- according to our algebraic computations in Appendix \ref{app:signs}, it might happen that $\negsign(\bC_L)=2$ and $\negsign(\bC_E)=0$.

\begin{remark}\label{rem:stabind}
\textup{
From \eqref{eq:singleaction} we see that 
$$\Action(\mu,\lambda,j,\sigma)\,=\, \funct_E[\urho,\uvits;\mu,\lambda,j,\sigma]\,=\,\funct_L[\uvol,\uvitsL;\mu,\lambda,j,\sigma]$$ with
$$\funct_E[\rho,\vits;\mu,\lambda,j,\sigma]\,:=\, \int_{0}^{\Xi} (\En(\rho,\rho_x)+\tfrac{1}{2} \rho\vits^2 - \sigma \rho\vits + \lambda \rho - j \vits + \mu)\,\dif x\,,$$
$$\funct_L[\vol,\vitsL;\mu,\lambda,j,\sigma]\,=\,\int_{0}^{\Upsilon} 
(\en(\vol,\vol_y)+\tfrac{1}{2} \vitsL^2 - j \vol\vitsL + \mu \vol - \sigma \vitsL + \lambda)\,\dif y\,.$$
The nullity of the stability index $\ind_E$ is a necessary condition for $\funct_E[\cdot;\mu,\lambda,j,\sigma]$ to have a local minimum at $(\urho,\uvits)$ under the constraints 
\begin{equation}\label{eq:Econstraints}
\begin{array}{l}
\textstyle\int_{0}^\Xi \dif x = \Xi\,,\;\int_{0}^\Xi \rho(x) \dif x = \int_{0}^\Xi \urho(x) \dif x\,,\\ [5pt]
\int_{0}^\Xi \vits(x) \dif x = \int_{0}^\Xi \uvits(x) \dif x\,,\;\int_{0}^\Xi (\rho\vits)(x) \dif x = \int_{0}^\Xi (\urho\uvits)(x) \dif x\,.
\end{array}
\end{equation}
Similarly, the nullity of the stability index $\ind_L$ is a necessary condition for $\funct_L[\cdot;\mu,\lambda,j,\sigma]$ to have a local minimum at $(\uvol,\uvitsL)$ under the constraints 
\begin{equation}\label{eq:Lconstraints}
\begin{array}{l}\textstyle\int_{0}^\Upsilon \dif y = \Upsilon\,,\;\int_{0}^\Upsilon \vol(y) \dif y = \int_{0}^\Upsilon \uvol(y) \dif y\,,\\ [5pt]
\int_{0}^\Upsilon \vitsL(y) \dif y = \int_{0}^\Upsilon \uvitsL(y) \dif y\,,\;\int_{0}^\Upsilon (\vol\vitsL)(y) \dif y = \int_{0}^\Upsilon (\uvol\,\uvitsL)(y) \dif y\,.\end{array}
\end{equation}
We have inserted the apparently trivial, first condition in \eqref{eq:Econstraints} and \eqref{eq:Lconstraints} in order to point out that these sets of constraints are actually equivalent under the change of variables
$$\rho(x)=1/\vol(Y(x)))\,,\;\vits(x)=\vitsL(Y(x))\,,\;Y'(x)=\rho(x)\,,$$
which happens to also ensure that 
$$\funct_L[\vol,\vitsL;\mu,\lambda,j,\sigma]\,=\,\funct_E[\rho,\vits;\mu,\lambda,j,\sigma]\,.$$
This is why we may expect $\ind_E$ and $\ind_L$  to vanish simultaneously.  
}
\end{remark}

What we can actually prove is the equality of the indices $\ind_E$ and $\ind_L$ under some `generic assumptions'.

\begin{theorem}\label{thm:indices}
Assume that $(\urho,\uvits)$ and $(\uvol,\uvitsL)$ are periodic solutions to, respectively, \eqref{eq:profEKE} and \eqref{eq:profEKL}, depending smoothly on the parameters $(\mu,\lambda,j,\sigma)$, as well as their periods $\Xi$ and $\Upsilon$. With notation introduced earlier, if $\Action_{\mu\mu}\neq0$, $\Action_{\lambda\lambda}\neq 0$, and the matrices $\bC_E$ and $\bC_L$ are nonsingular then 
we have the identity
$$\negsign(\Hess\Action)-2\,=\,\negsign(\Linvar_E)-\negsign(\bC_E)\,=\,\negsign(\Linvar_L)-\negsign(\bC_L)\,.$$ 
\end{theorem}

\begin{proof} As already noticed, we have 
$$\ind_E:=\negsign(\Linvar_E)-\negsign(\bC_E)=\negsign(\linvar_E)-\negsign(\bC_E)\quad \mbox{and}\quad \ind_L:=\negsign(\Linvar_L)-\negsign(\bC_L)=\negsign(\linvar_L)-\negsign(\bC_L)\,.$$
Furthermore, by Proposition \ref{prop:ActionC}, we have
$$\negsign(\Hess\Action)=\negsign(-\bC_E)\;\mbox{if }\,\Action_{\mu\mu}>0\,,\quad
\negsign(\Hess\Action)=\negsign(-\bC_E)\,+\,1\;\mbox{if }\,\Action_{\mu\mu}<0\,,$$
$$\negsign(\Hess\Action)=\negsign(-\bC_L)\;\mbox{if }\,\Action_{\lambda\lambda}>0\,,\quad
\negsign(\Hess\Action)=\negsign(-\bC_L)\,+\,1\;\mbox{if }\,\Action_{\lambda\lambda}<0\,.$$
We claim that, by Lemma \ref{lem:Sturm}, these relations imply
$$\negsign(\Hess\Action)=\negsign(-\bC_E)\,+\,\negsign(\linvar_E)\,-\,1\,=\,\negsign(-\bC_L)\,+\,\negsign(\linvar_L)\,-\,1\,, $$
from which we arrive at our final formula --- similarly as in the proof of Theorem \ref{thm:qKdVorb} --- by observing that for the 
$3\times 3$, noninsingular matrices $\bC_E$ and $\bC_L$ we have
$$\negsign(-\bC_E)=3-\negsign(\bC_E)\,,\quad \negsign(-\bC_L)=3-\negsign(\bC_L)\,.$$

It just remains to check that Lemma \ref{lem:Sturm} does imply that
\begin{equation}
\label{eq:signsL}
\negsign(\linvar_L)\,=1\;\mbox{if } \Action_{\lambda\lambda}>0\,,\quad\negsign(\linvar_L)\,=2\;\mbox{if } \Action_{\lambda\lambda}<0\,,
\end{equation}
\begin{equation}
\label{eq:signsE}
\negsign(\linvar_E)\,=1\;\mbox{if } \Action_{\mu\mu}>0\,,\quad\negsign(\linvar_E)\,=2\;\mbox{if } \Action_{\mu\mu}<0\,.
\end{equation}
This a matter of adapting notation, and checking the relationship between the energy levels for the vector-valued profile equations and for the reduced, scalar profile equations.

Eliminating $\uvitsL$ from the first equation in \eqref{eq:profEKL}, 
we can view it as an Euler--Lagrange equation for the energy
$$\enred(\vol,\vol_y):= \en(\vol,\vol_y) \,-\, \tfrac12 j^2\vol^2\,-\,(-\mu+j \sigma)\,\vol\,,$$
and eliminating $\uvitsL$ from the third equation in \eqref{eq:profEKL}, we find that the corresponding energy level is
$$\Legendre\enred[\uvol]=\lambda-\tfrac12 \sigma^2\,.$$
Therefore $\Action_{\lambda\lambda}=\Upsilon_\lambda$ is indeed the partial derivative of the period of $\uvol$ with respect to energy level, and since $\linvar_L=\Hess \enred[\uvol]=\Hess\en[\uvol]-j^2$,
Lemma  \ref{lem:Sturm} applies and shows \eqref{eq:signsL}. 

We proceed similarly to prove \eqref{eq:signsE}. 
By eliminating $\uvits$ from the first equation in \eqref{eq:profEKE} we find
$$\Euler \En[\urho] + \frac{j^2}{2{\urho}^2} \,=\,-\lambda+\tfrac12 \sigma^2\,,$$
which is the Euler--Lagrange equation associated with the energy
$$\enred(\rho,\rho_x):= \En(\rho,\rho_x) \,-\, j^2/(2\rho)\,-\,(-\lambda+\tfrac12 \sigma^2)\,\rho\,,$$
and by eliminating $\uvits$ from the third equation in \eqref{eq:profEKE}, we see that the corresponding energy level is
$$\Legendre\enred[\urho]=\mu-j\sigma\,.$$
Therefore $\Action_{\mu\mu}=\Xi_\mu$ is the partial derivative of the period of $\urho$ with respect to energy level, and since $\linvar_E\,=\,\Hess\En[\urho]-j^2/\urho^3$, Lemma  \ref{lem:Sturm} applies and shows \eqref{eq:signsE}. 

\end{proof}

\subsubsection{Connexion with quasilinear KdV equations}\label{ss:conn-qKdV}
As already mentioned in the proof of Theorem \ref{thm:indices},
the profile equations 
in \eqref{eq:profEKL} for traveling wave solutions $(\vol,\vitsL)=(\uvol,\uvitsL)(y+ j t)$ to {\rm (EKL)} reduce, after elimination of $\uvitsL$ to
$$\Euler \en [\uvol]\,-\,j^2\,\uvol\,+\,\mu\,-\,j\,\sigma\,=\,0\,,$$
which is nothing but the governing ODE for the profile of  traveling wave solutions $\vol=\uvol(x+ j^2 t)$ to (qKdV), and the 
Euler--Lagrange equation associated with
$$\lag[\vol;j\sigma-\mu,-j^2]= \en(\vol,\vol_y)- \tfrac12   j^2\vol^2 + (\mu - j\sigma)\,\vol\,,$$
the Lagrangian defined in Section \ref{ss:qKdV}.
Similarly, the velocity $\uvitsL$ can be eliminated from the abbreviated action integral
$$\Action(\mu,\lambda,j,\sigma)\,=\,\displaystyle\int_{0}^{\Upsilon} 
(\en(\uvol,\uvoly)+\tfrac{1}{2} \uvitsL^2 - j \uvol\,\uvitsL + \mu \uvol - \sigma \uvitsL + \lambda)\,\dif y$$
for (EKL). This yields
\begin{equation}\label{eq:actionsEKLqKdV}
\Action(\mu,\lambda,j,\sigma)\,=\,\displaystyle\int_{0}^{\Upsilon} 
(\en(\uvol,\uvoly)- \tfrac12   j^2\uvol^2 + (\mu - j\sigma)\,\uvol + \lambda-\tfrac12 \sigma^2)\,\dif y =\action(\lambda-\tfrac12 \sigma^2,j\sigma-\mu,-j^2)\,,
\end{equation}
where $\action$ is the abbreviated action integral associated with the $\Upsilon$-periodic traveling wave solutions $\vol=\uvol(x+ j^2 t)$ to (qKdV), as in Section \ref{ss:qKdV}. Therefore, by the chain rule, $\Hess\Action(\mu,\lambda,j,\sigma)$ can be expressed in terms of 
$\nabla \action$ and $\Hess\action$ evaluated at $(\lambda-\tfrac12 \sigma^2,j\sigma-\mu,-j^2)$. 
This will be used in \S \ref{ss:orbEK} below to compare the stability criteria for {\rm (EKL)} and (qKdV).

\subsubsection{Orbital stability in the Euler--Korteweg system}\label{ss:orbEK}
Let us now examine in which situation we may apply Theorem \ref{thm:orb} to the Euler--Korteweg system.
We consider energies as in \eqref{eq:en}, that is
$$\En = F(\rho)+\frac12 \Cap(\rho) \rho_x^2\,,$$
or equivalently,
$$\en = f(\vol)+\frac12 \cap(\vol) \vol_y^2\,,$$
with 
$$F(\rho)=\rho f(1/\rho), \qquad \Cap(\rho)= \rho^{-5} \cap(1/\rho)\,,$$
with $f:(0,+\infty)\to \R$ and $\cap:(0,+\infty)\to (0,+\infty)$ of class ${\mathcal C}^2$, hence also
$F:(0,+\infty)\to \R$ and $\Cap:(0,+\infty)\to (0,+\infty)$ are of class ${\mathcal C}^2$.
Even though many others are possible, we give now the most classical kind of nonlinearities that we may think of :
\begin{itemize}
\item \emph{Shallow-water} pressure law: $F(\rho)=\frac12 \rho^2$, or equivalently $f(\vol)=1/(2\vol)$, which gives $p(\vol)=1/(2\vol^2)$; then the  EKE system  is a dispersive modification of the Saint-Venant equations for shallow water flows, in which the `pressure' term is indeed known to be of the form $p(\rho)=\frac12 \rho^2$;
\item NLS \emph{capillarity}: $\Cap(\rho)=1/(4\rho)$;  then the  EKE system is the fluid formulation\footnote{Via the \emph{Madelung} transform: $\psi=\sqrt{\rho}\,\ee^{i\varphi}$, $\varphi_x=\vits$.} of the Non Linear Schr\"odinger equation 
$$\mbox{(NLS)} \quad i  \partial_t \psi + \tfrac{1}{2} \partial_x^2\psi= \psi \chem(|\psi|^2)\,,$$
with $g(\rho)=F'(\rho)$. In the shallow-water case, $g(\rho)=\rho$, which corresponds to the \emph{cubic} NLS.
\item \emph{Semilinear} EKE: $\Cap$ constant.
\item \emph{Semilinear} EKL: $\cap$ constant.
\end{itemize}

We warn here again the reader that, as explained in  \S\ref{ss:EKLEKE}, our notation from the abstract setting \ref{ss:abs} have to be adapted to parameters $(\mu,\lambda,j,\sigma)$ used to define  the action in \eqref{eq:singleaction} for both  {\rm (EKL)} and {\rm (EKE)} (see in particular Table \ref{tb:notations}).

By the connexion made in \S\ref{ss:conn-qKdV} between periodic waves in {\rm (EKL)} and (qKdV), and the one-to-one correspondence between periodic waves in {\rm (EKL)} and {\rm (EKE)} recalled in  \S\ref{ss:EKLEKE}, we have a good knowledge of situations in which {\bf (H0)} is satisfied for both {\rm (EKL)} and {\rm (EKE)}. Proposition \ref{prop:H0} exemplifies one such situation. More precisely, under the assumptions of that proposition, {\bf (H0)} holds true for both {\rm (EKL)} and {\rm (EKE)} 
for $(\mu,\lambda,j,\sigma)$ in the preimage by 
$$(\mu,\lambda,j,\sigma)\mapsto (\lambda-\tfrac12 \sigma^2,j\sigma-\mu,-j^2)$$ 
of the open set $\Omega$ defined in Proposition \ref{prop:H0}.
Other possibilities, with various families of periodic waves, are considered in \cite{BNR-JNLS14}. 

Taking {\bf (H0)} for granted, let us look at our other main assumptions {\bf (H1)}-{\bf (H2)}-{\bf (H3)}.

The abstract, nondegeneracy assumptions in {\bf (H1)} become, for {\rm (EKE)}
$$\Action_{\mu\mu}\neq 0\,,\quad \det \bC_E\neq 0\,,$$
and for 
{\rm (EKL)}
$$\Action_{\lambda\lambda}\neq 0\,,\quad \det \bC_L\neq 0\,.$$
Both of them can be reformulated in more convenient way by using Proposition \ref{prop:ActionC}, in which Eq.~\eqref{eq:detActionC} applied to {\rm (EKE)} and {\rm (EKL)} gives
\begin{equation}\label{eq:constEKE-EKL}
\Action_{\mu\mu}\, \det \bC_E\,=\,\Action_{\lambda\lambda}\,\det \bC_L\,=\,\,-\,\det (\Hess\Action)\,.
\end{equation}
(This identity may also be checked from the expressions of $\bC_E$ and $\bC_L$ given in \eqref{eq:CE} and \eqref{eq:CL} respectively.)
In particular, Proposition \ref{prop:ActionC} implies that, if $\Action_{\mu\mu}\,\Action_{\lambda\lambda}\,=\,\Xi_\mu\,\Upsilon_\lambda$ is nonzero, the matrices $\bC_L$ and $\bC_E$ are simultaneously nonsingular, and this happens when $\Hess\Action$ itself is nonsingular. In other words, {\bf (H1)} amounts to
$$\Action_{\mu\mu}\neq 0\,,\quad \det(\Hess\Action)\neq 0$$
for {\rm (EKE)}, and
$$\Action_{\lambda\lambda}\neq 0\,,\quad \det(\Hess\Action)\neq 0$$
for {\rm (EKL)}.

Regarding {\bf (H2)}, we claim that $\HH_\Upsilon:=H^1(\R/\Upsilon\Z)\times L^2(\R/\Upsilon\Z)$ and 
$\HH_\Xi:=H^1(\R/\Xi\Z)\times L^2(\R/\Xi\Z)$ do the job, respectively for {\rm (EKL)} and {\rm (EKE)}. For the former, this comes from a straightforward addition to what is done in  \S\ref{ss:qKdV}, since the Hamiltonian functional is just
$$(\vol,\vitsL)\mapsto \int_{0}^\Upsilon \en(\vol,\vol_y)\,\dif y \,+\,\tfrac{1}{2}\,\int_{0}^\Upsilon \vitsL^2\,\dif y\,=\,
E[\vol]\,+\,\tfrac12 \,\|\vitsL\|_{L^2}^2\,.$$
For the latter, the Hamiltonian functional is
$$(\rho,\vits)\mapsto \int_{0}^\Xi \Ham(\rho,\vits,\rho_x)\,\dif x \,=\,\int_{0}^\Xi \En(\rho,\rho_x)\,\dif x \,+\,\tfrac{1}{2}\,\int_{0}^\Xi\,\rho \vits^2\,\dif x\,.$$
The first part of this functional can be handled exactly as before, 
since $\en$ and $\En$ have the same abstract form. By this, we mean that the mapping
$\rho\mapsto \int_{0}^\Xi \En(\rho,\rho_x)\,\dif x$ is ${\class}^2$ on $H^1(\R/\Xi\Z)$, 
and for $\Xi$-periodic $\urho \in {\class}^2_b$ with values in $(0,+\infty)$,
$$\rho\mapsto (\langle \Hess\En[\urho] \rho, \rho\rangle\,+\,\alpha\,\|\rho\|^2_{L^2})^{1/2}$$ defines an equivalent norm on $H^1(\R/\Xi\Z)$, provided that $\alpha$ is large enough.
Furthermore,
$$\left\langle \Hess\Ham[\urho,\uvits] \left(\begin{array}{c}\rho\\ \vits\end{array}\right), \left(\begin{array}{c}\rho\\ \vits\end{array}\right)\right\rangle\,=\langle \Hess\En[\urho] \rho, \rho\rangle\,+\,\int_{0}^\Xi\,(\urho\, \vits^2+2\,\uvits\, \rho\,\vits)\,\dif x\,,$$
and 
$$\int_{0}^\Xi\,(\urho\, \vits^2+2\,\uvits\, \rho\,\vits)\,\dif x\,+\,\alpha\,\|\rho\|^2_{L^2}\,=\,
\int_{0}^\Xi \left(\begin{array}{c}\rho\\ \vits\end{array}\right)\cdot {\underline A} \left(\begin{array}{c}\rho\\ \vits\end{array}\right)\,\dif x\,,
$$
where
$${\underline A}:= \left(\begin{array}{cc}\alpha & \ubu \\ \ubu & \urho \vits\end{array}\right)$$
is bounded, and positive for $\alpha$ large enough, its upper and lower bound being uniform when the image of $(\urho,\ubu)$ is in a compact set of $(0,+\infty)\times\R$. This proves that {\bf (H2)} holds true.

Finally, in view of earlier work \cite{BDD1d,BDDmultiD} on the Cauchy problem for {\rm (EKL)} and {\rm (EKE)} in one space dimension, we  
claim that {\bf (H3)} is satisfied in $H^{s+1}\times H^s$ for $s>3/2$. 

\begin{theorem}\label{thm:EKL}
Let $s>3/2$. If $f: (0,+\infty) \to \R$ and $\cap:(0,+\infty) \to (0,+\infty)$ are ${\class}^{s+2}$, then, for all $\Upsilon>0$, 
$(\vol_0,\vitsL_0)\in H^{s+1}(\R/\Upsilon\Z)\times H^{s}(\R/\Upsilon\Z)$, $\vol_0>0$, there exists $T>0$ and a unique $(\vol,\vitsL)\in {\class}([0,T);H^{s+1}(\R/\Upsilon\Z)\times H^{s}(\R/\Upsilon\Z))$ solution to {\rm (EKL)} with $\en=f(\vol)+\frac12 \cap(\vol)\vol_y^2$, and 
$(\vol_0,\vitsL_0)\mapsto (\vol,\vitsL)$ is continuous.
\end{theorem}
\begin{theorem}\label{thm:EKE}
Let $s>3/2$. If $F: (0,+\infty) \to \R$ and $\Cap:(0,+\infty) \to (0,+\infty)$ are ${\class}^{s+2}$, then, for all $\Xi>0$,
$(\rho_0,\vits_0)\in H^{s+1}(\R/\Xi\Z)\times H^{s}(\R/\Xi\Z)$, $\rho_0>0$, there exists $T>0$ and a unique $(\rho,\vits)\in {\class}([0,T);H^{s+1}(\R/\Xi\Z)\times H^{s}(\R/\Xi\Z))$ solution to {\rm (EKE)} with $\En=F(\rho)+\frac12 \Cap(\rho)\rho_x^2$, and 
$(\rho_0,\vits_0)\mapsto (\rho,\vits)$ is continuous.
\end{theorem}

We thus have all the ingredients to apply Theorem \ref{thm:orb} to {\rm (EKL)} and {\rm (EKE)}.

\begin{theorem}\label{thm:EKLorb}
Under the assumptions of Theorem \ref{thm:EKL}, 
if ${\bf (H0)}$ is satisfied, then for a wave profile $(\uvol,\uvitsL)$ at which we have
$${\bf (S_L)}\quad \Action_{\lambda\lambda}\neq 0\,,\quad \det(\Hess\Action)\neq 0\,,\quad \negsign(\Hess\Action)\,=\,2\,,$$
the associated  periodic wave solution to {\rm (EKL)} is conditionally, orbitally stable in $H^1(\R/\Upsilon\Z)\times L^2(\R/\Upsilon\Z)$.
\end{theorem}

\begin{proof} Our assumptions include {\bf (H0)}, imply {\bf (H3)} by Theorem \ref{thm:EKL}, and are 
designed to also meet {\bf (H1)} and {\bf (H2)}  in the way explained earlier. 
Furthermore, the kernel of $\Linvar_L$ is spanned by $\transp{(\uvoly,\uvitsL_y)}$ because the kernel of $\linvar_L$ is spanned by $\uvoly$, which 
follows from the first statement in Lemma \ref{lem:Sturm}, applied as in the proof of Theorem \ref{thm:indices}.
Moreover, 
by Theorem \ref{thm:indices} we have
$$\negsign(\Linvar_L)-\negsign(\bC_L) \,=\,\negsign(\Hess\Action)-2\,,$$ 
which equals zero by assumption.
This means that Theorem \ref{thm:orb} does apply here.
\end{proof}

\begin{theorem}\label{thm:EKEorb}
Under the assumptions of Theorem \ref{thm:EKE}, 
if ${\bf (H0)}$ is satisfied, then for a wave profile $(\urho,\uvits)$ at which we have
$${\bf (S_E)}\quad \Action_{\mu\mu}\neq 0\,,\quad \det(\Hess\Action)\neq 0\,,\quad \negsign(\Hess\Action)\,=\,2\,,$$
the associated  periodic wave solution to {\rm (EKE)} is conditionally, orbitally stable in $H^1(\R/\Upsilon\Z)\times L^2(\R/\Upsilon\Z)$.
\end{theorem}

\begin{proof}
It is of course very similar to that of Theorem \ref{thm:EKLorb}. Our assumptions 
ensure that {\bf (H0)}-{\bf (H1)}-{\bf (H2)}-{\bf (H3)} are all met.
The kernel of $\Linvar_E$ is spanned by $\transp{(\urhox,\uvits_x)}$ because the one of $\linvar_E$ is spanned by $\urhox$, which follows from the first statement in Lemma \ref{lem:Sturm}, applied as in the proof of Theorem \ref{thm:indices}.
The latter also shows that
$$\negsign(\Linvar_L)-\negsign(\bC_L) \,=\,\negsign(\Hess\Action)-2\,=\,0$$
by assumption. So we can apply Theorem \ref{thm:orb} here too.
\end{proof}

We may now take advantage of the connexion between {\rm (EKL)} and {\rm (qKdV)} pointed out in \S \ref{ss:conn-qKdV} to compare the stability criteria in Theorems \ref{thm:qKdV}, \ref{thm:EKL}, \ref{thm:EKE}.

As shown in Appendix \ref{app:signs}, we have 
\begin{equation}\label{eq:detActionaction}
\det(\Hess\Action)=
(2\action_\speed\action_\mu -\action_\lambda^2)\,(\action_{\mu\mu}\action_{\lambda\lambda}-\action_{\lambda\mu}^2) - 4 j^2\,\action_\mu\,\det(\Hess \action)\,,
\end{equation}
\begin{equation}\label{eq:negsignActionaction}
\negsign(\Hess\Action) = 1 + \negsign (\Hess\action-(2\action_\speed\action_\mu- \action_\lambda^2)/(4j^2\action_\mu) J)\,,\;J:= \left(\begin{array}{ccc} 0 & 0 & 0\\
0 & 0 & 0 \\
0 & 0 & 1\end{array}\right)\,,
\end{equation}
where the derivatives of $\action$ are all evaluated at $(\lambda-\tfrac12 \sigma^2,j\sigma-\mu,-j^2)$.
The fact that $\uvol$ is a nonconstant periodic function of period $\action_\mu=\Upsilon$ implies at the same time 
that $\action_\mu>0$ and 
$$2\action_\speed\action_\mu -\action_\lambda^2=\Upsilon\,\int_{0}^{\Upsilon} \uvol^2 \dif y\,-\,\big(\int_{0}^{\Upsilon} \uvol\dif y\big)^2>0$$
by the Cauchy--Schwarz inequality. Therefore, the (qKdV) stability conditions in {(\bf s)} automatically imply, by
\eqref{eq:detActionaction}, that
$\det(\Hess\Action)> 0$
if $j\neq 0$ and $(\action_{\mu\mu}\action_{\lambda\lambda}-\action_{\lambda\mu}^2)\leq 0$. In this situation, we also have 
$$\negsign (\Hess\action-(2\action_\speed\action_\mu- \action_\lambda^2)/(4j^2\action_\mu) J)= \negsign (\Hess\action)\,,$$
since the involved $3\times 3$ matrices have the same first two principal minors, and their determinants have the same sign. So the third condition in {(\bf s)} implies $\negsign(\Hess\Action)=2$ by \eqref{eq:negsignActionaction}.
In fact, as soon as $j\neq 0$ and
$$\det (\Hess\action)\,\det(\Hess\action-(2\action_\speed\action_\mu- \action_\lambda^2)/(4j^2\action_\mu) J)>0,$$
or equivalently --- by Eq.~\eqref{eq:detActionaction} ---
$$\det (\Hess\action)\,\det(\Hess\Action)<0\,,$$
Eq.~ \eqref{eq:negsignActionaction} implies that
the stability conditions in ${\bf (S_L)}$ are equivalent to those in {(\bf s)}, evaluated at $(\lambda-\tfrac12 \sigma^2,j\sigma-\mu,-j^2)$. We may summarize this in the following statements.

\begin{theorem}\label{thm:qKdVEKL}
Under the assumptions of Theorems \ref{thm:qKdVorb} and \ref{thm:EKLorb}, if 
$$j\neq 0\,,\quad \Action_{\mu\mu}\neq 0\,,\quad \det (\Hess\action)\,\det(\Hess\Action)<0\,,$$
then
$$\negsign (\Hess\action) =1 \quad \Leftrightarrow\quad \negsign (\Hess\Action) =2\,.$$
If this is the case then the periodic traveling wave $(\vol,\vitsL)=(\uvol,\uvitsL)(y+jt)$ is an orbitally stable solution to {\rm (EKL)} and $\vol=\uvol(y+j^2t)$ is an orbitally stable solution to {\rm (qKdV)}.
\end{theorem}

\begin{corollary}\label{cor:qKdVEKL}
Under the assumptions of Theorems \ref{thm:qKdVorb} \& \ref{thm:EKLorb}, if $j\neq 0$ and 
\begin{itemize}
\item[{\bf (s2)}]
$\action_{\mu\mu}<0\,,\quad \left|\begin{array}{cc}
\action_{\lambda\lambda} & \action_{\mu\lambda} \\
\action_{\lambda\mu} & \action_{\mu\mu}
\end{array}\right| <0\,,\quad \det(\Hess\action)<0\,,$
\end{itemize}
holds true at $(\lambda-\tfrac12 \sigma^2,j\sigma-\mu,-j^2)$,
 then the periodic traveling wave $(\vol,\vitsL)=(\uvol,\uvitsL)(y+jt)$ is an orbitally stable solution to {\rm (EKL)} and $\vol=\uvol(y+j^2t)$ is an orbitally stable solution to {\rm (qKdV)}.
\end{corollary}

\section{Numerical investigation of specific examples}\label{s:num}

In this section, we focus on the quasilinear Korteweg-de Vries equation (qKdV) and on both versions of
the Euler-Korteweg system, {\rm (EKL)} or {\rm (EKE)}, with an energy as in \eqref{eq:en}.
In all cases, the traveling profiles $\ubv$ are governed by an equation of the form
\begin{equation}\label{eq:profil_EKL}
\frac12 \cap(\ubv) \dot\ubv^2 + \Potential( \ubv ; \lambda,\speed) = \mu \, ,
\end{equation}
where $\speed$ is linked to the speed of the wave, and $\lambda$, $\mu$ are basically constants of integration.
For (qKdV), notation in \eqref{eq:profil_EKL} are perfectly consistent with those used in  \S\ref{ss:qKdV}. Indeed, $\speed$ is exactly the speed of the wave, $\lambda$ is the Lagrange multiplier in the profile equation, $\mu$ is the energy level resulting from the integrated version of this profile equation, and the potential is
$$\Potential( \bv ; \lambda,\speed)\,=\,-\,f(\bv)\,-\,\tfrac12 \speed \bv^2\,+\,\lambda\,\bv\,.$$
We warn the reader once more, however, that the abstract notation for the parameters  in \eqref{eq:profil_EKL} has to be interpreted differently for the EK system. Indeed, as should be clear from the discussion in  \S\ref{ss:conn-qKdV}, the profile equations for {\rm (EKL)} are not readily of the form \eqref{eq:profil_EKL}, which is a \emph{reduced profile equation}, obtained only after eliminating the velocity profile $\uvitsL$. 
This yields the following table of correspondence (Table \ref{tb:morenotations}) between the `practical' parameters $(\mu,\lambda,j,\sigma)$ of the EKL wave profile $(\uvol,\uvitsL)$ introduced in  \S\ref{ss:EKLEKE}, and the abstract parameters $(\mu,\lambda,\speed)$ in \eqref{eq:profil_EKL} --- which can also be checked from the relation between the EK action and the qKdV action in \eqref{eq:actionsEKLqKdV}.

\begin{table}[H]
\begin{center}\begin{tabular}{|c|c|c|c|} \hline
abstract & $\mu$ & $\lambda$ & $\speed$ \\ \hline
qKdV & $\mu$ & $\lambda$ & $\speed$ \\ \hline
EKL & $\lambda-\tfrac12 \sigma^2$ & $j\sigma-\mu$ & $-j^2$ \\ \hline
\end{tabular}
\end{center}
\caption{Parameters in the reduced profile equation}
\label{tb:morenotations}
\end{table}

In fact, the computations performed in the proof of Theorem \ref{thm:indices} also show that the reduced profile equation for {\rm (EKE)} has an abstract form as in \eqref{eq:profil_EKL}, up to adapting notation. The correspondence is summarized in the next table (Table \ref{tb:morenotationsagain}).

\begin{table}[H]
\begin{center}\begin{tabular}{|c|c|c|c|} \hline
equation & energy level & potential $\Potential$ \\ [5pt] \hline
qKdV & $\mu$ & $\Potential(\bv)=-\,f(\bv)\,-\,\tfrac12 \speed \bv^2\,+\,\lambda\,\bv$ \\ [5pt] \hline
EKL & $\lambda-\tfrac12 \sigma^2$ &  $\Potential(\bv)=-\,f(\bv)\,+\,\tfrac12 j^2 \bv^2\,+\,(j\sigma-\mu)\,\bv$ \\ [5pt] \hline
EKE & $\mu-j\sigma$ &  $\Potential(\rho)=-\,F(\rho)\,+\,\tfrac12 j^2 /\rho\,+\, (\tfrac12 \sigma^2-\lambda)\,\rho$ \\ [5pt] \hline
\end{tabular}
\end{center}
\caption{Energy level and potential in the reduced profile equation}
\label{tb:morenotationsagain}
\end{table}

We find it convenient to keep the `abstract' notation in \eqref{eq:profil_EKL} for the discussion that follows.

\subsection{Methodology}

We wish to investigate numerically the conditions for orbital, coperiodic stability that are derived in Section \ref{s:orb}, and applied to (qKdV) in \S\ref{ss:qKdV}, and to {\rm (EKL)}, {\rm (EKE)} in \S\ref{ss:EKLEKE}-\ref{ss:conn-qKdV}.
Equation \eqref{eq:profil_EKL} is obviously solvable by separation of variables, which readily gives the action
\begin{equation}
\Action\,=\,\oint \,\cap(\bv) \,\dot\bv \;\dif \bv = \oint \sqrt{2 \cap(\bv) (\mu - \Potential(\bv ; \lambda,\speed))}\; \dif \bv\,,
\end{equation}
and the period
\begin{equation}
\Upsilon\,=\,\oint \,\frac{\dif \bv}{\dot\bv}  = \oint \frac{\dif \bv}{\sqrt{2(\mu - \Potential(\bv ; \lambda,\speed))/\cap(\bv)}} \,.
\end{equation}

To actually compute these integrals, we first need the points $\bv_2$, $\bv_3$, as denoted in 
Table~\ref{tb:var},
at which $\Potential$ achieves the energy level $\mu$ and in between which $\Potential$ is less than $\mu$, the point $\bv_2$ being the trough of the wave, and $\bv_3$, the crest. These points we can compute numerically by means of, \emph{e.g.} \emph{Newton}'s method. This is what we have done, with a \emph{relative tolerance} $\varepsilon=10^{-10}$.
Then we have computed numerically
$$\Action\,=\,2\int_{\bv_2}^{\bv_3}  \sqrt{2 \cap(\bv) (\mu - \Potential(\bv ; \lambda,\speed))}\; \dif \bv\,,
$$
and also
$$
\Upsilon\,=\,2\int_{\bv_2}^{\bv_3} \frac{\dif \bv}{\sqrt{2(\mu - \Potential(\bv ; \lambda,\speed))/\cap(\bv)}} \,,
$$
by standard numerical integration techniques\footnote{Here, the trapezoidal rule. We have also tried 
\emph{Simpson}'s rule without any significant benefit.}. Some care is needed to handle properly the square root singularity at endpoints. This has been done by first making the desingularizing change of variables $$\bv=\frac{\bv_3+\bv_2}{2}+\frac{\bv_3-\bv_2}{2}\sin \omega$$ 
(already used for instance in \cite{BronskiJohnson,BNR-JNLS14}).
Then every Newton--Cotes formulae that we tried worked well, and gave very close results, with a step size $\Delta \omega=10^{-4}$.
We just had to preferably use an open Newton--Cotes formula
to avoid too much sensitivity on the numerical error made on the preliminary computations of $\bv_2$ and $\bv_3$. 

Finally, we have computed approximations of the second order derivatives of the action by means of finite differences.
To avoid any confusion, the qKdV action is denoted by $\action$ in what follows --- as in \S\ref{ss:qKdV}, and the notation $\Action$ is reserved for the EK action. The second order derivatives of $\action$ and $\Action$ are computed in the natural, `abstract' variables, respectively $(\mu,\lambda,\speed)$ and $(\mu,\lambda_1,\lambda_2, \speed)$, up to referring to Table~\ref{tb:notations} for computing derivatives of the EK action with respect to the `concrete' parameters $(\mu,\lambda,j,\sigma)$. In both cases, we have used a  9-point stencil for the computation of those second order discrete derivatives, with a step size $\Delta \nu$ which turned out to be rather severely limited by the earlier steps --- especially the numerical integration step. 
An additional difficulty is that the \emph{condition number} of the Hessian of the action is most often very large.

The tolerance $\varepsilon=10^{-10}$ and the integration step size $\Delta \omega=10^{-4}$ are basically kept constant in the results presented below, whereas the finite difference step $\Delta \nu$ varies from place to place. 

One important observation to validate our numerical approach is that, for some specific nonlinearities which correspond to completely integrable PDEs, we have access to an explicit formula for some characteristics of the Hessian of the action. To be more precise, we know from \cite[Proposition 2]{BNR-GDR-AEDP} that the \emph{characteristic velocities} of the \emph{modulated} equations are those of a matrix known in terms of the Hessian of the action, the velocity and the period of periodic waves. For (qKdV), this matrix reads
$${\bf d} := \speed \left(\begin{array}{ccc}
1 & 0 & 0 \\
0 & 1 & 0 \\
0 & 0 & 1
\end{array}\right) \,+\,(\Hess \action)^{-1}\,\left(\begin{array}{ccc}
0 & 0 & -1 \\
0 & 1 & 0 \\
-1 & 0 & 0
\end{array}\right)$$
and for (EKL), it is
$${\bf D} = - j \left(\begin{array}{cccc}
1 & 0 & 0 & 0 \\
0 & 1 & 0 & 0 \\
0 & 0 & 0 & 1
\end{array}\right) \,+\,(\Hess \Action)^{-1}\,\left(\begin{array}{cccc}
0 & 0 & 0 &  -1 \\
0 & 0 &  1 &  0 \\
0 & 1 & 0 & 0 \\
-1 & 0 & 0 & 0
\end{array}\right)\,.$$
So, in cases when these characteristic velocities are known through `explicit' formulas --- which involve elliptic functions ---, we can compare them with the eigenvalues of ${\bf d}$ and ${\bf D}$
that we obtain through our numerical approach. This is what we have done for the `standard' Korteweg-de Vries (KdV) equation ($p(\bv)=\frac12 \bv^2$, $\cap\equiv 1$) and for the Euler--Korteweg version of the cubic NLS equation ($p(\vol)=1/(2\vol^2)$, $\Cap(\rho)=1/(4\rho)$, that is, $\cap(\vol)=1/(4\vol^4)$), 
by using the expressions of characteristic velocities collected in the book by Kamchatnov \cite{Kamchatnov} (chapter 3.5, p. 184 and chapter 5.1, p. 238) and in \cite[\S~2.2]{Jenkins}.

\begin{remark}
\textup{
This testing against explicit formulas incidentally enabled us to find out that the numerical computations displayed in \cite{BNR-JNLS14} were, to some extent, corrupted, because of inconsistent choices for the finite difference step $\Delta \nu$ and for the integration step size $\Delta \omega$ --- in fact, $\Delta \nu$ was taken too small.}
\end{remark}

\subsection{A few numerical results}

\subsubsection{Benchmarks}
Let us start with the KdV case (with actually $p(\bv)=3 \bv^2$, for convenience). On Figure \ref{fig:kdv_cond} hereafter, we have plotted the condition number of $\Hess\action$ for a reasonable range of periods $\Upsilon$, almost going down to the harmonic period --- that is, the small amplitude limit --- as well as the relative error of its determinant when compared to what is expected from the explicit formulas mentioned above.
As we can see,
it is for periods close to the harmonic period that the condition number of $\Hess\action$ is the highest, and the error on the determinant reaches the highest values, around $0.12 \%$. Were the step size $\Delta \nu$ diminished, this instability phenomenon would worsen.

\begin{figure}[H]
\begin{center}
\includegraphics[width=7.5cm]{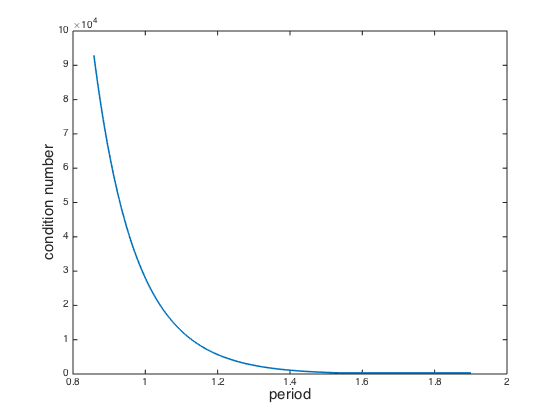}\hfill \includegraphics[width=7.5cm]{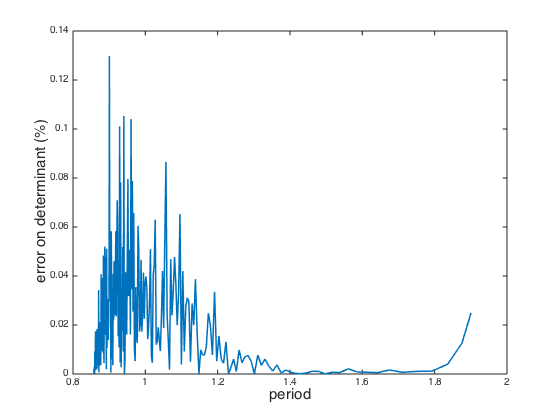}
\caption{Left: condition number of $\Hess \action$ as a function of the period $\Upsilon$ for KdV, with $\speed=60$, 
$\lambda=-60$ kept fixed. 
Right: relative error on $\det\Hess \action$. Finite difference step size: $\Delta \nu=0.005$.
}
\label{fig:kdv_cond}
\end{center}
\end{figure}

Figure \ref{fig:conditions_kdv} hereafter displays, with the same parameter values as in Figure \ref{fig:kdv_cond},
 the minors that encode the orbital stability condition in Theorem \ref{thm:qKdVorb}, namely
 $\mineur_1 = \action_{\mu \mu}$, $\mineur_2 = \action_{\mu \mu} \action_{\lambda \lambda} - \action_{\lambda \mu}^2$ and $\mineur_3 = \det\Hess\action$. 
\begin{figure}[H]
\begin{center}
 \includegraphics[width=16cm]{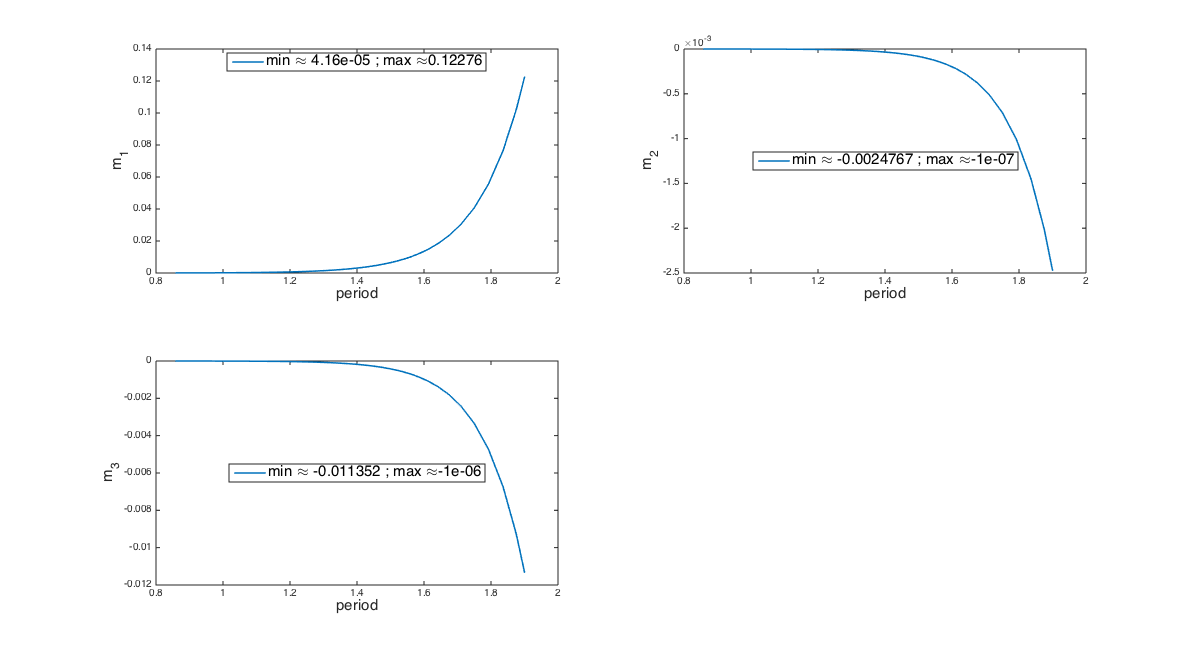}
\caption{Upper left: $\mineur_1 = \action_{\mu \mu}$; upper right: $\mineur_2 = \action_{\mu \mu} \action_{\lambda \lambda} - \action_{\lambda \mu}^2$; and lower left: $\mineur_3 = \det\Hess\action$, all of them as functions of the period $\Upsilon$ for KdV with $\speed=60$, 
$\lambda=-60$ kept fixed. Finite difference step size: $\Delta \nu=0.005$. }
\label{fig:conditions_kdv}
\end{center}
\end{figure}
By looking at extreme values of those minors given in boxes, or by zooming on and eye-checking the parts of their curves that look closest to zero, we recover the signs that confirm orbital stability for KdV. More precisely, we find that the signs of minors of $\Hess\action$ are as in the third row on Table \ref{tb:signs}, 
$$\action_{\mu\mu}>0\,,\quad\left|\begin{array}{cc}
\action_{\lambda\lambda} & \action_{\mu\lambda} \\
\action_{\lambda\mu} & \action_{\mu\mu}
\end{array}\right| <0\,,\quad \det(\Hess\action)<0\,,$$
(which is \eqref{eq:Johnson}, corresponding to the set of conditions found by Johnson in \cite{Johnson}).

Let us now consider the \emph{cubic} NLS nonlinearity, whose numerical behavior differs. Again, on Figure \ref{fig:nls_cond} below, we have plotted the condition number of $\Hess\Action$ and the relative error of its determinant.

\begin{figure}[H]
\begin{center}
\includegraphics[width=7.5cm]{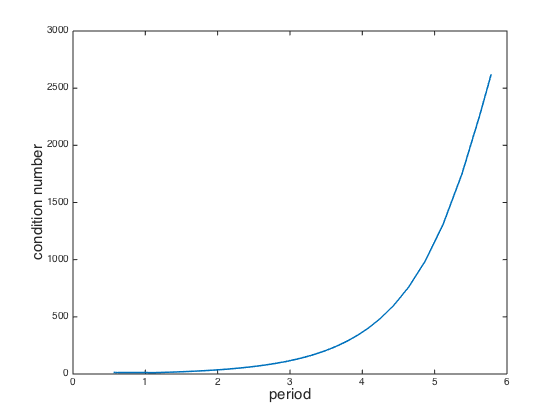} \hfill \includegraphics[width=7.5cm]{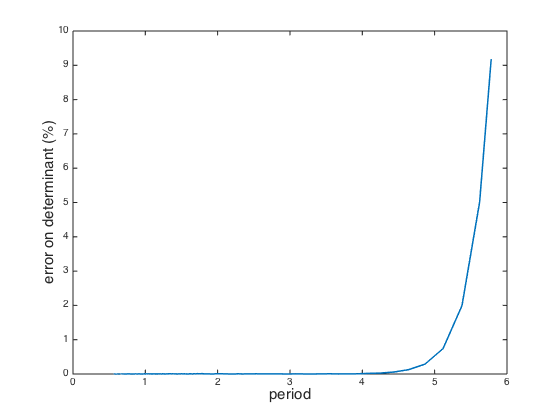}
\caption{Left: condition number of $\Hess \Action$ as a function of the period for NLS, with
$j=1$, $\sigma=0$, $\mu=-2.5$ kept fixed.
Right: relative error on $\det\Hess \Action$. Finite difference step size: $\Delta \nu=10^{-5}$.}
\label{fig:nls_cond}
\end{center}
\end{figure}

Here the most important numerical problems come from large periods, that is, when we approach the soliton limit. The condition number gets quite large, and the error on the determinant obviously blows up in this limit. Nevertheless, we can plot the minors encoding orbital stability as long as the period is not too large.
We first plot, in Figure \ref{fig:det_nls}, the four minors of $\Hess\Action$ encoding the stability conditions in ${\bf (S_L)}$ (in Theorem \ref{thm:EKLorb}), with in particular $\Mineur_1 = \Action_{\lambda\lambda}$ and $\Mineur_4 = \det(\Hess\Action)$, and then also look at the missing one for checking the stability conditions in ${\bf (S_E)}$ (in Theorem \ref{thm:EKEorb}), namely $\Action_{\mu\mu}$. 

\begin{figure}[H]
\begin{center}
\includegraphics[width=16cm]{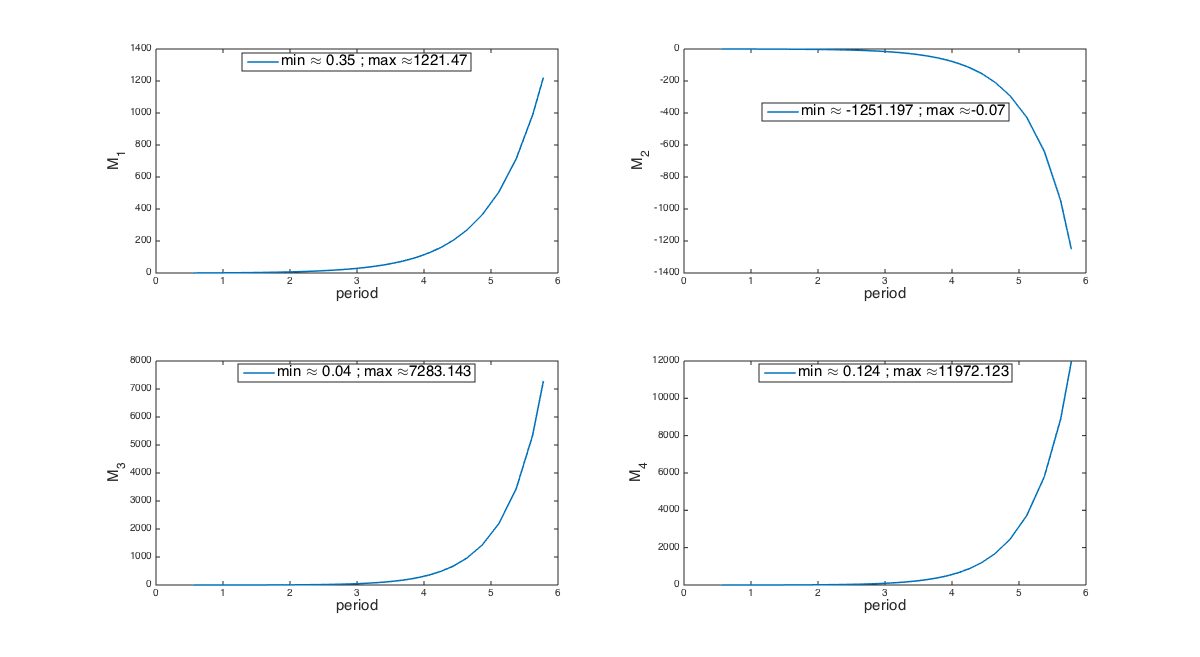}
\caption{The four principal minors: $\Mineur_1=\Action_{\lambda\lambda}$ (upper left);
$\Mineur_2 = \Action_{\mu \mu} \Action_{\lambda \lambda} - \Action_{\lambda \mu}^2$ (upper right);
$\Mineur_3$ (lower left); and
$\Mineur_4 = \det(\Hess\Action)$ (lower right),
as a function of the period for NLS case, with $j=1$, $\sigma=0$, $\mu=-2.5$ kept fixed. Finite difference step size: $\Delta \nu=10^{-5}$.}
\label{fig:det_nls}
\end{center}
\end{figure}
\begin{figure}[H]
\begin{center}
\includegraphics[width=7.5cm]{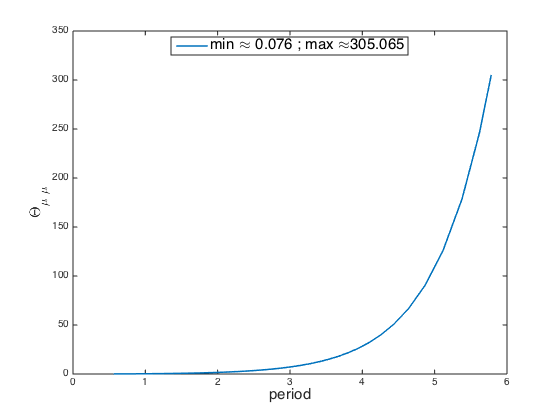}
\caption{$\Action_{\mu\mu}$ as a function of the period with the same data as in Figure \ref{fig:det_nls}.
}
\label{fig:conditions_nls}
\end{center}
\end{figure}

Again, the signs are not easily visible for periods approaching the small amplitude limit, but the extreme values of minors displayed in boxes show that 
$$\Action_{\mu\mu}> 0\,,\quad \Action_{\lambda\lambda}> 0\,,\quad \det(\Hess\Action)> 0\,,\quad \negsign(\Hess\Action)\,=\,2\,,$$
(the latter being a consequence of Sylvester's rule and the fact that  there are exactly two sign changes in the sequence of minors, since $\Mineur_2<0$ and all the others are positive),
which imply that both ${\bf (S_L)}$ and ${\bf (S_E)}$
are satisfied. This confirms, as known from \cite{GallayHaragus2}, that cubic NLS periodic waves are orbitally stable, including in the mass Lagrangian coordinates of the fluid formulation of NLS. 

\subsubsection{More qKdV test cases}

Let us now discuss (qKdV), with the aim of deriving new stability results in non integrable cases. We set again $\cap\equiv 1$ --- which means we actually concentrate on (gKdV) --- and we look at a pressure law given by 
$$ p(v) = e(\gamma+1) v^{\gamma} , \ \gamma \geq 2 \ , \ e = \pm 1\ .$$

The minus sign is usually referred to as the defocusing case. Its introduction is in fact irrelevant if $\gamma$ is an even integer, because in this case the symmetry $(x,t,v)\mapsto (-x,-t,-v)$ drives back the minus sign to a plus sign in (gKdV).

The case $\gamma =2$ corresponds to (KdV), while the plus sign with $\gamma=3$ corresponds to what is known as the \emph{modified} KdV equation (mKdV). This is another completely integrable case, for which the results are qualitatively similar to those of KdV as displayed on Figures \ref{fig:mkdv} - \ref{fig:conditions_mkdv} here after. According to \cite[Theorem 5.2]{AnguloPava_mKdV-NLS}, \textit{dnoidal} waves should be orbitally stable. We chose our parameters in order to observe this particular family of waves (see Figure 2.1 in \cite{Johnson}), and we find indeed the expected stability, the minors of $\Hess\action$ satisfying again \eqref{eq:Johnson}.

\begin{figure}[H]
\begin{center}
\includegraphics[width=7.5cm]{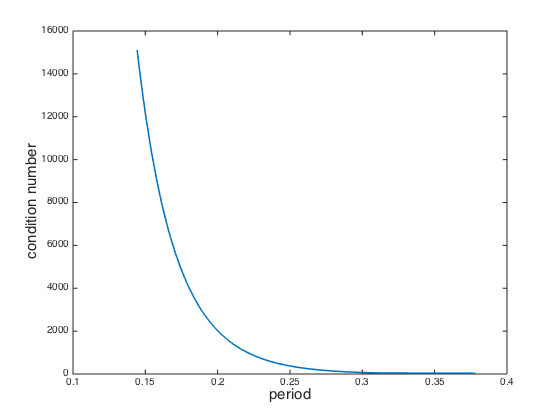}
\caption{Condition number of $\Hess \action$ as a function of the period for focusing (mKdV), that is $\gamma = 3$ and $e = 1$, with $c = 1000$ and $\lambda = -500$  kept fixed. Finite difference step size: $\Delta \nu=0.05$.}
\label{fig:mkdv}
\end{center}
\end{figure}

\begin{figure}[H]
\begin{center}
\hfill \includegraphics[width=16cm]{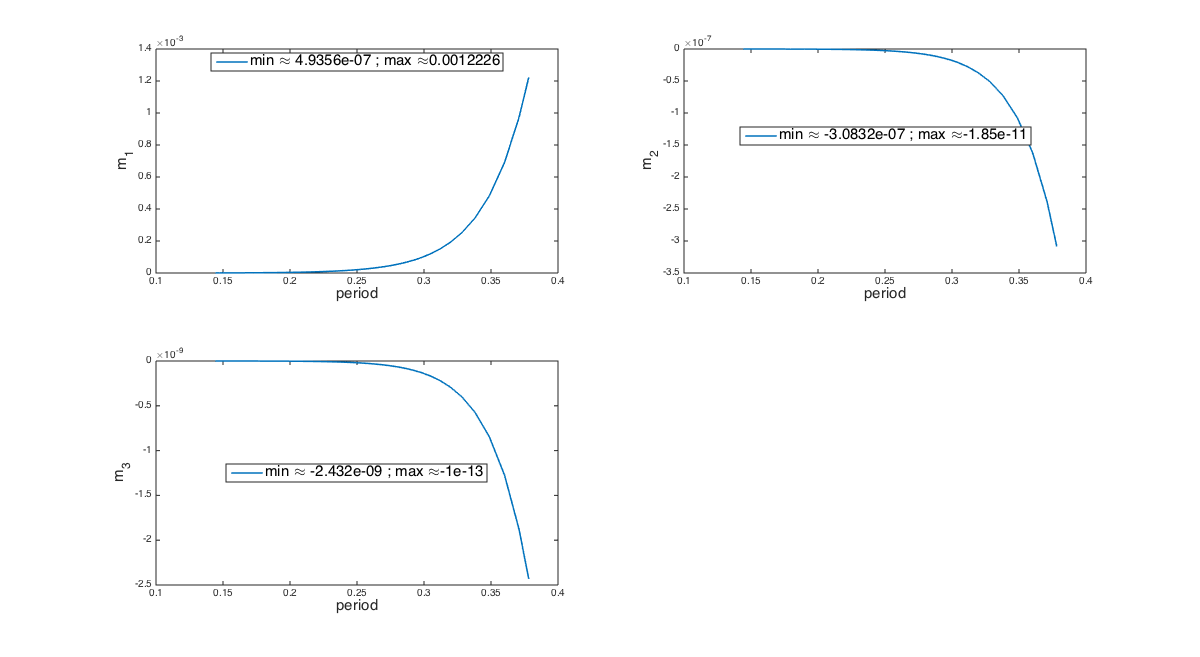}
\caption{Upper left: $\mineur_1 = \action_{\mu \mu}$; upper right: $\mineur_2 = \action_{\mu \mu} \action_{\lambda \lambda} - \action_{\lambda \mu}^2$; and lower left: $\mineur_3 = \det\Hess\action$,  as a function of the period for focusing (mKdV), with the same data as in Figure \ref{fig:mkdv}.}
\label{fig:conditions_mkdv}
\end{center}
\end{figure}

When going to the defocusing mKdV case (with a minus sign in the pressure and  $\gamma = 3$), we observe on Figure 
 \ref{fig:conditions_mkdv_f} that the intermediate minor $\mineur_2$ has the opposite sign, compared to what happens for the focusing mKdV. However, the negative signature of $\Hess\action$ remains equal to one, and this confirms
orbital stability for the defocusing mKdV equation, as expected from \cite{Deconinck-Nivala_mKdV}.

\begin{figure}[H]
\begin{center}
\includegraphics[width=7.5cm]{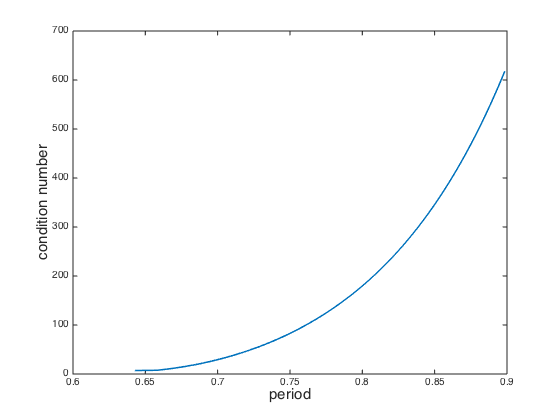}
\caption{Condition number of $\Hess \action$ as a function of the period for defocusing (mKdV), that is $\gamma = 3$ and $e = -1$, with $c = -100$ and $\lambda = -60$ kept fixed. Finite difference step size: $\Delta \nu = 0.005$.}
\label{fig:mkdv_f}
\end{center}
\end{figure}

\begin{figure}[H]
\begin{center}
\hfill \includegraphics[width=16cm]{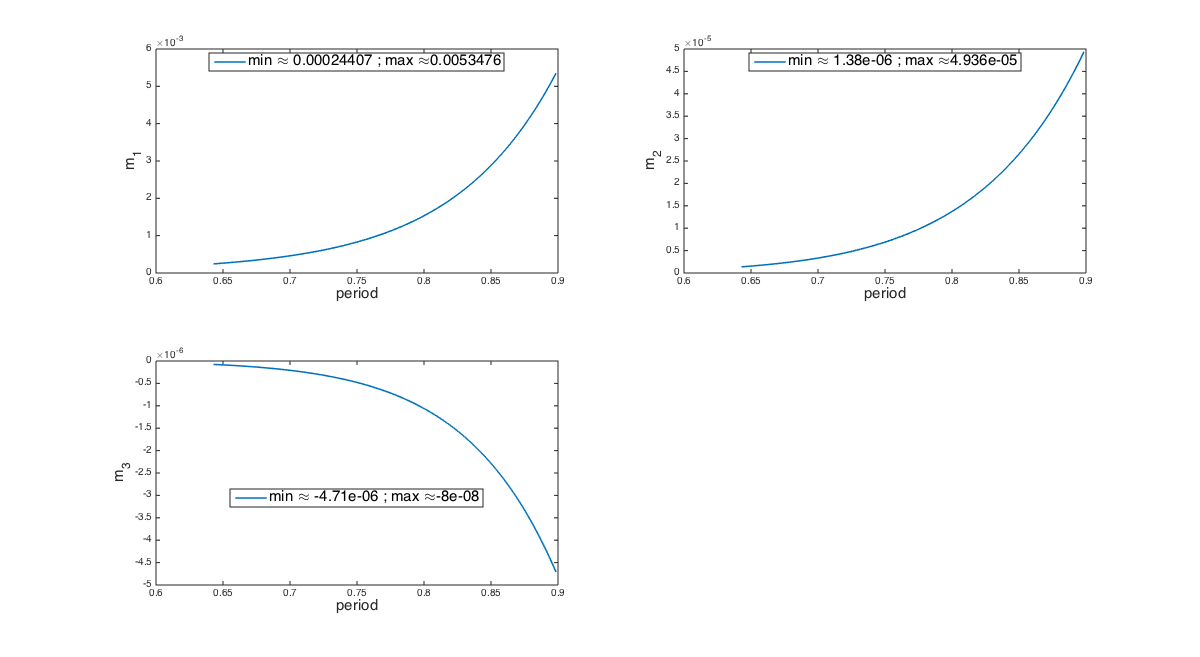}
\caption{Upper left: $\mineur_1 = \action_{\mu \mu}$; upper right: $\mineur_2 = \action_{\mu \mu} \action_{\lambda \lambda} - \action_{\lambda \mu}^2$; and lower left: $\mineur_3 = \det\Hess\action$,  as a function of the period for defocusing (mKdV), with the same data as in Figure \ref{fig:mkdv_f}.}
\label{fig:conditions_mkdv_f}
\end{center}
\end{figure}

Let us come back to the focusing gKdV, with this time $\gamma = 4$. This is a non integrable case, for which we are not aware of any analytical result regarding the stability of periodic waves. It is only known from \cite{PegoWeinstein} that solitary waves are orbitally stable, which is a necessary condition for the stability of periodic waves of large period,  by the work by Gardner \cite{Gardner97}.
Numerical difficulties arise here at both ends. We have indeed large values of the condition number in both the small amplitude limit and the soliton limit, as can be seen on Figure \ref{fig:gkdv3}. Nevertheless, there is numerical evidence for stability (Johnson's conditions in \eqref{eq:Johnson} are satisfied) at least for intermediate periods, for which the numerical results are most reliable, see Figure \ref{fig:conditions_gkdv3}. 

\begin{figure}[H]
\begin{center}
\includegraphics[width=7.5cm]{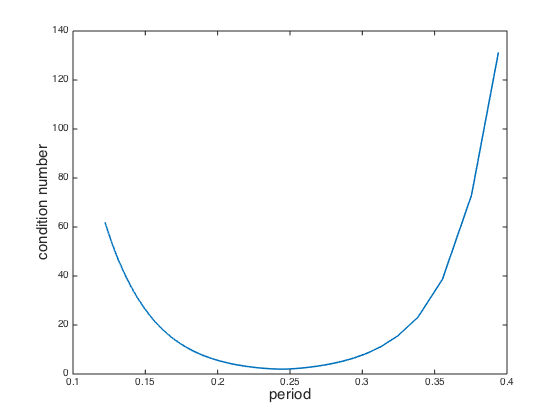}
\caption{Condition number of $\Hess \action$ as a function of the period for (gKdV) with $\gamma = 4$ and $e = 1$, and $c = 1000$, $\lambda = -500$ kept fixed. Finite difference step size: $\Delta \nu = 0.005$.}
\label{fig:gkdv3}
\end{center}
\end{figure}
\begin{figure}[H]
\begin{center}
\hfill \includegraphics[width=16cm]{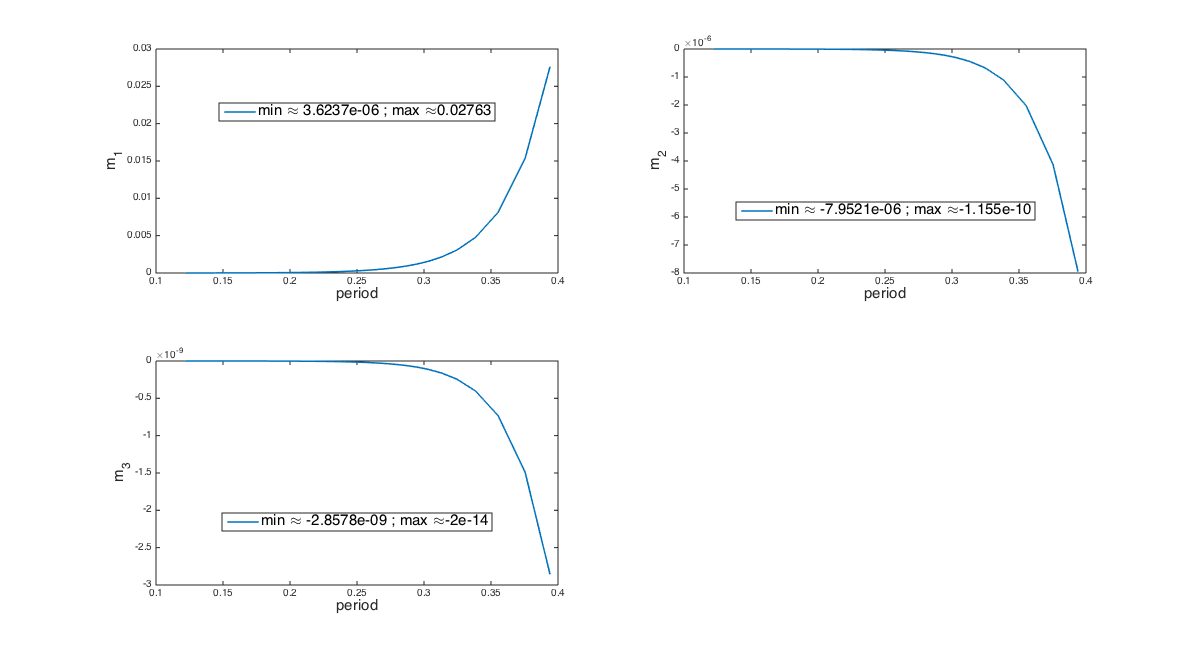}
\caption{Upper left: $\mineur_1 = \action_{\mu \mu}$; upper right: $\mineur_2 = \action_{\mu \mu} \action_{\lambda \lambda} - \action_{\lambda \mu}^2$; and lower left: $\mineur_3 = \det\Hess\action$, as a function of the period for (gKdV) with $\gamma = 4$ and $e = 1$, with the same data as in Figure \ref{fig:gkdv3}.}
\label{fig:conditions_gkdv3}
\end{center}
\end{figure}

\subsubsection{More EK test cases}

Let us now focus on the Euler--Korteweg system. As for the NLS case, we investigate the orbital stability conditions for both (EKL) and (EKE), so as to check that stability occurs at the same time in the two formulations. First, we consider a Boussinesq pressure law 
$$ p(v) = v - v^{\gamma} \ ,$$
with a constant capillarity $\cap \equiv 1$. In what follows, we  take $\gamma = 2$, which corresponds to the \textit{good} Boussinesq equation, as in \cite[\S~4.2.1]{BronskiJohnsonKapitulaII}.

Numerical results are displayed on Figures \ref{fig:bouss}-\ref{fig:conditions_bouss}. We observe a transition at period $\Upsilon_0 \simeq 3.68$, where $\det\Hess\Action$ changes sign (and $\negsign(\Hess \Action)$ passes from $2$ to $3$). For periods smaller than $\Upsilon_0$, we see that we are in the range of application of Theorems~\ref{thm:EKLorb} \&~\ref{thm:EKEorb}, which imply orbital stability for both (EKL) and (EKE).
For periods larger than $\Upsilon_0$, since $\det\Hess\Action<0$, Theorem \ref{thm:EvansEK} implies spectral instability.
We checked that the zone around $\Upsilon_0$ where the condition number becomes very high does not depend on our choice of discretization steps  --- unlike what happens in the KdV/NLS cases for periods approaching the harmonic one or going to infinity. The transition at $\Upsilon_0$ is not a numerical artifact. Our conclusions are thus consistent with those in \cite{HakkaevStanislavovaStefanov}, where it is pointed out that, at a fixed velocity, there exists a maximal period for the wave to be stable.

\begin{figure}[H]
\begin{center}
\includegraphics[width=7.5cm]{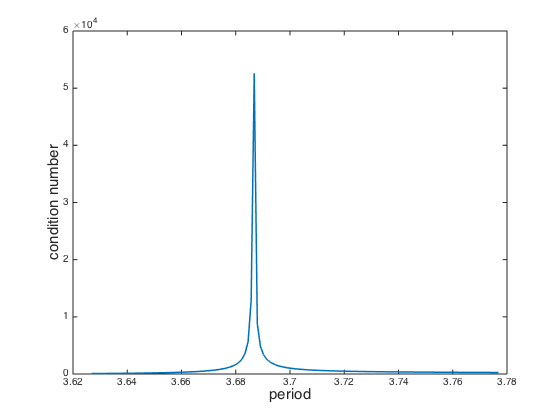} \hfill \includegraphics[width=7.5cm]{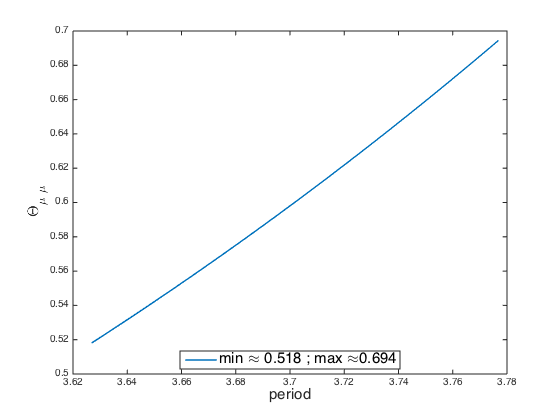}
\caption{Left: condition number of $\Hess \Action$ as a function of the period for EKL with Boussinesq pressure law with $\gamma = 2$, $j=-0.1$, $\sigma=0$, $\mu=-2$ kept fixed.
Right: $\Action_{\mu\mu}$. Finite difference step size: $\Delta \nu=0.5 \cdot 10^{-4}$.}
\label{fig:bouss}
\end{center}
\end{figure}

\begin{figure}[H]
\begin{center}
\includegraphics[width=16cm]{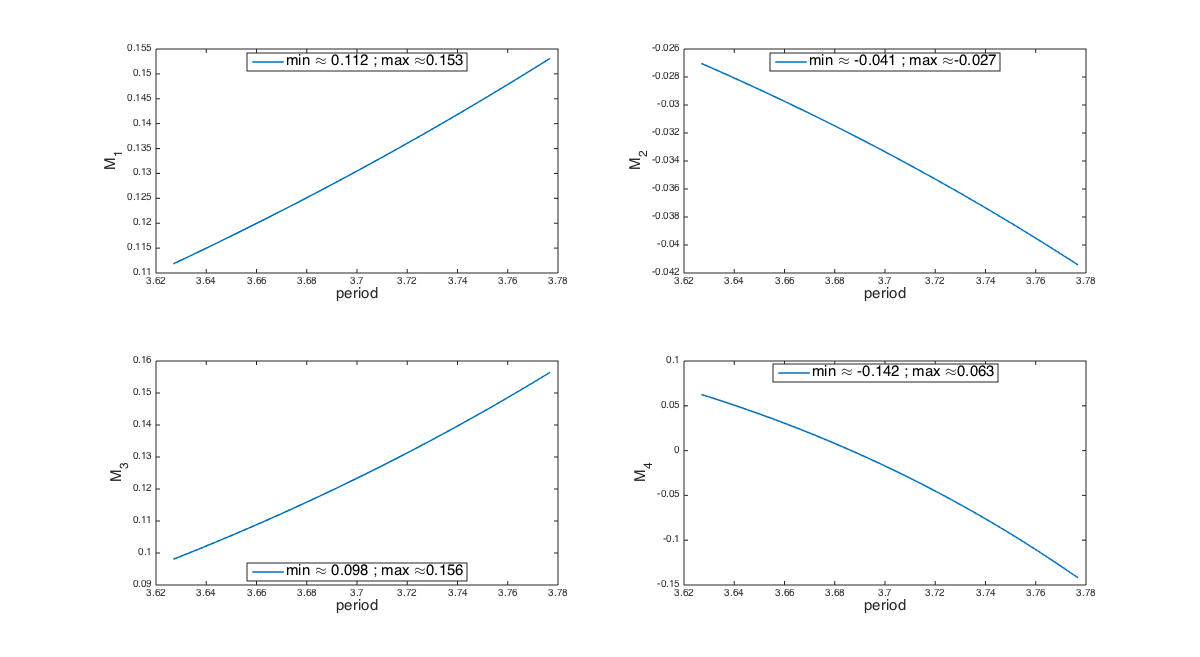}
\caption{The four principal minors: $\Mineur_1=\Action_{\lambda\lambda}$ (upper left);
$\Mineur_2 = \Action_{\mu \mu} \Action_{\lambda \lambda} - \Action_{\lambda \mu}^2$ (upper right);
$\Mineur_3$ (lower left); and
$\Mineur_4 = \det(\Hess\Action)$ (lower right), for EKL with Boussinesq pressure law with $\gamma = 2$ and the same data as on Figure \ref{fig:bouss}.}
\label{fig:conditions_bouss}
\end{center}
\end{figure}

Let us now consider a last case. The perfect gas pressure law 
$$p(v) = 1/2v \ ,$$
with a constant capillarity $\Cap(\rho) \equiv 1$, or equivalently, $\cap(v) = 1/v^5$.

\begin{figure}[H]
\begin{center}
\includegraphics[width=7.5cm]{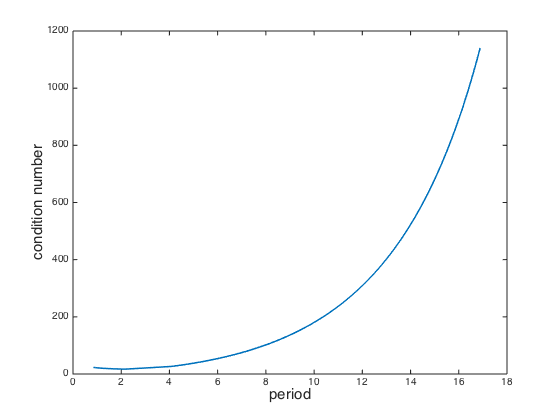} \hfill \includegraphics[width=7.5cm]{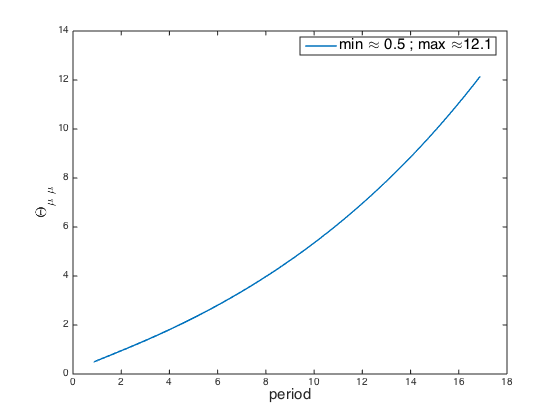}
\caption{Left: condition number of $\Hess \Action$ as a function of the period for EK with the perfect gas pressure law, with $j=-1$, $\sigma=0$, $\mu=-2.5$ kept fixed.
Right: $\Action_{\mu\mu}$. Finite difference step size: $\Delta \nu=0.5 \cdot 10^{-4}$.}
\label{fig:gp}
\end{center}
\end{figure}

\begin{figure}[H]
\begin{center}
\includegraphics[width=16cm]{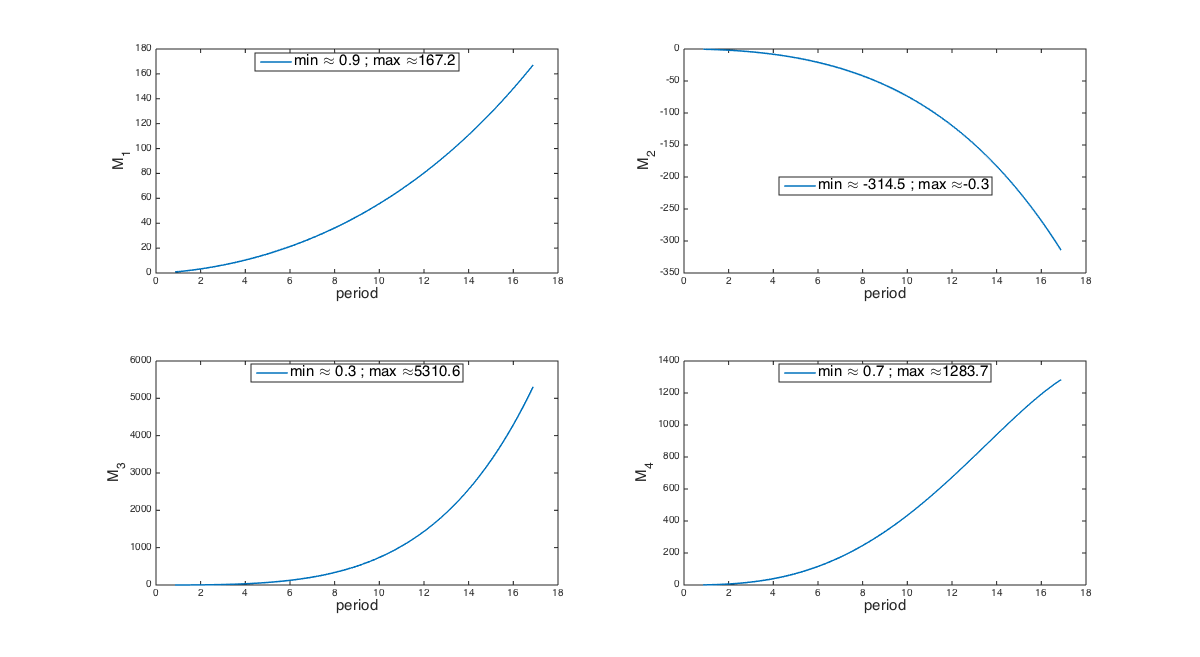}
\caption{The four principal minors: $\Mineur_1=\Action_{\lambda\lambda}$ (upper left);
$\Mineur_2 = \Action_{\mu \mu} \Action_{\lambda \lambda} - \Action_{\lambda \mu}^2$ (upper right);
$\Mineur_3$ (lower left); and
$\Mineur_4 = \det(\Hess\Action)$ (lower right), for EKL with the perfect gas pressure law and the same data as in Figure \ref{fig:gp}.}
\label{fig:conditions_gp}
\end{center}
\end{figure}

We observe that ${\bf (S_L)}$ and ${\bf (S_E)}$ are satisfied, so that both Theorems~\ref{thm:EKLorb} \&~\ref{thm:EKEorb} apply. The tested periodic waves in the Euler--Korteweg system with the perfect gas pressure law are orbitally stable in both mass Lagrangian coordinates and Eulerian coordinates.

\section*{Appendix}
\setcounter{subsection}{0}
\renewcommand{\thesubsection}{\Alph{subsection}}
\renewcommand{\thetheorem}{\Alph{subsection}{theorem}}

\subsection{Algebraic computations regarding $\Action$ and $\action$}\label{app:signs}
Recall from \eqref{eq:actionsEKLqKdV} that
$$\Action(\mu,\lambda,j,\sigma)\,=\,\action(\lambda-\tfrac12 \sigma^2,j\sigma-\mu,-j^2)\,,$$
and that the subscripts $\mu$, $\lambda$, $\speed$ will denote, when attached to $\action$, partial derivatives of $\action$ with respect to its first, second, and third variable respectively. Therefore, by applying twice the chain rule we obtain
$$ 
\begin{array}{rl}
\Hess\Action&\displaystyle
=
\left(\begin{array}{cccc} 
\Action_{\mu \mu} & \Action_{\lambda\mu} & \Action_{j \mu} & \Action_{\sigma \mu} \\
\Action_{\mu \lambda} & \Action_{\lambda\lambda} & \Action_{j \lambda} & \Action_{\sigma \lambda} \\
\Action_{\mu j} & \Action_{\lambda j} & \Action_{j j} & \Action_{\sigma j} \\
\Action_{\mu \sigma} & \Action_{\lambda \sigma} & \Action_{j \sigma} & \Action_{\sigma \sigma} \\
\end{array}\right)\\[2em]
&\displaystyle
= \left(\begin{array}{c|c|c|c} 
\action_{\lambda\lambda} & -\action_{\lambda\mu} & -\sigma\action_{\lambda\lambda}+2j\action_{\speed\lambda} & -j\action_{\lambda\lambda}+\sigma\action_{\lambda\mu}\\ \hline
& \action_{\mu\mu} & -2 j\action_{\speed\mu} +\sigma\action_{\lambda\mu} & -\sigma\action_{\mu\mu}+j\action_{\lambda\mu} \\ \hline
& & 4j^2\action_{\speed\speed} + \sigma^2 \action_{\lambda\lambda} & 
j\sigma\action_{\lambda\lambda} -2j^2\action_{\speed\lambda}+2j\sigma\action_{\speed\mu} \\ 
& & - 4 j \sigma \action_{\speed\lambda}-2\action_\speed& - \sigma^2\action_{\lambda\mu}+\action_\lambda \\ \hline 
& & &
j^2\action_{\lambda\lambda} + \sigma^2 \action_{\mu\mu}\\
& & & -2j\sigma\action_{\lambda\mu}-\action_\mu
\end{array}\right)\,,
\end{array}
$$
where the entries of  $\Hess\Action$ are evaluated at $(\mu,\lambda,j,\sigma)$, while the derivatives of $\action$ are meant to be evaluated at $(\lambda-\tfrac12 \sigma^2,j\sigma-\mu,-j^2)$ --- all along this section\footnote{In particular, the reader should keep in mind that with our conventions,
$$\action_{\lambda\lambda}=\Action_{\mu\mu}\,,\;\action_{\mu\mu}=\Action_{\lambda\lambda}\,.$$}---, and we have omitted to write the symmetric entries under the diagonal in order to save some space. 

By elementary manipulations on columns and rows --- as in Proposition \ref{prop:ActionC} --- we thus find a matrix 
${\bQ}\in {\sf SL}_{4}(\R)$ such that
$$ {\bQ} \,(\Hess\Action)\, \transp{{\bQ}}=
\left(\begin{array}{c|c|c|c} \action_{\lambda\lambda} & -\action_{\lambda\mu} & 2j\action_{\speed\lambda} & 0\\ \hline
-\action_{\lambda\mu} & \action_{\mu\mu} & -2 j\action_{\speed\mu}  & 0 \\ \hline
2j\action_{\speed\lambda} & -2 j\action_{\speed\mu} & 4j^2\action_{\speed\speed} -2\action_\speed& \action_\lambda \\ \hline 
0 & 0 & \action_\lambda & -\action_\mu
\end{array}\right)\,.$$
Assuming that $j\neq 0$ --- and recalling that $\Upsilon=\action_\mu>0$ 
--- we can perform further manipulations, and thus find 
${\bf R}\in {\sf GL}_{4}(\R)$ such that $\det {\bf R}=1/(2j)$ and
\begin{equation}\label{eq:Actionaction}
 {\bf R} \,(\Hess\Action)\, \transp{{\bf R}}=
\left(\begin{array}{c|c|c|c} \action_{\lambda\lambda} & \action_{\lambda\mu} & \action_{\speed\lambda} & 0\\ \hline
\action_{\lambda\mu} & \action_{\mu\mu} & \action_{\speed\mu}  & 0 \\ \hline
\action_{\speed\lambda} & \action_{\speed\mu} &  \action_{\speed\speed} - \eta & 0 \\ \hline 
0 & 0 & 0 & -\action_\mu
\end{array}\right)\,, \; \eta:=(2\action_\speed\action_\mu- \action_\lambda^2)/(4j^2\action_\mu)\,.
\end{equation}
Observe that the upper $3\times 3$ block in the right-hand side is
$\Hess\action -\eta J$, where 
$$J:= \left(\begin{array}{ccc} 0 & 0 & 0\\
0 & 0 & 0 \\
0 & 0 & 1\end{array}\right)\,.$$
In particular, \eqref{eq:Actionaction} yields
$$\frac{1}{4j^2}\,\det(\Hess\Action)\,=\,-\action_\mu\,\det (\Hess\action-\eta J)\,,$$
which equivalently reads
$$\det(\Hess\Action)=
(2\action_\speed\action_\mu -\action_\lambda^2)\,(\action_{\mu\mu}\action_{\lambda\lambda}-\action_{\lambda\mu}^2) - 4 j^2\,\action_\mu\,\det(\Hess \action)\,.
$$
Since $\action_\mu>0$,  \eqref{eq:Actionaction} also gives
$$\negsign(\Hess\Action) = 1 + \negsign (\Hess\action-\eta J)\,.$$
This proves Eqs.~\eqref{eq:detActionaction} and \eqref{eq:negsignActionaction}.

We can also draw some sign tables that yield the signatures of $\Hess\action$, $\bC_L$, and $\bC_E$ by inspection of their minors.

Let us first consider  $\bC_L$, written as in \eqref{eq:CL},
$$
\bC_L=- \frac{1}{\Action_{\lambda\lambda}}
\left(\begin{array}{ccc} 
\left|\begin{array}{cc}
\Action_{\mu\mu} & \Action_{\lambda\mu} \\
\Action_{\mu\lambda} & \Action_{\lambda\lambda}
\end{array}\right| & 
\left|\begin{array}{cc}
\Action_{\mu\sigma} & \Action_{\lambda\sigma} \\
\Action_{\mu\lambda} & \Action_{\lambda\lambda}
\end{array}\right|
& \left|\begin{array}{cc}
\Action_{\mu j} & \Action_{\lambda j} \\
\Action_{\mu\lambda} & \Action_{\lambda\lambda}
\end{array}\right| \\ [15pt]
\left|\begin{array}{cc}
\Action_{\sigma\mu} & \Action_{\lambda\mu} \\
\Action_{\sigma\lambda} & \Action_{\lambda\lambda}
\end{array}\right| & 
\left|\begin{array}{cc}
\Action_{\sigma\sigma} & \Action_{\lambda\sigma} \\
\Action_{\sigma\lambda} & \Action_{\lambda\lambda}
\end{array}\right|
& \left|\begin{array}{cc}
\Action_{\sigma  j} & \Action_{\lambda j} \\
\Action_{\sigma\lambda} & \Action_{\lambda\lambda}
\end{array}\right| \\ [15pt]
\left|\begin{array}{cc}
\Action_{j \mu} & \Action_{\lambda\mu} \\
\Action_{j \lambda} & \Action_{\lambda\lambda}
\end{array}\right| & 
\left|\begin{array}{cc}
\Action_{j \sigma} & \Action_{\lambda\sigma} \\
\Action_{j \lambda} & \Action_{\lambda\lambda}
\end{array}\right|
& \left|\begin{array}{cc}
\Action_{j j} & \Action_{\lambda j} \\
\Action_{j \lambda} & \Action_{\lambda\lambda}
\end{array}\right|
\end{array}\right)
$$
which equivalently reads
$$
\bC_L= - \frac{1}{\Action_{\lambda\lambda}}
\left(\begin{array}{ccc} 
\left|\begin{array}{cc}
\action_{\lambda\lambda} & -\action_{\lambda\mu} \\
-\action_{\mu\lambda} & \action_{\mu\mu}
\end{array}\right| & 
\left|\begin{array}{cc}
0 & 0 \\
-\action_{\mu\lambda} & \action_{\mu\mu}
\end{array}\right|
& \left|\begin{array}{cc}
2j\action_{\lambda\speed} & -2j\action_{\mu\speed} \\
-\action_{\mu\lambda} & \action_{\mu\mu}
\end{array}\right| \\ [15pt]
\left|\begin{array}{cc}
0 & -\action_{\lambda\mu} \\
0 & \action_{\mu\mu}
\end{array}\right| & 
\left|\begin{array}{cc}
-\action_{\mu} & 0 \\
0 & \action_{\mu\mu}
\end{array}\right|
& \left|\begin{array}{cc}
\action_{\lambda} & -2j\action_{\mu\speed} \\
0 & \action_{\mu\mu}
\end{array}\right| \\ [15pt]
\left|\begin{array}{cc}
2j\action_{\speed\lambda} & -\action_{\lambda\mu} \\
-2j\action_{\speed\mu} & \action_{\mu\mu}
\end{array}\right| & 
\left|\begin{array}{cc}
\action_{\lambda} & 0 \\
-2j\action_{\speed\mu} & \action_{\mu\mu}
\end{array}\right|
& \left|\begin{array}{cc}
4j^2\action_{\speed\speed} -2\action_\speed & -2j\action_{\mu\speed} \\
-2j\action_{\speed\mu} & \action_{\mu\mu}
\end{array}\right|
\end{array}\right)
$$
$$=
- \frac{1}{\Action_{\lambda\lambda}}
\left(\begin{array}{ccc} 
\left|\begin{array}{cc}
\action_{\lambda\lambda} & \action_{\lambda\mu} \\
\action_{\mu\lambda} & \action_{\mu\mu}
\end{array}\right| & 
0
& 2j \left|\begin{array}{cc}
\action_{\lambda\speed} & \action_{\mu\speed} \\
\action_{\lambda\mu} & \action_{\mu\mu}
\end{array}\right| \\ [15pt]
0 & 
-\action_{\mu}  \action_{\mu\mu}
& 
\action_{\lambda} \action_{\mu\mu}
 \\ [15pt]
2j\left|\begin{array}{cc}
\action_{\speed\lambda} & \action_{\mu\lambda} \\
\action_{\speed\mu} & \action_{\mu\mu}
\end{array}\right| & 
\action_{\lambda}  \action_{\mu\mu}
& 4j^2 \left|\begin{array}{cc}
\action_{\speed\speed}  & \action_{\mu\speed} \\
\action_{\speed\mu} & \action_{\mu\mu}
\end{array}\right| -2\action_\speed \action_{\mu\mu}
\end{array}\right)
$$
This is to be compared with the constraint matrix for (qKdV), which is given by
$$
\bc=- \frac{1}{\action_{\mu\mu}}
\left(\begin{array}{cc} 
\left|\begin{array}{cc}
\action_{\lambda\lambda} & \action_{\mu\lambda} \\
\action_{\lambda\mu} & \action_{\mu\mu}
\end{array}\right| & 
\left|\begin{array}{cc}
\action_{\lambda \speed} & \action_{\mu \speed} \\
\action_{\lambda\mu} & \action_{\mu\mu}
\end{array}\right|
\\ [15pt]
\left|\begin{array}{cc}
\action_{\speed \lambda} & \action_{\mu\lambda} \\
\action_{\speed \mu} & \action_{\mu\mu}
\end{array}\right| & 
\left|\begin{array}{cc}
\action_{\speed  \speed} & \action_{\mu \speed} \\
\action_{\speed \mu} & \action_{\mu\mu}
\end{array}\right|
\end{array}\right)
$$
Recalling that $\Action_{\lambda\lambda}=\action_{\mu\mu}$, we see that the first upper left minors of $\bC_L$ and $\bc$ are both equal to
$$\Delta_1:= - \frac{1}{\action_{\mu\mu}} \left|\begin{array}{cc}
\action_{\lambda\lambda} & \action_{\mu\lambda} \\
\action_{\lambda\mu} & \action_{\mu\mu}
\end{array}\right|\,=:\delta_1\,.$$
The second upper left minor of $\bC_L$ is 
$$\Delta_2:= - \action_\mu \left|\begin{array}{cc}
\action_{\lambda\lambda} & \action_{\mu\lambda} \\
\action_{\lambda\mu} & \action_{\mu\mu}
\end{array}\right|\,=\,\action_\mu\,\action_{\mu\mu}\,\delta_1 \,,$$
while the second upper left minor of $\bc$ is 
$$\delta_2:= \det \bc= \frac{1}{(\action_{\mu\mu})^2} \left|\begin{array}{cc} 
\left|\begin{array}{cc}
\action_{\lambda\lambda} & \action_{\mu\lambda} \\
\action_{\lambda\mu} & \action_{\mu\mu}
\end{array}\right| & 
\left|\begin{array}{cc}
\action_{\lambda \speed} & \action_{\mu \speed} \\
\action_{\lambda\mu} & \action_{\mu\mu}
\end{array}\right|
\\ [15pt]
\left|\begin{array}{cc}
\action_{\speed \lambda} & \action_{\mu\lambda} \\
\action_{\speed \mu} & \action_{\mu\mu}
\end{array}\right| & 
\left|\begin{array}{cc}
\action_{\speed  \speed} & \action_{\mu \speed} \\
\action_{\speed \mu} & \action_{\mu\mu}
\end{array}\right|
\end{array}\right|
\,=\,\frac{1}{\action_{\mu\mu}}\,\det(\Hess\action)\,.$$
As to the third and last first upper left minor of $\bC_L$, it reads
$$\Delta_3:=\det \bC_L = 
\frac{4j^2\action_\mu}{(\action_{\mu\mu})^2}\,\left|\begin{array}{cc} 
\left|\begin{array}{cc}
\action_{\lambda\lambda} & \action_{\mu\lambda} \\
\action_{\lambda\mu} & \action_{\mu\mu}
\end{array}\right| & 
\left|\begin{array}{cc}
\action_{\lambda \speed} & \action_{\mu \speed} \\
\action_{\lambda\mu} & \action_{\mu\mu}
\end{array}\right| 
\\ [15pt]
\left|\begin{array}{cc}
\action_{\speed \lambda} & \action_{\mu\lambda} \\
\action_{\speed \mu} & \action_{\mu\mu}
\end{array}\right| & 
\left|\begin{array}{cc}
\action_{\speed  \speed} & \action_{\mu \speed} \\
\action_{\speed \mu} & \action_{\mu\mu}
\end{array}\right|
\end{array}\right| 
\,-\,\frac{1}{\action_{\mu\mu}}\,(2\action_\speed\action_\mu -\action_\lambda^2)\,\left|\begin{array}{cc}
\action_{\lambda\lambda} & \action_{\mu\lambda} \\
\action_{\lambda\mu} & \action_{\mu\mu}
\end{array}\right|$$
that is,
$$\Delta_3= 4j^2\action_\mu \delta_2 + (2\action_\speed\action_\mu -\action_\lambda^2)\,\delta_1\,.$$
Since $\action_\mu=\Upsilon>0$ and $2\action_\speed\action_\mu -\action_\lambda^2>0$ (by Cauchy--Schwarz), we can compare $\negsign (\bC_L)$ and $\negsign (\bc)$ by drawing the following table, which relies on Sylvester's invariance theorem\footnote{This theorem implies that the negative signature of an $n\times n$ symmetric matrix is the number of sign changes in the sequence 
$(1,\Delta_1,\ldots,\Delta_n)$, where $\Delta_k$ denotes the matrix' $k$-th upper left minor.}, and where $+/-$ signs stand for positive/negative values:
\begin{table}[H]
$$
\begin{array}{|c|c|c|%c|c|
c|c|c|c|c|c|c|c|c|}\hline
\action_{\mu\mu} & \action_{\mu\mu}\action_{\lambda\lambda}-\action_{\lambda\mu}^2& \det(\Hess\action) & \negsign(\Hess\action) %& \delta_0 
& \delta_1 & \delta_2 & \negsign(\bc) %& \Delta_0 
& \Delta_{1} & \Delta_{2} & \Delta_{3} & \negsign(\bC_L) & \det(\Hess\Action) \\ \hline\hline
+ & + & + & 0 %& {\bf +} 
& -  & + & 2 %& {\bf +} 
& - & - & ? & \geq 1 & ?\\ \hline
+ & + & - & 1 %& {\bf +} 
& -  & - & 1 %& {\bf +} 
& - & - & - & 1 & +\\ \hline
+ & - & - & 1 %& {\bf +} 
& +  & - & 1 %& {\bf +} 
& + & + & ? & \geq 0 & ? \\ \hline
+ & - & + & 2 %& {\bf +} 
& +  & + & 0 %& {\bf +}
 & + & + & + & 0 & - \\ \hline \hline
- & + & + & 2 %& {\bf +} 
& +  & - & 1 %& {\bf +} 
& + & - & ? & \geq 1 & ?\\ \hline
- & + & - & 3 %& {\bf +} 
& +  & + & 0 %& {\bf +} 
& + & - & + & 2 & +\\ \hline
- & - & - & 1 %& {\bf +} 
& -  & + & 2 %& {\bf +} 
& - & + & ? & \geq 2 & ? \\ \hline
- & - & + & 2 %& {\bf +} 
& -  & - & 1 %& {\bf +} 
& - & + & - & 3 & - \\ \hline
\end{array}$$
\caption{Possible values of signs of minors and negative signatures regarding $\bC_L$}\label{tb:signs}
\end{table}

We have used here above the additional observation (already made in \cite[Remark 2]{BNR-GDR-AEDP}) that
$$\left|\begin{array}{cc} 
\left|\begin{array}{cc}
\action_{\lambda\lambda} & \action_{\mu\lambda} \\
\action_{\lambda\mu} & \action_{\mu\mu}
\end{array}\right| & 
\left|\begin{array}{cc}
\action_{\lambda \speed} & \action_{\mu \speed} \\
\action_{\lambda\mu} & \action_{\mu\mu}
\end{array}\right| 
\\ [15pt]
\left|\begin{array}{cc}
\action_{\speed \lambda} & \action_{\mu\lambda} \\
\action_{\speed \mu} & \action_{\mu\mu}
\end{array}\right| & 
\left|\begin{array}{cc}
\action_{\speed  \speed} & \action_{\mu \speed} \\
\action_{\speed \mu} & \action_{\mu\mu}
\end{array}\right|
\end{array}\right| \,=\,\action_{\mu\mu}\,\det(\Hess\action)\,.$$

Similarly as for $\bC_L$,  
we can express 
$$ \bC_E=-\frac{1}{\Action_{\mu\mu}}  
\left(\begin{array}{ccc} 
\left|\begin{array}{cc}
\Action_{\lambda\lambda} & \Action_{\mu\lambda} \\
\Action_{\lambda\mu} & \Action_{\mu\mu}
\end{array}\right| & 
\left|\begin{array}{cc}
\Action_{\lambda j} & \Action_{\mu j} \\
\Action_{\lambda\mu} & \Action_{\mu\mu}
\end{array}\right|
& \left|\begin{array}{cc}
\Action_{\lambda \sigma} & \Action_{\mu \sigma} \\
\Action_{\lambda\mu} & \Action_{\mu\mu}
\end{array}\right| \\ [15pt]
\left|\begin{array}{cc}
\Action_{j \lambda} & \Action_{\mu\lambda} \\
\Action_{j \mu} & \Action_{\mu\mu}
\end{array}\right| & 
\left|\begin{array}{cc}
\Action_{j  j} & \Action_{\mu j} \\
\Action_{j \mu} & \Action_{\mu\mu}
\end{array}\right|
& \left|\begin{array}{cc}
\Action_{j  \sigma} & \Action_{\mu \sigma} \\
\Action_{j \mu} & \Action_{\mu\mu}
\end{array}\right| \\ [15pt]
\left|\begin{array}{cc}
\Action_{\sigma \lambda} & \Action_{\mu\lambda} \\
\Action_{\sigma \mu} & \Action_{\mu\mu}
\end{array}\right| & 
\left|\begin{array}{cc}
\Action_{\sigma  j} & \Action_{\mu j} \\
\Action_{\sigma \mu} & \Action_{\mu\mu}
\end{array}\right|
& \left|\begin{array}{cc}
\Action_{\sigma  \sigma} & \Action_{\mu \sigma} \\
\Action_{\sigma \mu} & \Action_{\mu\mu}
\end{array}\right|
\end{array}\right)
\,$$
in terms the derivatives of $\action$ as
$$
\bC_E= - \frac{1}{\action_{\lambda\lambda}}
\left(\begin{array}{ccc} 
\left|\begin{array}{cc}
\action_{\mu\mu} & -\action_{\mu\lambda} \\
-\action_{\lambda\mu} & \action_{\lambda\lambda}
\end{array}\right| & 
\left|\begin{array}{cc}
-2j\action_{\mu\speed} & 2j\action_{\lambda\speed} \\
-\action_{\lambda\mu} & \action_{\lambda\lambda}
\end{array}\right|
& 
\left|\begin{array}{cc}
0 & 0 \\
-\action_{\lambda\mu} & \action_{\lambda\lambda}
\end{array}\right|
 \\ [15pt]
\left|\begin{array}{cc}
-2j\action_{\mu\speed} & -\action_{\lambda\mu} \\
 2j\action_{\lambda\speed}& \action_{\lambda\lambda}
\end{array}\right| & 
\left|\begin{array}{cc}
4j^2\action_{\speed\speed} -2\action_\speed & 2j\action_{\lambda\speed} \\
2j\action_{\speed\lambda} & \action_{\lambda\lambda}
\end{array}\right|
& 
\left|\begin{array}{cc}
\action_{\lambda} & 0 \\
2j\action_{\lambda\speed} & \action_{\lambda\lambda}
\end{array}\right| \\ [15pt]
\left|\begin{array}{cc}
0 & -\action_{\lambda\mu} \\
0 & \action_{\lambda\lambda}
\end{array}\right| & 
\left|\begin{array}{cc}
\action_{\lambda} & 2j\action_{\lambda\speed} \\
0 & \action_{\lambda\lambda}
\end{array}\right|
& 
\left|\begin{array}{cc}
-\action_{\mu} & 0 \\
0 & \action_{\lambda\lambda}
\end{array}\right|
\end{array}\right)
$$
We thus find that the upper left minors of $\bC_E$ are
$$\Delta_{E,1}:= - \frac{1}{\action_{\lambda\lambda}} \left|\begin{array}{cc}
\action_{\lambda\lambda} & \action_{\mu\lambda} \\
\action_{\lambda\mu} & \action_{\mu\mu}
\end{array}\right|\,=\, \frac{\action_{\mu\mu}}{\action_{\lambda\lambda}}\;\delta_1\,,$$
$$\Delta_{E,2}\begin{array}[t]{l}:=\displaystyle \frac{4j^2}{(\action_{\lambda\lambda})^2}\,\left|\begin{array}{cc} 
\left|\begin{array}{cc}
\action_{\mu\mu} & \action_{\lambda\mu} \\
\action_{\mu\lambda} & \action_{\lambda\lambda}
\end{array}\right| & 
\left|\begin{array}{cc}
\action_{\mu \speed} & \action_{\lambda \speed} \\
\action_{\mu\lambda} & \action_{\lambda\lambda}
\end{array}\right| 
\\ [15pt]
\left|\begin{array}{cc}
\action_{\speed \mu} & \action_{\lambda\mu} \\
\action_{\speed \lambda} & \action_{\lambda\lambda}
\end{array}\right| & 
\left|\begin{array}{cc}
\action_{\speed  \speed} & \action_{\lambda \speed} \\
\action_{\speed \lambda} & \action_{\lambda\lambda}
\end{array}\right|
\end{array}\right| \,-\,\frac{2\action_\speed}{\action_{\lambda\lambda}} \,\left|\begin{array}{cc}
\action_{\mu\mu} & \action_{\lambda\mu} \\
\action_{\mu\lambda} & \action_{\lambda\lambda}
\end{array}\right|\\ [10pt]
=\,\displaystyle\frac{1}{\action_{\lambda\lambda}}\,(\,
4j^2\,\det\Hess\action\,+\,2\action_\speed\action_{\mu\mu}\,\delta_1\,)\,,\end{array}
$$
$$\Delta_{E,3}=\det \bC_E\,\begin{array}[t]{l}=\,\displaystyle\frac{4j^2\action_\mu}{(\action_{\lambda\lambda})^2}\,\left|\begin{array}{cc} 
\left|\begin{array}{cc}
\action_{\mu\mu} & \action_{\lambda\mu} \\
\action_{\mu\lambda} & \action_{\lambda\lambda}
\end{array}\right| & 
\left|\begin{array}{cc}
\action_{\mu \speed} & \action_{\lambda \speed} \\
\action_{\mu\lambda} & \action_{\lambda\lambda}
\end{array}\right| 
\\ [15pt]
\left|\begin{array}{cc}
\action_{\speed \mu} & \action_{\lambda\mu} \\
\action_{\speed \lambda} & \action_{\lambda\lambda}
\end{array}\right| & 
\left|\begin{array}{cc}
\action_{\speed  \speed} & \action_{\lambda \speed} \\
\action_{\speed \lambda} & \action_{\lambda\lambda}
\end{array}\right|
\end{array}\right| 
\,-\,\frac{1}{\action_{\lambda\lambda}}\,(2\action_\speed\action_\mu -\action_\lambda^2)\,\left|\begin{array}{cc}
\action_{\mu\mu} & \action_{\lambda\mu} \\
\action_{\mu\lambda} & \action_{\lambda\lambda}
\end{array}\right|\\ [15pt]
=\displaystyle\frac{1}{\action_{\lambda\lambda}}\,
({4j^2\action_\mu}\,\det\Hess\action\,+\,{(2\action_\speed\action_\mu -\action_\lambda^2)\action_{\mu\mu}}\,\delta_1\,)
\,.\end{array}$$

\begin{comment}
A striking formula that comes out of the previous computations is the following
$$\action_{\lambda\lambda}\, \det \bC_E\,=\,\action_{\mu\mu}\,\det \bC_L\,=\,-\,\det (\Hess\Action)$$
which equivalently reads
$$\Action_{\mu\mu}\, \det \bC_E\,=\,\Action_{\lambda\lambda}\,\det \bC_L\,=\,-\,\det (\Hess\Action)\,,$$
that is Eq.~\eqref{eq:constEKE-EKL}.
\end{comment}

Furthermore, since $\action_\speed= \int_{0}^{\Upsilon} \frac12 \ubv^2 >0$, we can complement Table \ref{tb:signs} with a similar table in terms of the negative signature of $\bC_E$.

\begin{table}[H]
$$\begin{array}{|c|c|c|%c|c|
c|c|c|c|c|c|}\hline
\action_{\lambda\lambda} & \action_{\lambda\lambda}\action_{\mu\mu}-\action_{\lambda\mu}^2
& \det(\Hess\action) & \negsign(\Hess\action) %& \Delta_0 
& \Delta_{E,1} & \Delta_{E,2} & \Delta_{E,3} & \negsign(\bC_E) & \det(\Hess\Action) \\ \hline\hline
+ & + & + & 0 %& {\bf +} 
& - & ? & ? & \geq 1 & ?\\ \hline
+ & + & - & 1 %& {\bf +} 
& - & - & - & 1 & +\\ \hline
+ & - & - & 1 %& {\bf +} 
& + & ? & ? & \geq 0 & ? \\ \hline
+ & - & + & 2 %& {\bf +}
 & + & + & + & 0 & - \\ \hline \hline
- & + & + & 2 %& {\bf +} 
& + & ? & ? & \geq 0 & ?\\ \hline
- & + & - & 3 %& {\bf +} 
& + & + & + & 0 & +\\ \hline
- & - & - & 1 %& {\bf +} 
& - & ? & ? & \geq 2 & ? \\ \hline
- & - & + & 2 %& {\bf +} 
& - & - & - & 1 & - \\ \hline
\end{array}$$
\caption{Possible values of signs of minors and negative signatures regarding $\bC_E$}\label{tb:signsE}
\end{table}

\subsection{Proof of Lemma \ref{lem:Sturm}}\label{app:Sturm}

For convenience of the reader, let us recall the statement of this lemma.
\begin{lemma*}
\label{lem:Sturmapp} Assume that $\cap:I\to (0,+\infty)$ and
$\Potential:I\to \R$ are ${\class}^2$ on some open interval $I$ and such that the \emph{Euler--Lagrange} equation
$\Euler \enred  [\bv] =0$ associated with the energy
$$\enred: (\bv,\bv_x)\mapsto \tfrac{1}{2} \cap(\bv) \,\bv_x^2\,+\,\Potential (\bv)$$
admits a family of periodic solutions $\ubv$ taking values in $I$, parametrized by the energy level
$\mu \,=\,\Legendre \enred[\ubv]$, for $\mu\in J$, another open interval. If we denote by $\Upsilon$ the period of  $\ubv$, and assume that $\Upsilon_\mu$, its derivative with respect to $\mu$, does not vanish, then the self-adjoint differential operator $\linvar\,:=\,\Hess \enred[\ubv]$ has the following properties:
\begin{itemize}
\item the kernel of $\linvar$ on $L^2(\R/\Upsilon\Z)$ is the line spanned by $\ubv_x$;
\item the negative  signature $\negsign(\linvar)$ of $\linvar$ is given by the following rule:
\begin{itemize}
\item[$*$] if $\Upsilon_\mu>0$ then $\negsign(\linvar)=1$,
\item[$*$] if $\Upsilon_\mu<0$ then $\negsign(\linvar)=2$.
\end{itemize}
\end{itemize}
\end{lemma*}

This lemma belongs to the classical periodic Floquet/Sturm-Liouville theory that may be found in \cite[Part~I-Chapter~II]{Magnus-Winkler}, \cite[Chapter~XIII-Section~16]{Reed-Simon-IV}, \cite[Part~1-Section~5.6]{Teschl} or partially in \cite{Neves}. Some complements in \cite{Neves} are also used there and in \cite{Natali-Neves} to address the question of co-periodic stability.

\begin{proof}
The first point follows from an elementary, ODE argument. Indeed,  the equation $\linvar \bv=0$ is a second order ODE, of which we already know two independent solutions, namely $\ubv_x$ and $\ubv_\mu$, obtained by differentiating the Euler--Lagrange equation $\Euler \enred[\ubv]=0$. In addition, by differentiating the null function $\ubv(\cdot+\Upsilon)-\ubv$ with respect to $\mu$, we see that $\ubv_\mu(\cdot+\Upsilon)-\ubv_\mu=- \Upsilon_{\!\mu} \ubv_{x}$.
Since $\ubv_x$ is not identically zero, and $\Upsilon_{\!\mu}\neq 0$  by assumption, $\ubv_\mu$ cannot be $\Upsilon$-periodic, nor can be any linear combination of $\ubv_\mu$ and $\ubv_{x}$, since the latter is $\Upsilon$-periodic. Therefore, the set of $\Upsilon$-periodic solutions of $\linvar \bv=0$ is the line spanned by $\ubv_{x}$.

The second point involves some classical spectral theory of self-adjoint operators in general, and of \emph{Sturm--Liouville} operators in particular, as well as a technical argument based on a computation that is hardly ever done in details. The general theory says --- see \cite[Theorem~5.37]{Teschl} --- that, since $\ubv_x$ is in the kernel of $\linvar$ and vanishes only twice on one period, $0$ is either the second or the third eigenvalue of the Sturm--Liouville operator 
$$\linvar \,=\,- \partial_x K \partial_x + q\,,\qquad K:= \cap(\ubv)\,,\qquad q:=\, \Potential''(\ubv)   + \tfrac{1}{2} \cap''(\ubv) \ubv_x^2   - \partial_x(\cap'(\ubv) \ubv_{x})\,.$$ So we are left with showing that $\Upsilon_\mu>0$ corresponds to the first case, and $\Upsilon_\mu<0$ to the second. This will follow from a computation based on the \emph{discriminant} associated with the eigenvalue equation
$$\linvar \bv = \evr \bv\,.$$
This discriminant is the function $\disc=\disc(\evr)$ defined as
$$\disc(\evr)=\mbox{tr } (\monod(\Upsilon;\evr))\,,$$
where $\monod(\cdot;\evr)$ denotes the fundamental solution of the eigenvalue equation $\linvar \bv = \evr \bv$ viewed as the first-order system
\begin{equation}\label{eq:Sturm-syst}
\left\{\begin{array}{l}
\dot\bv= w/K\,,\\
\dot{w} = (q\,-\,\evr)\, \bv\,.
\end{array}\right.
\end{equation}
(We warn the reader that $\monod$ and $w$ are local notations, which do not have the same meaning as in the main part of the paper.) The reason why $\disc$ plays a prominent role here is that the matrix value function $\monod$ has a determinant constantly equal to one --- this follows from Liouville's formula, since the matrix in System \eqref{eq:Sturm-syst} is traceless ---  so that $\monod$  
at $x=\Upsilon$ admits $1$ as an eigenvalue if and only if $\disc(\evr)=2$. 
In addition, 
this is exactly the condition under which \eqref{eq:Sturm-syst} has a $\Upsilon$-periodic solution.
Therefore, $\disc(\evr)=2$ is an iff condition for $\evr$ to be an eigenvalue of $\linvar$ on $L^2(\R/\Upsilon\Z)$. 
In particular, we already know that $0$ is an eigenvalue of $\linvar$, so $\disc(0)=2$.

The connection of $\disc$ with the ordering of eigenvalues arises through its derivative.
Indeed, the sign of $\disc'(0)$ tells us if $0$ is the second or the third eigenvalue of $\linvar$ on $L^2(\R/\Upsilon\Z)$. This is a consequence of a deeper knowledge on the spectrum of $\linvar$ on $L^2(\R)$. By \emph{Floquet--Bloch} decomposition --- see point (d) in \cite[Theorem~XIII-85]{Reed-Simon-IV} combined with \cite[Theorems~XIII-88 \&~XIII-89]{Reed-Simon-IV} for the semi-linear case, or \cite[p.30-31]{R} for a general argument ---,
$$\sigma_{L^2(\R)}(\linvar)\,=\,\bigcup_{\nu \in [0,2\pi)} \sigma(\linvar_\nu)\,,$$
where $\linvar_\nu$ is defined by the same formula as $\linvar$, that is $\linvar_\nu: =  \,- \partial_x K \partial_x + q$, but on the domain
$$\{\bv\in H^2(\R/\Upsilon\Z)\,;\;\bv(\Upsilon)=\ee^{i\nu}\,\bv(0)\}\,.$$
By a similar reasoning as for $\linvar$, we can see that the eigenvalues of $\linvar_\nu$ are those $\evr$ such that
$\ee^{i\nu}$ is an eigenvalue of $\monod(\Upsilon;\evr)$ 
and this is equivalent to $\disc(\evr)=2\cos\nu$. We thus see that
$\sigma_{L^2(\R)}(\linvar)$ is the union of closed intervals in which $\disc$ achieves values in $[-2,2]$ (see Figure \ref{fig:spectrum} below). 

\begin{figure}[H]
\begin{center}
\includegraphics[width=85mm]{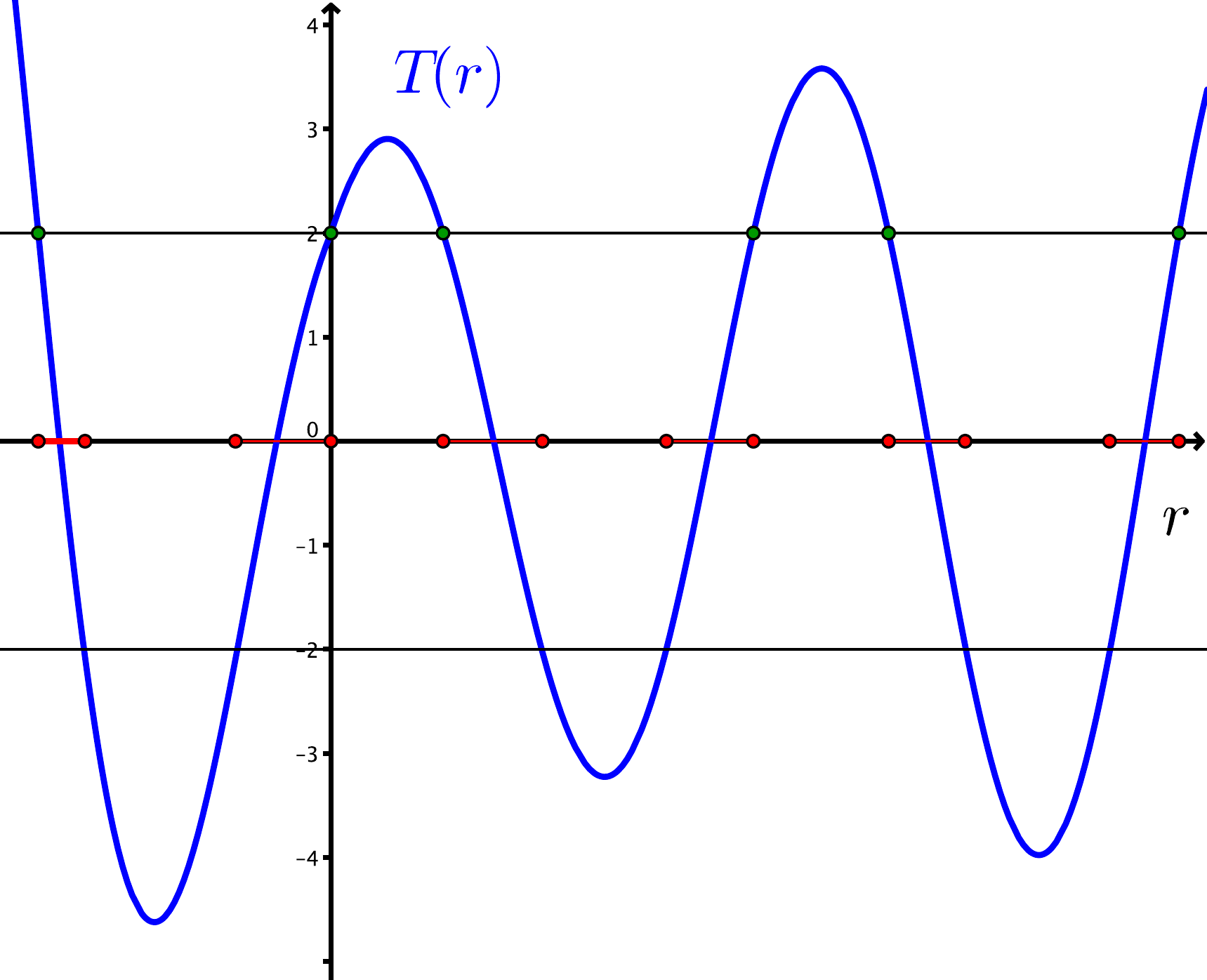} 
\end{center}
\caption{Oscillations of discriminant $\disc(\evr)$ (blue curve), spectrum of $\linvar_0$ (abscissas of green bullets), and (part of) spectrum of  $\linvar$ (red intervals).}\label{fig:spectrum}
\end{figure}

In particular, the eigenvalues of $\linvar$ on $L^2(\R/\Upsilon\Z)$ --- that is, of the operator $\linvar_0$ --- are found at the intersection of this graph with the line $\disc=2$. In addition, it is known --- see \cite[Theorem~5.33]{Teschl} --- that the
least eigenvalue of $\linvar_0$, say $\evr_1$, is in fact the lower bound for the spectrum of all the operators $\linvar_\nu$. By the mean value theorem, this implies $\disc(\evr) > 2$ for all $\evr<\evr_1$.
Paying closer attention to the variations of $\disc$, we thus infer that 
$0$ is the second eigenvalue of $\linvar_0$ if $T'(0)>0$, and the third one if $T'(0)<0$.

It remains to make the connection between the signs of $T'(0)$ and $\Upsilon_\mu$. This will follow from a fairly general computation, involving any two independent solutions $
(v_1,w_1)^{\sf T}$ and $
(v_2,w_2)$ of \eqref{eq:Sturm-syst}, smoothly parametrized by $r$ and such that
$$\left|\begin{array}{cc} v_1 & v_2 \\
w_1 & w_2\end{array}\right|\,\equiv 1\,.$$
(Again, this is possible since the matrix in System \eqref{eq:Sturm-syst} is traceless.)
The solution operator ${\bf s}(x,s;r)$, related to the fundamental solution by
$\monod(x;r)={\bf s}(x,0;r)$,
 is thus given by
$${\bf s}(x,s;r) = \left(\begin{array}{cc} \bbv(x)\wedge \bw(s) & - \bbv(x)\wedge \bbv(s) \\
\bw(x)\wedge \bw(s) & - \bw(x)\wedge \bbv(s)\end{array}\right)\,,$$
where we have introduced, for convenience, the vector-valued functions 
$$ \bbv:= \left(\begin{array}{c}v_1\\v_2\end{array}\right)\,,\;\bw:= \left(\begin{array}{c}w_1\\w_2\end{array}\right)\,,$$
and used the notation $\wedge$ as a shortcut for the determinant in $\R^2$.  Note in particular that the requirement made on the determinant of $(v_1,w_1)^{\sf T}$ and  $(v_2,w_2)^{\sf T}$ is equivalent to
$$\bbv\wedge \bw \equiv 1\,.$$
Furthermore, we have by definition
$$\disc(r)=\mbox{tr } ({\bf s}(\Upsilon,0;r)) = \bbv(\Upsilon;r)\wedge \bw(0;r) + \bbv(0;r) \wedge \bw(\Upsilon;r)\,.$$
Without loss of generality, we may assume that $\bbv(0;r)=\bbv(0;0)$ and  $\bw(0;r)=\bw(0;0)$ for all $r$, so that
$$\disc'(0)=\bbv_r(\Upsilon;r)\wedge \bw(0;0) + \bbv(0;0) \wedge \bw_r(\Upsilon;r)\,,$$
where the subscript $r$ stands for a partial derivative with respect to $r$. We claim that 
\begin{equation}
\label{eq:signdiscp}
\mbox{\rm sign}({\disc}'(0))=\mbox{\rm sign}(\bw(\Upsilon;0)\wedge \bw(0;0))\,.
\end{equation}
The proof of \eqref{eq:signdiscp} proposed here may be thought of as a slight variation on the one relying on Formula~(2.17) in \cite{Magnus-Winkler}.

In order to prove \eqref{eq:signdiscp}, let us first observe that,
for any solution $(v,w)^{\sf T}$ to \eqref{eq:Sturm-syst},  by differentiation  with respect to $r$,
$(v_r,w_r)^{\sf T}$ solves the system with source term
$$\left\{\begin{array}{l}
\dot\bv_r= w_r/K\,,\\
\dot{w}_r = (q\,-\,\evr)\, \bv_r \,-\,\bv\,,
\end{array}\right.$$
and is therefore given by the Duhamel formula associated with the solution operator ${\bf s}(x,s;r)$. This gives
$$\left\{\begin{array}{l}
\bv_r(x)=\int_{0}^{x} (\bbv(x)\wedge\bbv(s)) \bv(s)\,\dif s\\ [5pt]
{w}_r(x) = \int_{0}^{x} (\bw(x)\wedge\bbv(s)) \bv(s)\,\dif s
\end{array}\right.$$
where we have omitted to write the dependence on $r$ for simplicity, hence
$$\disc'(0)=\int_{0}^{\Upsilon}  (\bbv(\Upsilon)\wedge\bbv(s))  (\bbv(s) \wedge \bw(0)) + (\bbv(0) \wedge \bbv(s))(\bw(\Upsilon)\wedge\bbv(s))\,\dif s\,,$$
where all functions are evaluated at $r=0$.
The idea is now to find some cancellations in the integrand here above, by using the known constraint
$\bbv\wedge \bw \equiv 1$, together with the fact that $\disc(0)=2$, that is,
\begin{equation}\label{eq:T02}
\bbv(\Upsilon)\wedge \bw(0) + \bbv(0) \wedge \bw(\Upsilon)\,=\,2\,,
\end{equation}
where again all functions are evaluated at $r=0$.
By elementary algebra, these relations imply that
$$(\bbv(\Upsilon) -\bbv(0))\wedge (\bw(\Upsilon)-\bw(0))\,=\,0\,.$$
At least one of these two vectors $\bbv(\Upsilon) -\bbv(0)$ and $\bw(\Upsilon)-\bw(0)$ is nonzero,
because the solutions $(v_1,w_1)^{\sf T}$ and  $(v_2,w_2)^{\sf T}$ of \eqref{eq:Sturm-syst} cannot be both $\Upsilon$-periodic (recall that the kernel of $\linvar$ is one-dimensional).

Let us assume that $\bbu:= \bbv(\Upsilon) -\bbv(0)$ is nonzero. Then it is necessarily independent from 
$\bbv(0)$  --- otherwise, there should exist a real number $a$ such that $\bbv(\Upsilon) = a \bbv(0)$, and \eqref{eq:T02} would imply $a+1/a =2$, hence $a=1$ and $\bbu=0$. Then we may decompose
$$\bw(\Upsilon)= \bw(0)+\omega \bbu\,,\qquad \bbv(s)=\alpha(s) \bbv(0)+\beta(s)\bbu\,,\;\forall s\,.$$
We substitute these expressions in $T'(0)$, and by using the relation
$$ \bbu \wedge \bw(0)+ \omega\, \bbv(0)\wedge \bbu\,=\,0\,,$$
which comes from  $\bbv(\Upsilon)\wedge \bw(\Upsilon)= 1$ and  $\bbv(0)\wedge \bw(0)= 1$,
we find that all but one term cancel out, in such a way that
$$\disc'(0)=\int_{0}^{\Upsilon} \alpha(s)^2 \,\bbu\wedge \bbv(0)\,\dif s\,=\,\bbv(\Upsilon) \wedge \bbv(0)\; \int_{0}^{\Upsilon} \alpha(s)^2 \,\dif s\,.$$
This also gives
$$\omega^2 \disc'(0)=\bw(\Upsilon) \wedge \bw(0)\; \int_{0}^{\Upsilon} \alpha(s)^2 \,\dif s\,.$$
Observe in addition that $\alpha$ is not identically zero if we set $\alpha(0)=1$ (and $\beta(0)=0$), so that 
$\int_{0}^{\Upsilon} \alpha(s)^2 \,\dif s>0$.

The case $\bbv(\Upsilon) =\bbv(0)$ is even easier to deal with, because then
$$\disc'(0)=\int_{0}^{\Upsilon}  (\bbv(0)\wedge\bbv(s))  (\bw(\Upsilon)-\bw(0)) \wedge\bbv(s)\,\dif s\,.$$
If we decompose
$\bbv(s)=\alpha(s) \bbv(0)+\beta(s)\bw(0)$ with  $\alpha(0)=1$ and $\beta(0)=0$, which is possible because
$\bbv(0)\wedge \bw(0) =1$ so that $ \bbv(0)$ and $\bw(0)$ are not colinear, we can also write, using that 
$\bbv(0)\wedge (\bw(\Upsilon)-\bw(0))=0$ --- which comes from the difference between $\bbv(0)\wedge \bw(\Upsilon) = \bbv(\Upsilon)\wedge \bw(\Upsilon)=1$ and $\bbv(0)\wedge \bw(0)=1$,
$$\disc'(0)=\bw(\Upsilon) \wedge\bw(0)\,\int_{0}^{\Upsilon}  \beta(s)^2  \dif s\,.$$
We observe that $\beta(s)=\bbv(0)\wedge \bbv(s)$ is not identically zero, otherwise we would have 
$$0\,=\,\bbv(0)\wedge \bbv'(s)=\bbv(0)\wedge \bw(s) /K(s)$$
by the first equation in \eqref{eq:Sturm-syst} applied to both components of $\bbv$ and this would contradict $\bbv(0)\wedge \bw(0)=1$.

We thus have \eqref{eq:signdiscp} as soon as $\bw(\Upsilon)\wedge \bw(0)\neq 0$, and if  $\bw(\Upsilon)\wedge \bw(0)= 0$ we have $\mbox{\rm sign}({\disc}'(0))=\mbox{\rm sign}(\bbv(\Upsilon)\wedge \bbv(0))$.
We are in fact going to apply \eqref{eq:signdiscp} in the first case.
\begin{comment}
Before that, let ust note that the identity in \eqref{eq:signdiscp} can easily be extended to independent solutions $(v_1,w_1)^{\sf T}$ and $(v_2,w_2)^{\sf T}$ such that $\bbv\wedge \bw$ is not necessarily equal to one --- but still nonzero of course, by definition of independent solutions. If $\bbv\wedge \bw>0$, it suffices to divide both solutions by $\sqrt{\bbv\wedge \bw}$, and we still find that
$\mbox{\rm sign}({\disc}'(0))=\mbox{\rm sign}(\bw(\Upsilon)\wedge \bw(0))$.
If $\bbv\wedge \bw<0$, we can change one the solutions for its opposite and we are left with the previous case. Since this change also changes the sign of $\bw(\Upsilon)\wedge \bw(0)$, we finally have, as a more general version of \eqref{eq:signdiscp},
\begin{equation}
\label{eq:signdiscpbis}
\mbox{\rm sign}({\disc}'(0))=\mbox{\rm sign}(\bw(\Upsilon;0)\wedge \bw(0;0))\,\mbox{\rm sign}(\bbv(0;0)\wedge \bw(0;0))\,.
\end{equation}
(Recall that $\bbv\wedge \bw$ is constant.)
\end{comment}

Let us now apply  \eqref{eq:signdiscp}  to
$$\bv_1(\cdot;0)=\ubv_x\,,\quad w_1(\cdot;0)=K\,\ubv_{xx}\,,\quad\bv_2(\cdot;0)=\ubv_\mu\,,\quad w_2(\cdot;0)=K\,\ubv_{\mu,x}\,.$$
By differentiating the equation $\Legendre \enred[\ubv]=\mu$ with respect to $\mu$, we see that
$$
- \,K \,\ubv_{xx} \ubv_\mu \,+\,K\,\ubv_x\,\ubv_{\mu,x}\,=\,1\,,$$
which exactly means that $\bbv\wedge \bw\equiv 1$, and 
$$\bw(\Upsilon;0)\wedge \bw(0;0)= K(0)^2\,\ubv_{xx}(0)\,(\ubv_{\mu,x}(0)-\ubv_{\mu,x}(\Upsilon))\,=\,
K(0)^2\,\ubv_{xx}(0)^2\,\Upsilon_{\mu}$$
by using $\ubv_{\mu,x}(\Upsilon)-\ubv_{\mu,x}(0)=-\Upsilon_{\mu}\,\ubv_{xx}(0)$, which comes from the differentiation with respect to $\mu$ of $\ubv_{x}(\Upsilon)-\ubv_{x}(0)=0$. Since $K$ does not vanish and $\ubv_{xx}(0)\neq 0$ --- provided that $\ubv$ is chosen such that $\ubv_x(0)=0$, this eventually shows that
$$\mbox{\rm sign}({\disc}'(0))=\mbox{\rm sign}(\Upsilon_{\mu})\,.$$
\end{proof}

\paragraph{Acknowledgement.} This work has been partly supported by the European Research Council %\href{http://math.univ-lyon1.fr/~filbet/nusikimo/nusikimo.htm}
{ERC Starting Grant 2009, project 239983- NuSiKiMo}, and by ANR project BoND (ANR-13-BS01-0009-01). The doctoral scholarship of the second author is directly supported by ANR-13-BS01-0009-01. The first author acknowledges the kind hospitality of the Gran Sasso Science Institute in L'Aquila, where this paper was completed. The third author thanks Mat Johnson for helpful guidance through the vast literature on orbital stability of periodic waves of (gKdV).

\end{document}